\documentclass[reqno]{amsart}
\usepackage{amsfonts,amsmath,amssymb,mathtools,enumitem}
\usepackage{datetime}
\usepackage{hyperref}
\hypersetup{colorlinks=true, pdfstartview=FitV, linkcolor=BrickRed,citecolor=BrickRed, urlcolor=BrickRed, linktoc=page}

\usepackage{graphicx}
\usepackage[dvipsnames]{xcolor}

\definecolor{labelkey}{rgb}{0.6,0,0}

\usepackage{mathtools}
\mathtoolsset{showonlyrefs=true}

\usepackage[margin=2.5cm]{geometry}
\numberwithin{equation}{section}

\def\eps{\varepsilon}

\newcommand{\N}{\mathbb{N}}
\newcommand{\R}{\mathbb{R}}
\newcommand{\T}{\mathbb{T}}

\newcommand{\Sz}{\abs{\mathcal{S}}}
\newcommand{\W}{\widetilde{\mathcal{W}}}

\newcommand{\A}{\mathcal{A}}
\newcommand{\B}{\mathcal{B}}

\newcommand{\Lam}{\mathbf{\Lambda}}
\newcommand{\U}{\mathcal{U}}
\newcommand{\Q}{\mathcal{Q}}
\newcommand{\M}{\mathcal{M}}
\newcommand{\m}{\mathfrak{m}}
\newcommand{\V}{\mathcal{V}}
\newcommand{\Ups}{\Upsilon}
\newcommand{\h}{\textnormal{h}}

\newcommand{\Rp}{\mathcal{R}}

\renewcommand{\div}{{\mathrm{div }}}
\newcommand{\curl}{{\mathrm{curl }}}

\providecommand{\ip}[1]{\langle#1\rangle}
\providecommand{\abs}[1]{\left\lvert#1\right\rvert}
\providecommand{\norm}[1]{\left\|#1\right\|}

\newtheorem{theorem}{Theorem}[section]
\newtheorem{lemma}[theorem]{Lemma}
\newtheorem{corollary}[theorem]{Corollary}
\newtheorem{proposition}[theorem]{Proposition}

\newtheorem{remark}[theorem]{Remark}
\newtheorem*{notation}{Notation}

\begin{document}

\title{On the stabilizing effect of rotation in the 3d Euler equations}

\author{Yan Guo}
\address{Brown University, Providence, RI, USA}
\email{yan\_guo@brown.edu}

\author{Chunyan Huang}
\address{School of Statistics and Mathematics, Central University of Finance and Economics, China}
\email{hcy@cufe.edu.cn}

\author{Benoit Pausader}
\address{Brown University, Providence, RI, USA}
\email{benoit\_pausader@brown.edu}

\author{Klaus Widmayer}
\address{\'Ecole Polytechnique F\'ed\'erale de Lausanne, Switzerland}
\email{klaus.widmayer@epfl.ch}



\begin{abstract}
While it is well known that constant rotation induces linear dispersive effects in various fluid models, we study here its effect on long time nonlinear dynamics in the inviscid setting. More precisely, we investigate stability in the 3d rotating Euler equations in $\R^3$ with a \emph{fixed} speed of rotation. We show that for any $\M > 0$, axisymmetric initial data of sufficiently small size $\eps$ lead to solutions that exist for a long time at least $\eps^{-\M}$ and disperse. This is a manifestation of the stabilizing effect of rotation, regardless of its speed.

To achieve this we develop an anisotropic framework that naturally builds on the available symmetries. This allows for a precise quantification and control of the geometry of nonlinear interactions, while at the same time giving enough information to obtain dispersive decay via adapted linear dispersive estimates.

\end{abstract}

\setcounter{tocdepth}{1}
\maketitle
\tableofcontents
\vspace*{-1cm}

\section{Introduction}
In this article we study the incompressible $3d$ Euler equations in a uniformly rotating frame of reference. By a simple choice of frame and scale, we  will henceforth assume the axis of rotation to be $\vec{e}_3=(0,0,1)^\intercal$,
so that the equations for the velocity field $u:\R\times\R^3\to\R^3$ read\footnote{We remark that these equations expressed in the vorticity $\omega:=\curl(u):\R\times\R^3\to\R^3$ read $$\partial_t \omega+u\cdot\nabla\omega=\omega\cdot\nabla u+\partial_3 u,\quad \div(\omega)=0.$$}
\begin{equation}\label{eq:IER}
\begin{cases}
&\partial_tu+u\cdot\nabla u+\vec{e}_3\times u+\nabla p=0,\\
&\div(u)=0.
\end{cases}
\end{equation}
These are the standard $3d$ Euler equations with an additional term $\vec{e}_3\times u$, which gives the Coriolis force experienced in the rotating frame. 
To close these equations, the scalar pressure $p:\R\times\R^3\to\R$ is recovered from $u$ via $\Delta p=-(\partial_1 u_2-\partial_2 u_1)-\div[u\cdot\nabla u]$.

A particular class of solutions to \eqref{eq:IER} is given by functions that are invariant under rotations about $\vec{e}_3$. We shall henceforth simply refer to this feature as ``axisymmetry''. We note here that this property is preserved by the nonlinear flow of \eqref{eq:IER}, and allows for solutions with non-vanishing swirl.

Our main result then shows that small, axisymmetric solutions to \eqref{eq:IER} are stable for a long time, \emph{independently} of the speed of rotation:
\begin{theorem}\label{MainThm}
There exists a norm $X$, finite for Schwartz data and $\varepsilon_0>0$ such that the following holds: For any $\M>0$ there exists $N_0>0$ such that if the initial data $u_0$ is \emph{axisymmetric} and satisfies
\begin{equation}
 \norm{u_0}_{X}+\norm{u_0}_{H^{N_0}}\le\varepsilon<\varepsilon_0,
\end{equation}
then there exists a solution $u\in C([0,T_\varepsilon]:H^{N_0})$ of \eqref{eq:IER} with initial data $u_0$, which exists for a long time $T_\varepsilon>\varepsilon^{-\M}$.
\end{theorem}
A few comments are in order.
\begin{enumerate}[leftmargin=*]
 \item We highlight that other than smallness and axisymmetry, no assumptions are made on the initial data. In particular, our solutions may have non-vanishing, general localized swirl. According to the classical theory, such solutions would then a priori exist only for a short time proportional to $\eps^{-1}$. In fact, recent work \cite{ElgindiBU} shows that in a lower regularity setting, in the absence of rotation axisymmetric initial data may lead to finite time singularity formation.
 
 Our Theorem \ref{MainThm} is thus a strong manifestation of the stabilizing effect of rotation, and suggests that it may prevent singularity formation for sufficiently smooth solutions of small amplitude.

 \item Rotational effects play an important role in many physical processes, in particular in the geophysical setting \cite{GFD,mcwilliamsGFD}, which has also  been studied in the mathematical literature in a large variety of contexts  -- see \cite{GS2007} for some overview, and \cite{GILT2020} (and references therein) for a recent result in the context of the primitive equations.
 
 \item As is well known, at the \emph{linearized} level constant rotation as in \eqref{eq:IER} induces dispersion. In our setting of the full space $\R^3$ this leads to amplitude decay. However, the dispersion in \eqref{eq:IER} is anisotropic and degenerate, and in general only produces decay at a critical rate of $t^{-1}$ (see also Proposition \ref{DecayProp} and Remark \ref{rem:sharpdecay} below). 
 
 This contrasts with another common model for the effect of rotation, the so-called $\beta$-plane approximation: this is a tangent plane model for the effect of the Coriolis force \cite{Gal2008}. There the dispersion is still anisotropic, but less degenerate than for \eqref{eq:IER}, so that the critical decay rate $t^{-1}$ also holds in this $2d$ model. In particular, global stability for small data then holds \cite{globalbeta}.
 
 \item Prior works on the rotating Navier-Stokes equations (i.e.\ \eqref{eq:IER} with additional viscous force $\nu\Delta u$ on the right hand side) track the \emph{speed of rotation} through Strichartz estimates for the linear rotation operator \cite{CDGG2002a,CDGG2006,GS2007}, which was then used in various singular limits, and to show the global existence of certain solutions in the viscous setting, provided that the speed of rotation is high enough.
 
 In the rotating Euler equations \eqref{eq:IER}, decay and Strichartz estimates have been used to treat the case of \emph{fast} rotation \cite{Dut2005,KohE,MR3488136,WC2018,AF2018}. Here the (inverse) rotation speed plays the role of a small parameter that can be used to suppress the truly nonlinear dynamics. This allows for an extension of the lifespan of solutions by a factor that is logarithmic (provided some Sobolev regularity of solutions -- in more refined Besov spaces one can obtain a fixed polynomial \cite{WC2018}) in the speed of rotation, without even making use of the precise structure of the nonlinearity.
 
 Letting $u_\omega(t,x):=\omega\cdot u(\omega t,x)$ for $\omega>0$, we note that if $u$ is a solution to \eqref{eq:IER} on $[0,T]$, then $u_\omega$ solves \eqref{eq:IER} with $\vec{e}_3$ replaced by $\omega\cdot\vec{e}_3$, i.e.\ $u_\omega$ solves the rotating Euler equations with speed of rotation $\omega$, on a time interval $[0,\omega^{-1}T]$. While our assumptions on the initial data are more restrictive than those of \cite{KohE,MR3488136} (and in particular require axisymmetry), this rescaling shows that for initial data of size $\eps$ in \eqref{eq:IER} the increase in the lifespan of solutions due to the above works is of order $\log(\eps)$, whereas our result in Theorem \ref{MainThm} gives an improvement to any polynomial scale $\eps^{-\M}$, $\M>0$.

 \item Our work does \emph{not} make use of small parameters, and rather treats \eqref{eq:IER} from the perspective of nonlinear dispersive equations. Building on the dispersion in the linearized equation of \eqref{eq:IER}, we treat the full, highly quasilinear problem in a perturbative fashion. Here the precise structure of the nonlinearity plays a crucial role, and in particular is responsible for our restriction to axisymmetric initial data. 
 
 As is readily seen, axisymmetry is preserved by the nonlinear flow of \eqref{eq:IER}. Moreover, it leads to a simplification in the nonlinearity, which is essential for our arguments. Indeed, the structure of the equations without this restriction is much less favorable: as already observed in \cite{BMN1997} in the setting of \eqref{eq:IER} on a (generic) torus $\T^3$, the rotating Euler equations \eqref{eq:IER} contain the two-dimensional Euler equations (without rotation) as a resonant subsystem. From this perspective, the assumption of axisymmetry trivializes this dynamic.

 \item The main challenges we face lie in the anisotropy (which leads to fewer symmetries and conserved quantities), the critical rate of dispersive decay and the quasilinear nature of \eqref{eq:IER}. To address these, we are developing here a new framework that aims to maximally exploit the available symmetries. For this, we introduce an anisotropic functional setting, which allows us to control precisely the nonlinear interactions, but also requires an adapted (linear) dispersive analysis -- for more see Section \ref{ssec:method}.
 
 The methods developed here are natural and well-adapted to the current problem, but (we believe) are also versatile and flexible enough to be useful in many other contexts.
\end{enumerate}

\subsection{Our method}\label{ssec:method}
We build on classical techniques for small data/global regularity problems in nonlinear dispersive differential equations: vector fields \cite{Klainer85} and normal forms \cite{ShatahKG85} as unified in a spacetime resonance approach \cite{GMSGWW3d,GNT1,GermSTR} and developed in \cite{CapWW,KP,IP2013,IoPau1,BG2014,DIPau,Den2018}. The basic observation to start this analysis is that the linearization of \eqref{eq:IER} is dispersive with dispersion relation
\begin{equation}
 \Lambda(\xi)=\frac{\xi_3}{\abs{\xi}},\qquad \xi\in\R^3.
\end{equation}
This is anisotropic and degenerate, and leads to $L^\infty$ decay at the critical rate $t^{-1}$ -- see also Proposition \ref{DecayProp}. 

In the context of the nonlinear problem, this anisotropy also manifests itself in a loss of symmetries and conservation laws compared to the Euler equations without rotation. In particular, we observe that we have two vector fields for \eqref{eq:IER} that commute with each other \emph{and with the equation}, namely rotations $\Omega$ about the axis $\vec{e}_3$ and scaling $S$ (see Section \ref{sec:structure}). Our goal is to rely as much as possible on these symmetries. However, since they do not give control over all directions, we complement them with a third vector field $\Ups$. This is chosen such that it commutes with $\Omega$ and $S$, but it does not commute with the equation.

The resulting framework then balances the control needed to guarantee the dispersive decay with control that can be nonlinearly propagated, and establishes the relevant bounds in a bootstrap argument (see Theorem \ref{thm:btstrap} in Section \ref{sec:overview}).
For this it is important to quantify precisely the nonlinear interactions, and due to the anisotropy we work in a functional setting that localizes the horizontal and vertical components in frequency space, i.e.\ we introduce Littlewood-Paley decompositions relative to the horizontal $\abs{\xi_\h}/\abs{\xi}$, with $\xi_\h=(\xi_1,\xi_2)$, and vertical components $\abs{\xi_3}/\abs{\xi}$ of a vector $\xi=(\xi_1,\xi_2,\xi_3)\in\R^3$ (see also Section \ref{ssec:loc}).

\subsubsection*{Linear Decay}
The vector fields $S,\Omega$ commute with the equation and it is thus direct to derive energy estimates for them. Unfortunately, this does not suffice to obtain decay in $L^\infty$. Instead, we rely on a linear dispersive analysis (see Section \ref{sec:lin_decay}) which is fine tuned to the symmetries of the equation. More concretely, we use a stationary phase argument that is adapted to the vector fields, and only relies on one copy of the vector field $\Ups$. The required control is expressed in a norm $\norm{.}_D$. This in turn is based on a $B$ norm that we introduce, and which combines the advantages of an $L^2$-based norm with the scaling of the Fourier transform in $L^\infty$.

\subsubsection*{Nonlinear Analysis}
To understand how the control required for decay can be propagated nonlinearly, we carry out a detailed analysis of the structure of our problem (see Section \ref{sec:structure}). This inspires our choice of two scalar unknowns $U_+$ and $U_-$. The assumption of axisymmetry then plays a crucial role in simplifying the nonlinear interactions. Filtering out the linear evolution we can then reformulate \eqref{eq:IER} via bilinear Duhamel formulas for two scalar profiles $\U_+,\U_-$ (see \eqref{eq:IER_disp_Duham}).

The nonlinear analysis then revolves around bilinear estimates for these expressions. We follow a strategy in three bootstrap steps. In each step we successively upgrade control of a larger amount of vector fields $S$ on our unknowns to a stronger type of control for a lower amount of vector fields $S$:
\begin{enumerate}
 \item Building on the $L^\infty$ decay, we first derive energy estimates for a large amount of vector fields $S$ on our unknowns in Section \ref{sec:energy}.
 \item Next we develop bilinear estimates (Section \ref{sec:bilin}) that can be used in Section \ref{sec:B_VFprop} to upgrade the energy control to $B$ norm estimates for fewer vector fields $S$ on our unknowns. This proceeds through integration by parts along $S$, but remarkably does not make use of normal form.
 \item Finally, to obtain the last piece required for decay, we have to establish $L^2$ bounds for few vector fields $S$ on top of $\Ups$ on our unknowns. This is the most delicate part of our arguments, and is achieved by a combination of bilinear estimates (that again come from integrations by parts along the vector fields $S$, see Section \ref{sec:more_bilin}) and normal form methods. The relevant details are presented in Section \ref{sec:Ups_propagation}.
\end{enumerate}
For a more precise overview of these steps we refer the reader to Section \ref{sec:overview}.

\subsection{Some notations}
For future use we introduce some notations. For $v\in\R^3$ we write $v_h=(v_1,v_2,0)$ for its horizontal part, and $v_\h^\perp=(-v_2,v_1,0)$ for its rotation about the $x_3$ axis by 90 degrees. Similarly, we will use the following differential operators:
\begin{equation}
\nabla_\h=(\partial_1,\partial_2,0),\quad\div_\h(u)=\partial_1u_1+\partial_2u_2,\quad\nabla_\h^\perp=(-\partial_2,\partial_1,0),\quad\curl_\h(u)=\partial_1u_2-\partial_2u_1.
\end{equation}
Besides, repeated indices are summed. Roman indices run in $\{1,2\}$ and greek indices run in $\{1,2,3\}$. We thus see that
\begin{equation*}
\hbox{curl}_\h(u)=\in^{jk}\partial_ju^k,\qquad\hbox{div}_\h(u)=\partial_ju^j.
\end{equation*}

\subsubsection{Localizations}\label{ssec:loc}
In the following, we fix $\chi$ a $C^\infty$ nonincreasing radial bump function which equals $1$ when $\vert x\vert\le 1$ and which is supported on $B(0,2)$; we define
\begin{equation*}
\varphi(x)=\chi(x)-\chi(x/2)
\end{equation*}
and let $\varphi_k(x)=\varphi(2^{-k}x)$ and
\begin{equation*}
\varphi_k^\pm(x)=\varphi_k(x)\mathfrak{1}_{\{\pm x\ge0\}}.
\end{equation*}
We also introduce $\overline{\varphi}\in C^\infty(\mathbb{R}\setminus\{0\})$ a smooth function such that
\begin{equation}\label{Overlinephi}
\overline{\varphi}\varphi=\varphi.
\end{equation}

For $k\in\mathbb{Z}$ and $p,q\in\mathbb{Z}_-$, we define the Littlewood-Paley projection $P_k$ and the generalized projections $P_{k,p,q}$ by the formula
\begin{equation}\label{eq:loc_def}
\begin{split}
\mathcal{F}\left\{P_kf\right\}(\xi)&:=\varphi_k(\xi)\widehat{f}(\xi),\\
\mathcal{F}\left\{P_{k,p,q}f\right\}(\xi_1,\xi_2,\xi_3)&:=\varphi(2^{-k}\vert\xi\vert)\varphi(2^{-2(p+k)}(\xi_1^2+\xi_2^2))\varphi(2^{-q-k}\xi_3)\widehat{f}(\xi_1,\xi_2,\xi_3).
\end{split}
\end{equation}
Sometimes, we also denote by $P^\prime_{k,p,q}$ a multiplier with similar properties as $P_{k,p,q}$ such that $P_{k,p,q}P^\prime_{k,p,q}=P_{k,p,q}$. We note that $p,q$ are not independent parameters and in particular, we note that, on the support of $P_{k,p,q}$, there holds that $2^{2p+q}=\min\{2^{2p},2^q\}$ and $2^{2p}+2^q\simeq 1$.

Given $t\in[0,T]$, we choose a decomposition of the indicator function $\mathfrak{1}_{[0,t]}$ with functions $\tau_0,\ldots,\tau_{L+1}:\R \to [0,1]$, $\abs{L-\log_2(2+t)}\leq 2$, satisfying that
\begin{equation}
\begin{aligned}
& \mathrm{supp} \,\tau_0 \subseteq [0,2], \quad \mathrm{supp} \,\tau_{L+1}\subseteq [t-2,t],
  \quad \mathrm{supp}\,\tau_m\subseteq [2^{m-1},2^{m+1}]
  \quad \text{for} \quad  m \in \{1,\dots,L\},
\\
& \sum_{m=0}^{L+1}\tau_m(s) = \mathbf{1}_{[0,t]}(s), \quad \tau_m\in C^1(\R) \quad \text{and} \quad \int_0^t|\tau'_m(s)|\,ds\lesssim 1
  \quad \text{for} \quad m\in \{1,\ldots,L\}.
\end{aligned}
\end{equation}

\subsubsection{Multipliers}
Our nonlinearities are best understood in Fourier space and involve quadratic pseudo-products. Given a multiplier $\mathfrak{m}$, we define
\begin{equation}\label{eq:defQm}
\begin{split}
 \mathcal{F}\left\{Q_{\mathfrak{m}}[f,g]\right\}(\xi)&=\frac{1}{(2\pi)^\frac{3}{2}}\int_{\R^3}\mathfrak{m}(\xi,\eta)\widehat{f}(\xi-\eta)\widehat{g}(\eta) d\eta.
\end{split}
\end{equation}
We note the simple variant of H\"older's inequality
\begin{equation}\label{ProdRule}
\begin{split}
 \Vert Q_{\mathfrak{m}}[f,g]\Vert_{L^r}&\lesssim \Vert \mathcal{F}\mathfrak{m}\Vert_{L^1}\Vert f\Vert_{L^p}\Vert g\Vert_{L^q},\qquad \frac{1}{r}=\frac{1}{p}+\frac{1}{q},
\end{split}
\end{equation}
and we often consider a frequency-localized version of this:
\begin{equation*}
\begin{split}
\Vert \mathfrak{m}\Vert_{\mathcal{W}}&=\sup_{k,k_1,k_2}\Vert \mathfrak{m}\Vert_{\mathcal{W}_{kk_1k_2}},\qquad \Vert \mathfrak{m}\Vert_{\mathcal{W}_{kk_1k_2}}=\Vert \mathcal{F}\left\{\mathfrak{m}\varphi_{kk_1k_2}\right\}\Vert_{L^1},\\
\varphi_{kk_1k_2}(\xi_1,\xi_2)&:=\varphi(2^{-k}(\xi_1+\xi_2))\varphi(2^{-k_1}\xi_1)\varphi(2^{-k_2}\xi_2).
\end{split}
\end{equation*}
Our elementary multipliers are variants of the Riesz transforms:
\begin{equation*}
\begin{split}
 \mathcal{A}_\mathcal{R}&:=\{1,\,\frac{\xi_i}{\vert\xi\vert},\,\frac{\xi_j}{\vert\xi_\h\vert},\,\frac{\vert\xi_\h\vert}{\vert\xi\vert},\,\,\, 1\le i\le 3,\,\, 1\le j\le 2\},
\end{split}
\end{equation*}
and we shall also abbreviate $\Rp_j:=\frac{\xi_j}{\abs{\xi_\h}}$.
A special role will be played by the dispersion relation
\begin{equation*}
\Lambda(\xi)=\xi_3/\vert \xi\vert.
\end{equation*}
All multipliers in $\mathcal{A}_\mathcal{R}$ are bounded\footnote{For all but $\sqrt{1-\Lambda^2}$, this is a direct consequence of the boundedness of the Riesz transform; for $\sqrt{1-\Lambda^2}$ this follows from a linear variant of \eqref{ProdRule}.} on $L^p$, $1<p<\infty$.
In addition, one easily sees that these multipliers are bounded when localized in frequency
\begin{equation*}
\begin{split}
\Vert \Lambda P_kf\Vert_{L^p}+\Vert \sqrt{1-\Lambda^2}P_kf\Vert_{L^p}\lesssim \Vert f\Vert_{L^p},\quad 1\le p\le\infty.
\end{split}
\end{equation*}

\section{Structure of the Equations and Choice of Unknowns}\label{sec:structure}

\subsection{Symmetries and Vector Fields}
The equations \eqref{eq:IER} satisfy the following symmetries:
\begin{enumerate}
\item Scaling:
If $u$ is a solution, then so is
\begin{equation}
 u_\lambda(t,x)=\lambda u(t,\lambda^{-1}x).
\end{equation}
The generator of this symmetry is the vector field $S$ acting on vector fields $v$ and functions $f$ as
\begin{equation}\label{ScalingVF}
Sv=\sum_{j=1}^3x_j\partial_jv-v,\qquad Sf=x\cdot\nabla f.
\end{equation}

\item Rotation:
Notice that for any rotation matrix $\Theta\in O(3)$ on $\R^3$ we have $(v\cdot\nabla v)\circ\Theta=(\Theta^\intercal v\circ\Theta)\cdot\nabla(v\circ\Theta)$, so that if $v$ solves the Euler equations, then so does $\Theta^\intercal v\circ\Theta$ (with pressure $p\circ\Theta$). If moreover $\Theta$ is a rotation about the axis $\vec{e}_3\in\R$, then we obtain a rotation symmetry for \eqref{eq:IER}, which is generated by 
\begin{equation}\label{RotationVF}
\Omega v=(x_1\partial_2-x_2\partial_1)v-v_\h^\perp,\qquad \Omega f=(x_1\partial_2-x_2\partial_1)f.
\end{equation}

\end{enumerate}

In both cases, we observe that the vector field $V\in\{S,\Omega\}$ commutes with the Hodge decomposition and leads to the linearized equation:
\begin{equation}\label{eq:IER-VF}
 \partial_t V u+V u\cdot\nabla u+u\cdot\nabla V u+\vec{e}_3\times V u+\nabla V p=0,\quad \div\, V u=0.
\end{equation}

We note that both $S$ and $\Omega$ are natural in the sense that they correspond to flat derivatives in spherical coordinates $(\rho,\theta,\phi)$:
\begin{equation*}
\Omega=\partial_\theta,\qquad S=\rho\partial_\rho.
\end{equation*}
In particular, they commute and they both behave well under Fourier transform: we have 
\begin{equation}
 \widehat{Sf}=-3-S\hat{f}=S^*\hat{f},\qquad \widehat{\Omega f}=\Omega\hat{f}.
\end{equation}
In practice, we will thus be able to equivalently work with $\mathcal{F}^{-1}(V\hat{f})$ or $V f$, $V\in\{\Omega,S\}$ (since they differ by at most a multiple of $f$), and will henceforth ignore this distinction.

To complement these to have control of the full gradient, a natural choice is the spherical derivative corresponding to the vertical angle $\phi$,
\begin{equation}\label{DefUps}
 \Ups:=\partial_\phi,
\end{equation}
which commutes with $\Omega$ and $S$.
As we shall see, $\Ups$ is very well adapted to the structure of the equations and the geometry of the nonlinear phases we encounter (see also the discussion in Section \ref{sec:geom}).

\subsection{Choice of Unknowns and Nonlinearity in Axisymmetry}
To motivate our choice of variables, we first observe that the linear part of \eqref{eq:IER},
\begin{equation}
 \partial_t u+\vec{e}_3\times u+\nabla p=0, \quad \div\; u=0,
\end{equation}
is dispersive. Here $\Delta p=\curl_\h u$, so using the divergence condition one sees directly that the linear system is equivalent to
\begin{equation}\label{eq:linIER}
\begin{aligned}
\partial_t\curl_\h u-\partial_3u_3&=0,\qquad
\partial_tu_3+\partial_3\Delta^{-1}\curl_\h u=0.
\end{aligned}
\end{equation}
The dispersion relations $\pm i\Lambda(\xi)$ thus satisfy $(i\Lambda)^2=-\frac{\xi_3^2}{\abs{\xi}^2}$, and we choose
\begin{equation}
 \Lambda(\xi)=\frac{\xi_3}{\abs{\xi}}.
\end{equation}
Furthermore, we will denote by $\Lam$ the associated real operator
\begin{equation}
 \Lam=i\Lambda=\partial_3\vert\nabla\vert^{-1}.
\end{equation}

\subsubsection{Scalar unknowns}

Note that, thanks to the incompressibility condition, $u$ has two scalar degrees of freedom. In order to exploit this better, we will work with the variables
\begin{equation}\label{eq:a-c-def}
 a:=\abs{\nabla_\h}^{-1}\curl_\h u,\qquad c:=\abs{\nabla} \abs{\nabla_\h}^{-1}u_3
\end{equation}
which are motivated by the property \eqref{MainPropU}. Note that $u$ can be recovered from $(a,c)$:
\begin{equation}\label{Formu0}
 u=U_a+U_c,
\end{equation}
where
\begin{equation}\label{Formu}
 \begin{split}
 U_a&:=-\nabla_\h^\perp\abs{\nabla_\h}^{-1}a,\qquad U_a^j=\in^{jk}\vert\nabla_\h\vert^{-1}\partial_ka,\\
 U_c&:= \Lam\nabla_\h\abs{\nabla_\h}^{-1}c+\sqrt{1-\Lambda^2}c\;\vec{e}_3,\qquad U_c^j={\bf\Lambda}\vert\nabla_\h\vert^{-1}\partial_jc,\quad U_c^3=\sqrt{1-\Lambda^2}c
\end{split}
\end{equation}
and for any vector field $V\in\{S,\Omega\}$ and any Fourier multiplier $m:\mathbb{R}^3\to\mathbb{R}$,
\begin{equation}\label{MainPropU}
\begin{aligned}
 Vu&=U_{Va}+U_{Vc},\qquad V\in\{S,\Omega\},\\
 \norm{mu}_{L^2}^2&=\norm{ma}_{L^2}^2+\norm{mc}_{L^2}^2=\norm{mU_a}_{L^2}^2+\norm{mU_c}_{L^2}^2.\\
\end{aligned}
\end{equation}
Note that the linearization of $a$ coincides with (the derivative of) the pressure
\begin{equation}
 \Delta p=\vert\nabla_\h\vert a-\div\left[u\cdot\nabla u\right]=\vert\nabla_\h\vert a-\partial_\alpha\partial_\beta\left[u^\alpha u^\beta\right].
\end{equation}
Plugging in the equation, we obtain that
\begin{equation}\label{eq:IER2}
\begin{aligned}
\partial_ta-\Lam c&=-\vert\nabla_\h\vert^{-1}\partial_j\partial_p\in^{jk}\left\{u^pu^k\right\}-\vert\nabla\vert{\bf\Lambda}\in^{jk}\vert\nabla_\h\vert^{-1}\partial_j\left\{u^3u^k\right\},\\
\partial_tc-\Lam a&=-\vert\nabla\vert{\bf\Lambda}\sqrt{1-\Lambda^2}\left\{\vert\nabla_\h\vert^{-2}\partial_j\partial_k\left\{u^ju^k\right\}+u^3u^3\right\}-\vert\nabla\vert \vert\nabla_\h\vert^{-1}\partial_j(1-2\Lambda^2)\left\{u^ju^3\right\}.
\end{aligned}
\end{equation}
In this form the structure of the nonlinearity is more apparent: it is a quasilinear quadratic form in $a,c$ without singularities at low frequency. In addition, we will see next that in the case of axisymmetry, there is more useful structure.

\subsubsection{The equations in axisymmetry}\label{sec:axisymm}
Let us now consider the special class of axisymmetric initial data for \eqref{eq:IER}, namely those that are symmetric with respect to rotations about the $\vec{e}_3$ axis. This property is preserved by the nonlinear flow of the equation, as witnessed by the rotational symmetry.

Moreover, in this axisymmetric setting the nonlinearity of \eqref{eq:IER2} has a simplified form: using the \emph{dispersive unknowns}
\begin{equation}
 U_+:=a+c,\quad U_-:=a-c,
\end{equation}
the equations \eqref{eq:IER2} can be written as the following system
\begin{equation}\label{eq:IER_disp}
\begin{aligned}
 (\partial_t-\Lam)U_+&=\frac{1}{4}\left[Q_{m^{++}_1+m^{++}_2}[U_+,U_+]+Q_{m^{+-}_1+m^{+-}_2}[U_+,U_-]+Q_{-m^{++}_1+m^{++}_2}[U_-,U_-]\right],\\
 (\partial_t+\Lam)U_-&=\frac{1}{4}\left[Q_{m^{++}_1-m^{++}_2}[U_+,U_+]+Q_{m^{+-}_1-m^{+-}_2}[U_+,U_-]+Q_{-m^{++}_1-m^{++}_2}[U_-,U_-]\right].
\end{aligned}
\end{equation}

\begin{lemma}\label{lem:IERmult}
 In the axisymmetric case, the equations \eqref{eq:IER2} can equivalently be written as \eqref{eq:IER_disp}. 
 Here the multipliers are such that there exist constants $c_{\alpha\beta}^{jk}\in\{-1,0,1\}$ with $\alpha=(\alpha_1,\alpha_2)\in\{0,1,2\}^2$, $\beta=(\beta_1,\beta_2)\in\{0,1\}^2$ and $1\leq j,k\leq 2$, for which
\begin{equation}
 m_i^{\mu\nu}(\xi,\eta)=\abs{\xi}\cdot\!\!\!\!\prod_{\zeta\in\{\xi,\xi-\eta,\eta\}}\sum_{\substack{\alpha\in\{0,1,2\}^2,\\\beta\in\{0,1\}^2,1\leq j,k\leq 2}} c_\alpha^{jk} [\Lambda(\zeta)]^{\alpha_1}[\sqrt{1-\Lambda(\zeta)^2}]^{\alpha_2}[\Rp_j(\zeta)]^{\beta_1}[\Rp_k(\zeta)]^{\beta_2}.
\end{equation}
Their precise expressions are given below in \eqref{eq:mult1} and \eqref{eq:mult2} -- here we report on two key features:
\begin{enumerate}
 \item $m_i^{\mu\nu}$ always contains a factor $\Lambda(\zeta_1)\sqrt{1-\Lambda^2(\zeta_2)}$ for some $\zeta_1,\zeta_2\in\{\xi,\xi-\eta,\eta\}$, i.e.\ it contains at least one copy of $\Lambda$ and $\sqrt{1-\Lambda^2}$.
 \item The Riesz transforms $\Rp_j$ always appear as ``angles'' between an input and the output, i.e.\ in the form $\Rp_j(\xi)\in^{jk}\Rp_k(\vartheta)$ or $\Rp_j(\xi)\cdot\Rp_j(\vartheta)$ for some $\vartheta\in\{\xi-\eta,\eta\}$.
\end{enumerate}
\end{lemma}

\begin{remark}
 This lemma shows that the multipliers $m_i^{\mu\nu}$ are of the form $m_i^{\mu\nu}=\abs{\nabla}n_i^{\mu\nu},$ with $i\in\{1,2\},\mu,\nu\in\{-,+\}$, for multipliers $n_i^{\mu\nu}$ that are products and linear combinations (in all three arguments of the output and the two inputs) of the basic multipliers $\Lambda$, $\sqrt{1-\Lambda^2}$ and $\Rp_k$ (and thus homogeneous of degree $0$).
\end{remark}

\begin{proof}
We remark that under the assumption of axisymmetry we have
\begin{equation*}
\begin{split}
\in^{jk}\partial_j\partial_p\left\{u^pu^k\right\}&=-\partial_j\partial_p\left\{\vert\nabla_\h\vert^{-1}\partial_ja\cdot {\bf\Lambda}\vert\nabla_\h\vert^{-1}\partial_pc\right\}+\in^{jk}\in^{pq}\partial_j\partial_p\left\{\partial_q(\vert\nabla_\h\vert^{-1}a)\cdot\partial_k({\bf\Lambda}\vert\nabla_\h\vert^{-1}c)\right\}\\
\partial_j\partial_k\left\{u^ju^k\right\}&=\partial_j\partial_k\left\{{\bf\Lambda}\vert\nabla_\h\vert^{-1}\partial_jc\cdot{\bf\Lambda}\vert\nabla_\h\vert^{-1}\partial_kc\right\}+\in^{jp}\in^{kq}\partial_j\partial_k\left\{\partial_p(\vert\nabla_\h\vert^{-1}a)\cdot\partial_q(\vert\nabla_\h\vert^{-1}a)\right\}\\
\in^{jk}\partial_j\left\{u^3u^k\right\}&=-\partial_j\left\{\vert\nabla_\h\vert^{-1}\partial_ja\cdot \sqrt{1-\Lambda^2}c\right\}\\
\partial_j\left\{u^ju^3\right\}&=\partial_j\left\{{\bf\Lambda}\vert\nabla_\h\vert^{-1}\partial_jc\cdot\sqrt{1-\Lambda^2}c\right\}
\end{split}
\end{equation*}
and the equations become
\begin{equation}\label{eq:IER3}
\begin{aligned}
\partial_ta-\Lam c&=\vert\nabla\vert\sqrt{1-\Lambda^2}\vert\nabla_\h\vert^{-2}\partial_j\partial_p\left\{\vert\nabla_\h\vert^{-1}\partial_ja\cdot {\bf\Lambda}\vert\nabla_\h\vert^{-1}\partial_pc\right\}+\vert\nabla\vert{\bf\Lambda}\vert\nabla_\h\vert^{-1}\partial_j\left\{\vert\nabla_\h\vert^{-1}\partial_ja\cdot\sqrt{1-\Lambda^2}c\right\}\\
&\quad-\vert\nabla\vert \sqrt{1-\Lambda^2}\in^{jk}\in^{pq}\vert\nabla_\h\vert^{-2}\partial_j\partial_p\left\{\vert\nabla_\h\vert^{-1}\partial_qa\cdot{\bf\Lambda}\vert\nabla_\h\vert^{-1}\partial_kc\right\}\\
\partial_tc-\Lam a
&=-\vert\nabla\vert\vert\nabla_\h\vert^{-1}\partial_j(1-2\Lambda^2)\left\{{\bf\Lambda}\vert\nabla_\h\vert^{-1}\partial_jc\cdot\sqrt{1-\Lambda^2}c\right\}-\vert\nabla\vert{\bf\Lambda}\sqrt{1-\Lambda^2}\left\{\sqrt{1-\Lambda^2}c\cdot\sqrt{1-\Lambda^2}c\right\}\\
&\quad-\vert\nabla\vert{\bf\Lambda}\sqrt{1-\Lambda^2}\vert\nabla_\h\vert^{-2}\partial_j\partial_k\left\{{\bf\Lambda}\vert\nabla_\h\vert^{-1}\partial_jc\cdot{\bf\Lambda}\vert\nabla_\h\vert^{-1}\partial_kc\right\}\\
&\quad-\vert\nabla\vert{\bf\Lambda}\sqrt{1-\Lambda^2}\vert\nabla_\h\vert^{-2}\partial_j\partial_k\in^{jp}\in^{kq}\left\{\vert\nabla_\h\vert^{-1}\partial_pa\cdot\vert\nabla_\h\vert^{-1}\partial_qa\right\}
\end{aligned}
\end{equation}
To simplify notation we introduce the bilinear expressions
\begin{equation}\label{eq:mult0}
\begin{aligned}
 Q_{m^1}[f,g]&:=\abs{\nabla}\sqrt{1-\Lambda^2}\Rp_j\Rp_p\{\Rp_j f\cdot\Rp_p g\},\\
 Q_{m^2}[f,g]&:=\abs{\nabla}\sqrt{1-\Lambda^2}\Rp_j\Rp_p\in^{jk}\in^{pq}\{\Rp_k f\cdot\Rp_q g\},\\
 Q_{m^3}[f,g]&:=\abs{\nabla}\Rp_j\{\Rp_j f\cdot\sqrt{1-\Lambda^2}g\},\\
 Q_{m^4}[f,g]&:=\abs{\nabla}\sqrt{1-\Lambda^2}\{\sqrt{1-\Lambda^2}f\cdot\sqrt{1-\Lambda^2}g\}
\end{aligned}
\end{equation}
with which we can express the equations \eqref{eq:IER3} as
\begin{equation}\label{eq:IER4}
\begin{aligned}
 \partial_ta-\Lam c&=Q_{m^1-m^2}[a,\Lam c]+\Lam Q_{m^3}[a,c],\\
 \partial_tc-\Lam a&=-\Lam Q_{m^1}[\Lam c,\Lam c]-\Lam Q_{m^2}[a,a]-(1-2\Lambda^2)Q_{m^3}[\Lam c,c]-\Lam Q_{m^4}[c,c].
\end{aligned} 
\end{equation}
Now we substitute the dispersive unknowns via
\begin{equation}
 U_+:=a+c,\quad U_-:=a-c,
\end{equation}
and simplify to obtain \eqref{eq:IER_disp} with
\begin{equation}\label{eq:mult1}
\begin{aligned}
 Q_{m^{++}_1}[f,g]&:=\abs{\nabla}\sqrt{1-\Lambda^2}\Rp_j\Rp_p\{\delta^{jk}\delta^{pq}-\in^{jk}\in^{pq}\}\{\Rp_k f\cdot\Lam\Rp_q g\}+\Lam\abs{\nabla}\Rp_j\{\Rp_j f\cdot\sqrt{1-\Lambda^2}g\},\\
 Q_{m^{+-}_1}[f,g]&:=\abs{\nabla}\sqrt{1-\Lambda^2}\Rp_j\Rp_p\{\delta^{jk}\delta^{pq}-\in^{jk}\in^{pq}\}\{\Lam\Rp_k f\cdot\Rp_q g-\Rp_k f\cdot\Lam\Rp_q g\}\\
 &\qquad -\Lam\abs{\nabla}\Rp_j\{\Rp_j f\cdot\sqrt{1-\Lambda^2}g-\sqrt{1-\Lambda^2}f\cdot\Rp_j g\},
\end{aligned}
\end{equation}
and
\begin{equation}\label{eq:mult2}
\begin{aligned}
 Q_{m^{++}_2}[f,g]&:=-\Lam\abs{\nabla}\sqrt{1-\Lambda^2}\Rp_j\Rp_p\{\Rp_j \Lam f\cdot\Rp_p \Lam g+\in^{jk}\in^{pq}\Rp_k f\cdot\Rp_q g\}\\
 &\qquad -(1-2\Lambda^2)\abs{\nabla}\Rp_j\{\Rp_j \Lam f\cdot\sqrt{1-\Lambda^2}g\}-\Lam\abs{\nabla}\sqrt{1-\Lambda^2}\{\sqrt{1-\Lambda^2}f\cdot\sqrt{1-\Lambda^2}g\},\\
 Q_{m^{+-}_2}[f,g]&:=\Lam\abs{\nabla}\sqrt{1-\Lambda^2}\Rp_j\Rp_p\{\Rp_j \Lam f\cdot\Rp_p \Lam g+\in^{jk}\in^{pq}\Rp_k f\cdot\Rp_q g\}\\
 &\qquad +(1-2\Lambda^2)\abs{\nabla}\Rp_j\{\Rp_j \Lam f\cdot\sqrt{1-\Lambda^2}g+\sqrt{1-\Lambda^2}f\cdot\Rp_j \Lam g\}\\
 &\qquad +2\Lam\abs{\nabla}\sqrt{1-\Lambda^2}\{\sqrt{1-\Lambda^2}f\cdot\sqrt{1-\Lambda^2}g\}.
\end{aligned} 
\end{equation}
\end{proof}

\begin{notation}[Notation for Multipliers]
 To simplify notation, in the rest of this article we shall reserve the notation $\m$ for any one of the multipliers of the above Lemma \ref{lem:IERmult}. We highlight again that these multipliers always contain a factor of $\abs{\xi}$, (at least) one factor of $\Lambda(\zeta_1)$ and (at least) one factor $\sqrt{1-\Lambda^2(\zeta_2)}$, for $\zeta_1,\zeta_2\in\{\xi,\xi-\eta,\eta\}$.
\end{notation}

\subsubsection{Profiles and Bilinear Expressions}
Next we introduce the \emph{profiles} $\U_\pm$ of our dispersive unknowns $U_\pm$ as
\begin{equation}\label{eq:def_profiles}
 \U_+:=e^{-t\Lam}U_+,\qquad \U_-:=e^{t\Lam}U_-.
\end{equation}
We can then express \eqref{eq:IER_disp} in terms of these profiles, and see that the bilinear terms are of the form
\begin{equation}\label{eq:def_Qm}
 \Q_\m[\U_\mu,\U_\nu]:=\mathcal{F}^{-1}\left(\int_{\R^3}\m(\xi,\eta)e^{\pm it\Phi_{\mu\nu}(\xi,\eta)}\widehat{\U_\mu}(\xi-\eta)\widehat{\U_\nu}(\eta)d\eta\right),\qquad \mu,\nu\in\{-,+\},
\end{equation}
for a phase function
\begin{equation}\label{eq:def_phi}
 \Phi_{\mu\nu}(\xi,\eta):=\Lambda(\xi)+\mu\Lambda(\xi-\eta)+\nu\Lambda(\eta),\qquad \mu,\nu\in\{-,+\}
\end{equation}
and with $\m$ being one of the multipliers of Lemma \ref{lem:IERmult}. By Duhamel's formula we thus have from \eqref{eq:IER_disp} that
\begin{equation}\label{eq:IER_disp_Duham}
\begin{aligned}
 \U_+(t)=\U_+(0)+\frac{1}{4}\int_0^t  \left[\Q_{m^{++}_1+m^{++}_2}[\U_+,\U_+]+\Q_{m^{+-}_1+m^{+-}_2}[\U_+,\U_-]+\Q_{-m^{++}_1+m^{++}_2}[\U_-,\U_-]\right] ds,\\
 \U_-(t)=\U_-(0)+\frac{1}{4}\int_0^t  \left[\Q_{m^{++}_1-m^{++}_2}[\U_+,\U_+]+\Q_{m^{+-}_1-m^{+-}_2}[\U_+,\U_-]+\Q_{-m^{++}_1-m^{++}_2}[\U_-,\U_-]\right] ds.
\end{aligned} 
\end{equation}
We will use this expression as the basis for our bootstrap arguments.

\begin{remark}[More on the role of axisymmetry]\label{rem:role_axisym}
 Another important consequence of the axisymmetry assumption is that if $f$ is axisymmetric, then automatically 
 \begin{equation}
  \Omega f=0.
 \end{equation}
 This simplifies some more terms (see e.g.\ Lemma \ref{lem:VFcross}). However, the vector field $\Omega$ still plays an important role, since it does not vanish on expressions of several arguments, such as the phase functions $\Phi$ (see e.g.\ Lemma \ref{lem:vfsizes-mini}).
\end{remark}

\paragraph{Localizations}
To quantify the nonlinear interactions we will often localize both input and output frequencies in \eqref{eq:def_Qm} in both horizontal and vertical directions, as in \eqref{eq:loc_def}. This will be done frequently and not always explicitly, so we fix here the notational convention that for expressions involving frequencies $\xi$, $\xi-\eta$ and $\eta$ we write their localization parameters when relevant as 
\begin{equation}
 \begin{alignedat}{3}
  &\abs{\xi}\sim 2^k,\quad &&\sqrt{1-\Lambda^2(\xi)}\sim 2^{p},\quad &&\abs{\Lambda(\xi)}\sim 2^{q},\\
  &\abs{\xi-\eta}\sim 2^{k_1},\quad &&\sqrt{1-\Lambda^2(\xi-\eta)}\sim 2^{p_1},\quad &&\abs{\Lambda(\xi-\eta)}\sim 2^{q_1},\\
  &\abs{\eta}\sim 2^{k_2},\quad &&\sqrt{1-\Lambda^2(\eta)}\sim 2^{p_2},\quad &&\abs{\Lambda(\eta)}\sim 2^{q_2}.
 \end{alignedat}
\end{equation}
For $w\in\{k,p,q\}$ we will then write
\begin{equation}
 w_{\max}:=\max\{w,w_1,w_2\},\qquad w_{\min}:=\min\{w,w_1,w_2\}.
\end{equation}

\section{Norms and Overview of Proof}\label{sec:overview}
We introduce the norms
\begin{equation}\label{eq:defnorms}
\begin{aligned}
 \norm{f}_B&:=\sup_{k,p,q} 2^{-p-\frac{q}{2}}\norm{P_{k,p,q}f}_{L^2},\\
 \norm{f}_D&:=\sup_{0\le \vert a\vert\le 3}\left\{\norm{S^af}_B+\norm{\Ups S^af}_{L^2}\right\}.
\end{aligned}
\end{equation}

Our Theorem \ref{MainThm} then follows from the following bootstrap:
\begin{theorem}\label{thm:btstrap}
 Let $\mathcal{M}>0$ be given, and define $N_0:=25\mathcal{M}$. Then there exists $\eps_0^\ast \in(0,\frac{1}{8})$ such that the following holds: For $\eps_0\in(0,\eps_0^\ast)$ and $\eps_1:=2\eps_0$, assume that $U_\pm$ are solutions to \eqref{eq:IER_disp} 
 with profiles $\U_{\pm}:=e^{\mp it\Lambda}U_{\pm}$ and satisfy for some $N\geq 5$ that
\begin{equation}
\begin{aligned}\label{eq:assump_id}
 \norm{U_\pm(t=0)}_{H^{2N_0}\cap H^{-1}}+\norm{S^aU_\pm(t=0)}_{L^2\cap H^{-1}}&\le\varepsilon_0,\qquad 0\le a\le 4N,\\
 \norm{S^a\U_\pm(t=0)}_{B}&\le\varepsilon_0,\qquad 0\le a\le 2N,\\
 \norm{\Ups S^a\U_\pm(t=0)}_{L^2}&\le\varepsilon_0,\qquad 0\le a\le N+3.  
\end{aligned}
\end{equation}
Then if for $0\leq t\lesssim \varepsilon_0^{-\mathcal{M}}$ there holds that
\begin{equation}\label{eq:assump_btstrap}
\begin{aligned}
 \norm{U_\pm(t)}_{H^{2N_0}\cap H^{-1}}+\norm{S^aU_\pm(t)}_{L^2\cap H^{-1}}&\le 2\varepsilon_1,\qquad 0\le a\le 4N,\\
 \norm{S^a\U_\pm(t)}_{B}&\le 2\varepsilon_1,\qquad 0\le a\le 2N,\\
 \norm{\Ups S^a\U_\pm(t)}_{L^2}&\le  2\varepsilon_1,\qquad 0\le a\le N+3,
\end{aligned}
\end{equation}
it follows that in fact there hold the improved estimates
\begin{align}
 \norm{U_\pm(t)}_{H^{2N_0}\cap H^{-1}}+\norm{S^aU_\pm(t)}_{L^2\cap H^{-1}}&\le\varepsilon_1,\qquad 0\le a\le 4N,\\
 \norm{S^a\U_\pm(t)}_{B}&\le\varepsilon_1,\qquad 0\le a\le 2N,\\
 \norm{\Ups S^a\U_\pm(t)}_{L^2}&\le  \varepsilon_1,\qquad 0\le a\le N+3.
\end{align}

\end{theorem}

\subsection{Proof of Theorem \ref{thm:btstrap}} 
Note that under the above assumptions \eqref{eq:assump_btstrap}, by definition of the norm $D$ in \eqref{eq:defnorms} there also holds that
 \begin{equation}\label{eq:btstrap_D}
  \norm{S^a\U_\pm(t)}_{D}\lesssim \norm{(1,S^3)S^a \U_\pm(t)}_B+\norm{\Ups (1,S^3)S^a\U_\pm(t)}_{L^2}\lesssim \varepsilon_1,\qquad 0\leq a\leq N.
 \end{equation}
Moreover, we can interpolate to obtain that
\begin{equation}\label{eq:vfHNinterpol}
 \norm{S^a\U_\pm(t)}_{H^{N_0}}\le\varepsilon_1,\qquad 0\le a\le 2N.
\end{equation}

We then show that:
\begin{enumerate}
 \item\label{it:disp_decay} (Dispersive Decay) We have pointwise decay for $S^a\U_\pm(t)$, $0\leq a\leq N$, since in Section \ref{sec:lin_decay}, Proposition \ref{DecayProp}, we show that
 \begin{equation}
  \norm{e^{\pm it\Lambda}P_{k,p,q}f}_{L^\infty}\lesssim \ip{t}^{-1}2^{\frac{3}{2}k}\norm{f}_{D}.
 \end{equation}
 \item\label{it:en_est} (Energy estimates) We have the energy bounds
 \begin{equation}
  \norm{U_\pm(t)}_{H^{2N_0}\cap H^{-1}}+\norm{S^aU_\pm(t)}_{L^2\cap H^{-1}}\le \varepsilon_1,\qquad 0\le a\le 4N,
 \end{equation}
 as we show in Section \ref{sec:energy}, Corollary \ref{cor:EE}.
 
 \item\label{it:BVF_prop} (Upgrade of $L^2$ to $B$ bounds on the vector fields $S$)
 Then there holds that
 \begin{equation}
  \norm{S^a\U_\pm(t)}_{B}\leq \varepsilon_1,\qquad 0\leq a\leq 2N,
 \end{equation}
 as we show in Section \ref{sec:B_VFprop}, Proposition \ref{prop:B_VFprop}.
 
 \item\label{it:Ups_prop} (Propagation of the $L^2$ norm of $\Ups S^a \U_\pm$) We have that
 \begin{equation}
  \norm{\Ups S^a\U_\pm(t)}_{L^2}\leq \varepsilon_1,\qquad 0\leq a\leq N+3,
 \end{equation}
 as is shown in Section \ref{sec:Ups_propagation}, Proposition \ref{prop:Ups_prop}.
\end{enumerate}

The other sections not mentioned above provide the foundation for these arguments: In Section \ref{sec:geom} we discuss how the vector fields $V$ and $\Ups$ interact with the nonlinear expressions we encounter. Section \ref{sec:bilin} gives a bilinear estimate that enables the propagation of vector fields $V\in\{S,\Omega\}$ in the $B$ norm as in Step \eqref{it:BVF_prop}. Next, in Section \ref{sec:more_bilin} we give more bilinear estimates as they are relevant for repeated integration by parts needed for the propagation of $\Ups$ in $L^2$, Step \eqref{it:Ups_prop}. Finally, in the appendix we collect some basic, useful computations that are used at various points of the article, but pose no conceptual challenges.

\section{Linear Decay}\label{sec:lin_decay}

In this section we demonstrate the following dispersive decay estimate for the semigroup $e^{it\Lambda}$ generated by the linear flow of our equations:
\begin{proposition}\label{DecayProp}
Let $f$ be axisymmetric. Then there holds that
\begin{equation}\label{eq:Linfdecay}
 \norm{P_{k,p,q}e^{it\Lambda}f}_{L^\infty}\lesssim t^{-1}2^{\frac{3k}{2}}\norm{f}_D.
\end{equation}
\end{proposition}

\begin{remark}\label{rem:sharpdecay}
Proposition \ref{DecayProp} implies a translation invariant estimate by writing $e^{it\Lambda}P_{k,p,q}f=\left(e^{it\Lambda}P_{k,p,q}\delta\right)\ast f$ so that we get
\begin{equation*}
\begin{split}
\Vert e^{it\Lambda}P_{k,p,q}f\Vert_{L^\infty}&\lesssim \Vert e^{it\Lambda}P^\prime_{k,p,q}\delta\Vert_{L^\infty}\Vert P_{k,p,q}f\Vert_{L^1}\lesssim 2^{3k}\langle t\rangle^{-1}\Vert P_{k,p,q}f\Vert_{L^1}.
\end{split}
\end{equation*}
It can be seen that this decay rate is sharp: If $f\in L^2\cap C^0(\R^3)$ is radial, then one computes that $(e^{it\Lambda}f)(0)=\frac{\sin(t)}{t}f(0)$.
\end{remark}

The proof of Proposition \ref{DecayProp} establishes that the oscillatory integral
\begin{equation}\label{Int}
 I_{k,p,q}(f)({\bf x},t):=\int_{\mathbb{R}^3}e^{i\left[t\Lambda(\xi)+\langle x,\xi\rangle\right]}\widehat{P_{k,p,q}f}(\xi_1,\xi_2,\xi_3)d\xi=P_{k,p,q}e^{it\Lambda}f({\bf x})
\end{equation}
is bounded as
\begin{equation}
 \abs{I_{k,p,q}(f)({\bf x},t)}\lesssim \ip{t}^{-1}2^\frac{3k}{2}\sum_{n=0}^3\left\{\norm{\Upsilon S^nf}_{L^2}+\norm{S^nf}_B\right\}.
\end{equation}

\begin{proof}[Proof of Proposition \ref{DecayProp}]
 Recall that
\begin{equation*}
\Upsilon=\partial_\phi=-\sqrt{1-\Lambda^2}\partial_\Lambda,
\end{equation*}
which follows from the change of variables
\begin{equation*}
\begin{split}
 {\bf x}&=(\rho\cos\theta\sin\varphi,\rho\sin\theta\sin\varphi,\rho\cos\varphi)=\rho(\sqrt{1-\Lambda^2}\cos\theta,\sqrt{1-\Lambda^2}\sin\theta,\Lambda),\qquad\Lambda=\cos\varphi,
\end{split}
\end{equation*}
from which we deduce moreover that
\begin{equation*}
\begin{split}
 d{\bf x}=\rho^2\sin\varphi d\theta d\varphi d\rho=\rho^2d\theta d\Lambda d\rho.
\end{split}
\end{equation*}

By scaling and rotation symmetry, we may assume that $k=0$ and ${\bf x}=(x,0,z)$ for some $x\ge0$. If $t2^{2p+q}\le C$, we simply use a crude integration to get
\begin{equation*}
\begin{split}
\vert I_{0,p,q}\vert\lesssim 2^{2p+q}\Vert \widehat{f}\Vert_{L^\infty}
\end{split}
\end{equation*}
and we can use Lemma \ref{lem:ControlLinfty}. If $t2^{2p+q}\ge1$, we switch to spherical coordinates $\xi\mapsto(\rho,\Lambda,\theta)$ and, after integrating in $\theta$, we need to consider the integral
\begin{equation*}
\begin{split}
I(x,z,t)&:=\int_{\mathbb{R}_+\times[-1,1]}e^{i\left[t\Lambda+\rho\Lambda z\right]}\varphi(2^{-p}\rho\sqrt{1-\Lambda^2})\varphi(2^{-q}\rho\Lambda)\cdot J_0(\rho\sqrt{1-\Lambda^2}x)\cdot \widehat{f}\rho^2\varphi(\rho)d\rho d\Lambda.
\end{split}
\end{equation*}
By standard results on Bessel functions, this reduces to studying
\begin{equation*}
\begin{split}
I^\pm(x,z,t)&:=\int_{\mathbb{R}_+\times[-1,1]}e^{i\Psi}\varphi(2^{-p}\rho\sqrt{1-\Lambda^2})\varphi(2^{-q}\rho\Lambda)\cdot H_\pm(\rho\sqrt{1-\Lambda^2}x)\cdot \widehat{f}\rho^2\varphi(\rho)d\rho d\Lambda,\\
\Psi&:=t\Lambda+\rho\left[\Lambda z\pm\sqrt{1-\Lambda^2}x\right],
\end{split}
\end{equation*}
where
\begin{equation*}
\left\vert\left(\frac{d}{dx}\right)^aH_\pm(x)\right\vert\lesssim \langle x\rangle^{-\frac{1}{2}-a}.
\end{equation*}
We focus on the case with sign $+$; the other estimate is similar. We can compute the gradient
\begin{equation}\label{SPgradients}
\begin{split}
\partial_\Lambda\Psi&=t+\rho\left[z-\frac{\Lambda}{\sqrt{1-\Lambda^2}} x\right],\qquad\partial_\rho\Psi=\Lambda z+\sqrt{1-\Lambda^2}x,\\
\partial_\Lambda^2\Psi&=-\frac{\rho x}{\left[1-\Lambda^2\right]^\frac{3}{2}},\qquad\partial_\Lambda\partial_\rho\Psi=z-\frac{\Lambda}{\sqrt{1-\Lambda^2}} x,\qquad\partial_\rho^2\Psi=0.
\end{split}
\end{equation}

\medskip

{\bf Case 1}:
\begin{equation*}
\begin{split}
0\le x\le C^{-1}t2^{p+q},\qquad\hbox{ and }\qquad \vert z\vert\le C^{-1} t2^{2p}
\end{split}
\end{equation*}
In these conditions, there holds that
\begin{equation*}
\begin{split}
\vert\partial_\Lambda\Psi\vert\ge t/2
\end{split}
\end{equation*}
and we can integrate by parts to get
\begin{equation*}
\begin{split}
I^+(x,z,t)&=i(I_1+I_2),\\
I_1&:=\int_{\mathbb{R}_+\times[-1,1]}e^{i\Psi}\frac{1}{\partial_\Lambda\Psi}\varphi(2^{-p}\rho\sqrt{1-\Lambda^2})\varphi(2^{-q}\rho\Lambda)\cdot H_\pm(\rho\sqrt{1-\Lambda^2}x)\cdot \partial_\Lambda\widehat{f}\rho^2\varphi(\rho)d\rho d\Lambda\\
I_2&:=\int_{\mathbb{R}_+\times[-1,1]}e^{i\Psi}\partial_\Lambda\left\{\frac{1}{\partial_\Lambda\Psi}\varphi(2^{-p}\rho\sqrt{1-\Lambda^2})\varphi(2^{-q}\rho\Lambda)\cdot H_\pm(\rho\sqrt{1-\Lambda^2}x)\right\}\cdot \widehat{f}\rho^2\varphi(\rho)d\rho d\Lambda
\end{split}
\end{equation*}
and using also \eqref{SPgradients}, we see that
\begin{equation*}
\begin{split}
\left\vert\frac{\partial^2_\Lambda\Psi}{(\partial_\Lambda\Psi)^2}\right\vert+\left\vert\partial_\Lambda (\varphi(2^{-q}\Lambda))\right\vert\lesssim t^{-1}2^{-2p-q}
\end{split}
\end{equation*}
and therefore
\begin{equation*}
\begin{split}
\vert I_1\vert&\lesssim 2^{p+\frac{q}{2}}\Vert \partial_\Lambda \widehat{f}\Vert_{L^2}\lesssim 2^\frac{q}{2}\Vert \Upsilon f\Vert_{L^2},\\
\vert I_2\vert&\lesssim t^{-1}2^{-2p}\cdot 2^{2p+q}\cdot\Vert\widehat{f}\Vert_{L^\infty}
\end{split}
\end{equation*}
and we can use Lemma \ref{lem:ControlLinfty} again.

\medskip

{\bf Case 2}:
\begin{equation*}
\begin{split}
x\ge C^{-1}t2^{p+q},\,\hbox{ and }\,\vert z\vert\le C^{-2}t2^{2p}\qquad\hbox{ or }\qquad  \vert x\vert\le C^{-2}t2^{p+q},\,\hbox{ and }\,\vert z\vert\ge C^{-1}t2^{2p}
\end{split}
\end{equation*}
in this case, we have that
\begin{equation*}
\vert\partial_\rho\Psi\vert\gtrsim t2^{2p+q}
\end{equation*}
 and we can integrate by parts with respect to $\rho$ to obtain
 \begin{equation*}
\begin{split}
I^+(x,z,t)&=i(I_1+I_2),\\
I_1&:=\int_{\mathbb{R}_+\times[-1,1]}e^{i\Psi}\frac{1}{\partial_\rho\Psi}\varphi(2^{-p}\rho\sqrt{1-\Lambda^2})\varphi(2^{-q}\rho\Lambda)\cdot H_\pm(\rho\sqrt{1-\Lambda^2}x)\cdot \partial_\rho\widehat{f}\rho^2\varphi(\rho)d\rho d\Lambda\\
I_2&:=\int_{\mathbb{R}_+\times[-1,1]}e^{i\Psi}\frac{1}{\partial_\rho\Psi}\partial_\rho\left\{\varphi(2^{-p}\rho\sqrt{1-\Lambda^2})\varphi(2^{-q}\rho\Lambda)\cdot H_\pm(\rho\sqrt{1-\Lambda^2}x)\right\}\cdot \widehat{f}\rho^2\varphi(\rho)d\rho d\Lambda
\end{split}
\end{equation*}
and a crude integration gives
\begin{equation*}
\begin{split}
\vert I_1\vert+\vert I_2\vert&\lesssim (t2^{2p+q})^{-1}\cdot 2^{2p+q}\cdot\left[\Vert \widehat{f}\Vert_{L^\infty}+\Vert \partial_\rho\widehat{f}\Vert_{L^\infty}\right]
\end{split}
\end{equation*}
and we can use Lemma \ref{lem:ControlLinfty} once more.

\medskip

{\bf Case 3}:
\begin{equation*}
\begin{split}
x\ge C^{-2}t2^{p+q},\,\hbox{ and }\,\vert z\vert\ge C^{-2}t2^{2p}
\end{split}
\end{equation*}
we observe that, in these conditions, there holds that
\begin{equation*}
\begin{split}
2^q\vert \partial_\rho\Psi\vert+\vert\partial_\rho\partial_\Lambda\Psi\vert\gtrsim t
\end{split}
\end{equation*}
which follows from
\begin{equation*}
\begin{split}
\Lambda\partial_\Lambda\partial_\rho\Psi-\partial_\rho\Psi&=-2\frac{\Lambda}{\sqrt{1-\Lambda^2}}x,\qquad\partial_\Lambda\partial_\rho\Psi+\frac{\Lambda}{1-\Lambda^2}\partial_\rho\Psi=-\frac{z}{1-\Lambda^2}
\end{split}
\end{equation*}
using the first estimate if $2^{-p}\ll 1$ and the second otherwise. We can now decompose
\begin{equation*}
\begin{split}
I=\sum_{\ell\ge0}I_\ell,\qquad
I_\ell(x,z,t)&:=\int_{\mathbb{R}_+\times[-1,1]}e^{i\Psi}\chi\cdot H_\pm(\rho\sqrt{1-\Lambda^2}\vert x\vert)\cdot\ \varphi(2^{-\ell}\partial_\rho\Psi)\widehat{f}\rho^2\varphi(\rho)d\rho d\Lambda.\\
\chi&:=\varphi(2^{-p}\rho\sqrt{1-\Lambda^2})\varphi(2^{-q}\rho\Lambda).
\end{split}
\end{equation*}
On the support of $I_0$, we have that $\vert\partial_\Lambda\partial_\rho\Psi\vert\gtrsim t$, and
\begin{equation*}
\begin{split}
\vert I_0\vert&\lesssim \Vert \widehat{f}\Vert_{L^\infty}\iint \vert H^+(\rho\sqrt{1-\Lambda^2}x)\vert\cdot \varphi(g)\varphi(t^{-1}\partial_\Lambda g)\cdot\rho^2\varphi(\rho)d\rho d\Lambda,\qquad g(\rho,\Lambda)=\partial_\rho\Psi\\
&\lesssim t^{-1}\cdot (2^{2p+q}t)^{-\frac{1}{2}}\Vert \widehat{f}\Vert_{L^\infty}.
\end{split}
\end{equation*}
If we integrate by parts twice in $\rho$, we find that
\begin{equation*}
\begin{split}
I_\ell(x,z,t)&:=\int_{\mathbb{R}_+\times[-1,1]}e^{i\Psi}\frac{1}{(\partial_\rho\Psi)^2}\varphi(2^{-\ell}\partial_\rho\Psi)\partial_\rho^2\left\{\chi\cdot H_\pm(\rho\sqrt{1-\Lambda^2}x)\cdot\widehat{f}\rho^2\varphi(\rho)\right\}d\rho d\Lambda
\end{split}
\end{equation*}
and we deduce that
\begin{equation*}
\begin{split}
\vert I_\ell\vert&\lesssim 2^{-2\ell}\int_{\mathbb{R}_+\times[-1,1]}\vert \widetilde{H}(\rho\sqrt{1-\Lambda^2} x)\vert \cdot\ \varphi(2^{-\ell}\partial_\rho\Psi)\overline{\chi}F\overline{\varphi}(\rho)d\rho d\Lambda,\\
\overline{\chi}&:=\overline{\varphi}(2^{-p}\rho\sqrt{1-\Lambda^2})\overline{\varphi}(2^{-q}\rho\Lambda),\qquad \widetilde{H}(x):=\vert H^+(x)\vert+\langle x\rangle\vert \frac{dH^+}{dx}(x)\vert+\langle x\rangle^2\vert\frac{d^2H^+}{dx^2}(x)\vert,\\
F&:=\vert \widehat{f}\vert+\vert \partial_\rho\widehat{f}\vert+\vert \partial_\rho^2\widehat{f}\vert,
\end{split}
\end{equation*}
so that
\begin{equation*}
\begin{split}
\vert I_\ell\vert&\lesssim (t2^{2p+q})^{-\frac{1}{2}}\sum_{1\le q+\ell\le \ln(t)}2^{-2\ell}\Vert F\Vert_{L^\infty}\int_{\mathbb{R}_+\times[-1,1]}  \varphi(2^{-\ell}g)\varphi(t^{-1}\partial_\Lambda g)\overline{\chi}\overline{\varphi}(\rho)d\rho d\Lambda\\
&\quad+(t2^{2p+q})^{-\frac{1}{2}}\sum_{q+\ell\ge \ln(t)}2^{-2\ell}\Vert F\Vert_{L^\infty}\int_{\mathbb{R}_+\times[-1,1]}  \overline{\chi}\overline{\varphi}(\rho)d\rho d\Lambda\\
&\lesssim (t2^{2p+q})^{-\frac{1}{2}}\Vert F\Vert_{L^\infty}\left\{\sum_{1\le q+\ell\le \ln(t)}2^{-\ell} t^{-1}+\sum_{q+\ell\ge \ln(t)}2^{-2\ell}2^{2p+q}\right\}
\end{split}
\end{equation*}
and summing and using Lemma \ref{lem:ControlLinfty} finishes the proof.

\end{proof}

\section{Energy Estimates}\label{sec:energy}
Here we establish energy estimates for \eqref{eq:IER}. We first show how increments of $\norm{u}_{H^m}$, $\norm{u}_{H^{-1}}$ and $\norm{S^a u}_{L^2}$ can be controlled by $u,\nabla u$ in $L^\infty$.
\begin{proposition}\label{prop:EnergyIncrement}
Assume that $u$ solves \eqref{eq:IER} on $0\le t\le T$, then there holds that
\begin{equation}\label{EnergyIncrementEq}
\begin{split}
 \Vert u(t)\Vert_{H^{m}}^2-\Vert u(0)\Vert_{H^{m}}^2&\lesssim \int_{s=0}^tA_1(s)\cdot \Vert u(s)\Vert_{H^m}^2\cdot \frac{ds}{1+s},\\
 \Vert S^nu(t)\Vert_{L^2}^2-\Vert S^nu(0)\Vert_{L^2}^2&\lesssim \int_{s=0}^t\left[A_0(s)+A_1(s)\right]\cdot \left\{\Vert u(s)\Vert_{H^n}^2+\sum_{b=0}^n\Vert S^bu(s)\Vert_{L^2}^2\right\}\cdot \frac{ds}{1+s},\\
 \Vert \vert\nabla\vert^{-1}S^nu(t)\Vert_{L^2}^2-\Vert \vert\nabla\vert^{-1}S^nu(0)\Vert_{L^2}^2&\lesssim \int_{s=0}^t A_{0}(s)\cdot \sum_{b=0}^n\Vert S^bu(s)\Vert_{L^2}^2\cdot \frac{ds}{1+s},\\
\end{split}
\end{equation}
where
\begin{equation}
\begin{split}
 A_{0}(s)&=(1+s)\Vert u(s)\Vert_{L^\infty},\qquad
 A_1(s)=(1+s)\Vert \nabla_x u(s)\Vert_{L^\infty}.
\end{split}
\end{equation}
\end{proposition}

From this we can directly obtain the following bootstrap:
\begin{corollary}\label{cor:EE}
Assume that $U_\pm$ are solutions to \eqref{eq:IER_disp} with initial data satisfying \eqref{eq:assump_id}, with corresponding profiles $\U_\pm$, and satisfy for $0\le t\le T^{\ast}$
\begin{equation}\label{AssID_en}
\begin{aligned}
 \norm{U_\pm(t)}_{H^{2N_0}\cap H^{-1}}+\norm{S^aU_\pm(t)}_{L^2\cap H^{-1}}&\le 2\eps_1,\qquad 0\le a\le 4N,\\
 \norm{S^a\U_\pm(t)}_{B}&\le 2\eps_1,\qquad 0\le a\le 2N,\\
 \norm{S^a\U_\pm(t)}_{D}&\le 2\eps_1,\qquad 0\le a\le N.
\end{aligned}
\end{equation}
Then for $0\leq t\leq T^\ast$ there holds that
\begin{equation}\label{eq:EE_bd}
 \norm{U_\pm(t)}_{H^{2N_0}\cap H^{-1}}+\norm{S^aU_\pm(t)}_{L^2\cap H^{-1}}\le \eps_0+C\eps_1^2\ip{t}^{\frac{2}{N_0}},\qquad 0\le a\le 4N.
\end{equation}
In particular, for $T^\ast=\eps_1^{-\M}$ and $N_0=4\M$ we have if $\eps_1$ is sufficiently small that
\begin{equation}\label{eq:EE_btstrap}
 \norm{U_\pm(t)}_{H^{2N_0}\cap H^{-1}}+\norm{S^aU_\pm(t)}_{L^2\cap H^{-1}}\le \eps_1,\qquad 0\le a\le 4N.
\end{equation}
\end{corollary}

\begin{proof}[Proof of Corollary \ref{cor:EE}]
Note that since $S$ and derivatives commute with the Hodge decomposition, it suffices to control $u$ instead of $U_\pm$. Estimate \eqref{eq:EE_btstrap} follows directly from \eqref{eq:EE_bd} provided $\eps_1$ is sufficiently small.

The estimate \eqref{eq:EE_bd} follows from Proposition \ref{prop:EnergyIncrement} once we establish
\begin{equation}\label{EstimA01}
\begin{split}
 A_0(t)+A_1(t)&\lesssim  \eps_1\ip{t}^{\frac{2}{N_0}}.
\end{split}
\end{equation} 
Estimate \eqref{EstimA01} is a direct consequence of Sobolev estimates if $0\le t\le 1$. From now on, we assume $t\ge1$. We remark that
\begin{equation*}
\begin{split}
A_0+A_1\le\sum_k \sum_\pm (1+2^k)\Vert P_kU_\pm(t)\Vert_{L^\infty}.
\end{split}
\end{equation*}
Since $\norm{P_kU_\pm(t)}_{L^\infty}\lesssim 2^{-k(2N_0-2)}\norm{P_kU_\pm(t)}_{H^{N_0}}$ the bound \eqref{EstimA01} also holds if $2^k> t^{\frac{1}{N_0}}$. Now, for fixed $k$, we can estimate
\begin{equation*}
\begin{split}
\sum_{\substack{p,q,\\ 2^{q+2p}t\le 1}}\Vert P_{k,p,q}e^{it\Lambda}\mathcal{U}_\pm\Vert_{L^\infty}&\lesssim 2^{3k}\cdot \sum_{\substack{p,q,\\ 2^{q+2p}t\le 1}}2^{q+2p}\Vert \widehat{\mathcal{U}_\pm}\Vert_{L^\infty}\lesssim t^{-1} 2^{\frac{3k}{2}} \norm{\U_\pm}_{D}\lesssim \varepsilon_1 t^{-1+\frac{3}{2N_0}},
\end{split}
\end{equation*}
where we also used Lemma \ref{lem:ControlLinfty}. It remains to observe that there are $\langle\log(t)\rangle$ choices for $p,q$ such that $2^{q+2p}t\ge1$, therefore, it suffices to show that for fixed $p,q$ as above,
\begin{equation*}
\Vert P_{k,p,q}e^{it\Lambda}\mathcal{U}_\pm\Vert_{L^\infty}\lesssim \varepsilon_1 t^{-1}
\end{equation*}
which follows from Proposition \ref{DecayProp}.
\end{proof}

\begin{proof}[Proof of Proposition \ref{prop:EnergyIncrement}]
We can recast the equation as
\begin{equation*}
\begin{split}
\partial_tu^\alpha+(\vec{e}_3\times u)^\alpha+(\nabla p)^\alpha+\partial_\beta\left\{u^\alpha u^\beta\right\}=0
\end{split}
\end{equation*}
and commuting with the vector field $S$, we find that
\begin{equation*}
\begin{split}
\partial_t(Su)^\alpha+(\vec{e}_3\times Su)^\alpha+(\nabla Sp)^\alpha+\partial_\beta\left\{S\left\{u^\alpha u^\beta\right\}\right\}=0,
\end{split}
\end{equation*}
where we have extended $S$ on tensor products so that $S \left\{u^\alpha u^\beta\right\}=(Su)^\alpha u^\beta+u^\alpha (Su)^\beta$. Iterating, we obtain that
\begin{equation*}
\begin{split}
\partial_t(S^nu)^\alpha+(\vec{e}_3\times S^nu)^\alpha+(\nabla S^np)^\alpha+\partial_\beta\left\{S^n\left\{u^\alpha u^\beta\right\}\right\}=0
\end{split}
\end{equation*}
and we deduce that
\begin{equation*}
\begin{split}
\frac{d}{dt}\Vert \vert\nabla\vert^{-1}S^nu\Vert_{L^2}^2&=-\Vert \vert\nabla\vert^{-1}S^nu\Vert_{L^2}\Vert S^n\left\{u\cdot u\right\}\Vert_{L^2}
\end{split}
\end{equation*}
and using Leibniz rule and \eqref{Interpol1}, we see that
\begin{equation*}
\begin{split}
\Vert S^n\left\{ u\cdot u\right\}\Vert_{L^2}&\lesssim \Vert u\Vert_{L^\infty}\sum_{a=0}^n\Vert S^nu\Vert_{L^2}
\end{split}
\end{equation*}
and this gives the last equation in \eqref{EnergyIncrementEq}. We can also rewrite this as
\begin{equation*}
\begin{split}
0&=\partial_t (S^nu)^\alpha+(\vec{e}_3\times  S^nu)^\alpha+(\nabla S^np)^\alpha+u^\beta \partial_\beta\left\{ S^nu\right\}^\alpha
+(S^nu)^\beta\partial_\beta u^\alpha+\sum_{\substack{n_1+n_2\le n,\\
n_1,n_2\ge1}} c_{n_1,n_2}(S^{n_1}u)^\beta \cdot S^{n_2}\partial_\beta u^\alpha
\end{split}
\end{equation*}
which gives, after symmetrization
\begin{equation*}
\begin{split}
\frac{d}{dt}\Vert S^nu\Vert_{L^2}^2&\lesssim \Vert S^nu\Vert_{L^2}\Vert \nabla u\Vert_{L^\infty}\Vert S^nu\Vert_{L^2}+\Vert S^nu\Vert_{L^2}\sum_{1\le n_1\le n-1} \Vert S^{n_1}u\Vert_{L^\frac{2n}{n_1}}\Vert \nabla S^{n_1}u\Vert_{L^\frac{2n}{n-n_1}}
\end{split}
\end{equation*}
and we can use \eqref{Interpol1} again. Similarly, commuting with gradients,
\begin{equation*}
\begin{split}
0&=\partial_t(\partial_x^\mu u)^\alpha+(\vec{e}_3\times \partial_x^\mu u)^\alpha+(\nabla \partial_x^\mu p)^\alpha+u^\beta \partial_\beta\left\{\partial_x^\mu u^\alpha\right\}+(\partial_x^\mu u)^\beta\partial_{\beta}u^\alpha+\sum_{\substack{\mu_1+\mu_2=\mu,\\
\vert\mu\vert_1,\vert\mu\vert_2\ge1}}\partial_x^{\mu_1}u^\beta\cdot\partial_\beta\partial_x^{\mu_2}u^\alpha\\
\end{split}
\end{equation*}
and standard estimates give
\begin{equation*}
\begin{split}
\frac{d}{dt}\Vert \partial_x^\mu u\Vert_{L^2}^2&\le \Vert \nabla_xu\Vert_{L^\infty}\Vert u\Vert_{H^{\vert\mu\vert}}^2.
\end{split}
\end{equation*}

\end{proof}

\begin{lemma}

There holds that
\begin{equation}\label{Interpol1}
\begin{split}
\Vert S^af\Vert_{L^\frac{2n}{a}}&\lesssim \Vert f\Vert_{L^\infty}^{1-\frac{a}{n}}\left\{\sum_{a=0}^n\Vert S^af\Vert_{L^2}\right\}^\frac{a}{n}
\end{split}
\end{equation}

\end{lemma}

\begin{proof}

We define, for $0\le a\le n$
\begin{equation*}
A(a,n):=\sum_{0\le b\le a}\Vert S^bf\Vert_{L^\frac{2n}{a}}
\end{equation*}
and we first observe that \eqref{Interpol1} holds for $a=0$ or $a=n$, while for $1\le a\le n-1$,
\begin{equation}\label{Recurr}
\begin{split}
A(a,n)\le \sqrt{A(a-1,n)\cdot A(a+1,n)}.
\end{split}
\end{equation}
We now observe that a nonnegative sequence (in $a$) satisfying \eqref{Recurr} can be bounded by its boundary values:
\begin{equation}\label{ConvexityEE}
\begin{split}
A(a,n)\le A(0,n)^{1-\frac{a}{n}}A(n,n)^\frac{a}{n}.
\end{split}
\end{equation}
which finishes the proof of \eqref{Interpol1}.

It only remains to prove \eqref{Recurr} and \eqref{ConvexityEE}. For the first statement, \eqref{Recurr}, it suffices to consider the case $b=a$ and integrate by parts then use H\"older's inequality to get
\begin{equation*}
\begin{split}
\Vert S^af\Vert_{L^\frac{2n}{a}}^\frac{2n}{a}&=\lim_{\varepsilon\to0}\int_{\mathbb{R}^3}(\varepsilon+\vert S^af\vert^2)^{\frac{n}{a}-1}\cdot S^af\cdot S^afdx=\lim_{\varepsilon\to0}\int_{\mathbb{R}^3}S^{a-1}f\cdot S^\ast \left\{S^af\cdot (\varepsilon+\vert S^af\vert^2)^{\frac{n}{a}-1}\right\}dx\\
&\lesssim \Vert S^{a-1}f\Vert_{L^\frac{2n}{a-1}}(\Vert S^{a+1}f\Vert_{L^\frac{2n}{a+1}}+\Vert S^af\Vert_{L^\frac{2n}{a+1}})\Vert S^af\Vert_{L^\frac{2n}{a}}^\frac{2(n-a)}{a}.
\end{split}
\end{equation*}

Now \eqref{ConvexityEE} follows from the fact that $b_a=\log A(a,n)$ is discretely superharmonic:
\begin{equation*}
b_{a-1}-2b_a+b_{a+1}\ge0,\quad a\in\{1,n-1\}
\end{equation*}
and is upper bounded at its extremum by
\begin{equation*}
c_a=(1-\frac{a}{n})b_0+\frac{a}{n}b_n.
\end{equation*}
If the extremum of $d_a=b_a-c_a$ is attained at $1\le a\le n-1$, then we see that $d_a=d_{a-1}=d_{a+1}$ is constant.

\end{proof}

\section{On the Geometry of the Phase and Vector Fields}\label{sec:geom}

With the usual notation for $\Lambda(\xi)=\frac{\xi_3}{\abs{\xi}}$ we consider vector fields related to spherical coordinates
\begin{equation}
 \xi=(\rho\cos\theta\sin\phi,\rho\sin\theta\sin\phi,\rho\cos\phi)=(\rho\cos\theta\sqrt{1-\Lambda^2},\rho\sin\theta\sqrt{1-\Lambda^2},\rho\Lambda),
\end{equation}
with
\begin{equation}
 \Lambda=\cos\phi,\quad \sqrt{1-\Lambda^2}=\sin\phi.
\end{equation}
Then we have
\begin{equation}
\begin{split}
\partial_\rho\xi&=\rho^{-1}\xi,\qquad \partial_\theta\xi=\xi_\h^\perp,\\
\partial_\phi\xi&=(\rho\cos\theta\cos\phi,\rho\sin\theta\sin\phi,-\rho\sin\phi)=\frac{\Lambda}{\sqrt{1-\Lambda^2}}\xi_\h-\sqrt{1-\Lambda^2}\rho\,\vec{e}_3,\\
\partial_\Lambda\xi&=-\frac{\Lambda}{\sqrt{1-\Lambda^2}}(\rho\cos\theta,\rho\sin\theta,0)+(0,0,\rho).
\end{split}
\end{equation}
Hence there holds that
\begin{equation}\label{eq:vfs_coords}
 S_\xi=\xi\cdot\nabla_\xi=\rho\partial_\rho,\qquad \Omega_\xi=\xi_\h^\perp\cdot\nabla_\xi=\partial_\theta,
\end{equation}
and we define
\begin{equation}\label{eq:ups}
\begin{aligned}
 \Upsilon_\xi:=\partial_\phi=-\sqrt{1-\Lambda^2}\partial_\Lambda&=\frac{\Lambda}{\sqrt{1-\Lambda^2}}(\xi_1\partial_{\xi_1}+\xi_2\partial_{\xi_2})-\vert\xi\vert\sqrt{1-\Lambda^2}\partial_{\xi_3}\\
 &=\frac{\Lambda}{\sqrt{1-\Lambda^2}}S_{\xi,\h}-\vert\xi\vert\sqrt{1-\Lambda^2}\partial_{\xi_3}=\frac{1}{\sqrt{1-\Lambda^2}}\left[\Lambda S_\xi-\abs{\xi}\partial_{\xi_3}\right].
\end{aligned} 
\end{equation}
We thus have the approximate expressions
\begin{equation*}
\begin{split}
 \Ups_\xi&\simeq 2^{q-p}S_{\xi,\h}+2^{k+p}\partial_{\xi_3},
\end{split}
\end{equation*}
and $\Ups_\xi$ is horizontal if $\xi$ is vertical and vice-versa.

Alternatively, from the relation $\cos(\phi)=\Lambda$ we obtain that $\sin(\phi)\partial_{\xi_3}\phi=-\frac{\abs{\xi_\h}^2}{\abs{\xi}^3}$, so that $\partial_{\xi_3}\phi=-\rho^{-1}\sqrt{1-\Lambda^2}$ and we have the following simple expression for the vertical derivative
\begin{equation}\label{eq:dxi3}
 \partial_{\xi_3}=\Lambda\partial_\rho-\rho^{-1}\sqrt{1-\Lambda^2}\partial_\phi=\rho^{-1}[\Lambda S_\xi-\sqrt{1-\Lambda^2}\Upsilon_\xi].
\end{equation}
From \eqref{eq:dxi3} we can deduce additionally that
\begin{equation}
 \frac{\Lambda}{\sqrt{1-\Lambda^2}}S_{\xi,\h}=\Upsilon_\xi+\abs{\xi}\sqrt{1-\Lambda^2}\partial_{\xi_3}=\Lambda\cdot[\Lambda\Upsilon_\xi+\sqrt{1-\Lambda^2}S_\xi],
\end{equation}
so
\begin{equation}\label{eq:Sperp}
 S_{\xi,\h}=\sqrt{1-\Lambda^2}\left[\Lambda\Upsilon_\xi+\sqrt{1-\Lambda^2}S_\xi\right].
\end{equation}

\begin{notation}[For Vector Fields]
Here and in what follows we use the following notational convention for the vector fields $S,\Omega,\Ups$:
\begin{enumerate}[leftmargin=*,nosep]
 \item We shall denote by $V$ or $V'$ elements of $\{S,\Omega\}$ exclusively, $\Ups$ will always be denoted as such.
 \item When no specific variables are indicated, they act naturally on the arguments of a given function (e.g.\ $Sf(y)=y\cdot\nabla_yf(y)$); in all other cases the arguments of the vector fields shall be explicitly given, so that e.g.\
 \begin{equation}
  \Omega_{\xi-\eta}=(\xi_\h-\eta_\h)^\perp\cdot\nabla_{\eta_\h},\quad S_{\xi-\eta}=(\xi-\eta)\cdot\nabla_{\eta}.
 \end{equation}
 \item When the precise amount of vector fields $S$ does not matter, for $L\in\N$ we shall denote by $S^{\leq L}$ an expression of involving at most $L$ orders of $S$.
\end{enumerate}
\end{notation}

\subsection{Vector Fields and the Phase}
Let us now understand the interaction of the vector fields $S,\Omega,\Ups$ with the phase functions as they appear in \eqref{eq:IER_disp_Duham}. For notational convenience, here we fix the signs of our phase functions $\pm\Phi$ as\footnote{since all sign combinations occur this is no restriction -- see also \eqref{eq:def_phi}}
\begin{equation}\label{eq:def_phi2}
 \Phi(\xi,\eta)=\pm\Lambda(\xi)+\Lambda(\xi-\eta)\pm\Lambda(\eta).
\end{equation}
When the precise order of signs is of no consequence we will refer to any one of these options simply as ``the phase''.

By construction, since $S$ and $\Omega$ commute with the equation we have that
\begin{equation}
 S\Lambda=\Omega\Lambda=0.
\end{equation}
It thus follows that
\begin{equation}
  V_\eta\Phi(\xi,\eta)=V_\eta\Lambda(\xi-\eta),\qquad V\in\{S,\Omega\}.
\end{equation}
The following quantity plays a crucial role in our analysis: We let
\begin{equation}\label{eq:def_sigma}
  \bar\sigma\equiv\bar\sigma(\xi,\eta):=\xi_3\eta_\h-\eta_3\xi_\h=-(\xi\times\eta)_\h^\perp,
\end{equation}
and note that
\begin{equation}\label{eq:def_sigma2}
 \bar\sigma(\xi,\eta)=\bar\sigma(\xi-\eta,\eta)=-\bar\sigma(\xi,\xi-\eta).
\end{equation}

\begin{lemma}\label{lem:vfsizes-mini}
There holds that
 \begin{equation}\label{eq:vf_sigma_0}
  S_\eta\Phi=\bar\sigma(\xi,\eta)\cdot\frac{\xi_\h-\eta_\h}{\abs{\xi-\eta}^3},\quad \Omega_\eta\Phi=-\bar\sigma(\xi,\eta)\cdot\frac{(\xi_\h-\eta_\h)^\perp}{\abs{\xi-\eta}^3},
 \end{equation}
 and hence
 \begin{equation}\label{eq:vflobound_0}
  \abs{S_\eta\Phi}+\abs{\Omega_\eta\Phi}\sim 2^{-2k_1}2^{p_1}\abs{\bar\sigma(\xi,\eta)}.
 \end{equation}
\end{lemma}

\begin{proof}
The proof is a direct computation, using that
\begin{equation}\label{eq:Lambda_grad}
 \nabla\Lambda(\xi)=-\frac{\xi_3}{\abs{\xi}^3}\xi_\h+\frac{\abs{\xi_\h}^2}{\abs{\xi}^3}\vec{e}_3.
\end{equation}
From this we obtain that
\begin{equation}
\begin{aligned}
 S_\eta\Phi&=\eta\cdot\nabla_\eta\Lambda(\xi-\eta)=\frac{1}{\abs{\xi-\eta}^3}\left[(\xi_3-\eta_3)\eta_\h\cdot(\xi_\h-\eta_\h)-\eta_3\abs{\xi_\h-\eta_\h}^2\right]\\
 &=\frac{\xi_\h-\eta_\h}{\abs{\xi-\eta}^3}\cdot[(\xi_3-\eta_3)\eta_\h-\eta_3(\xi_\h-\eta_\h)]\\
 &=\frac{\xi_\h-\eta_\h}{\abs{\xi-\eta}^3}\cdot\bar{\sigma}(\xi,\eta),
\end{aligned} 
\end{equation}
and similarly
\begin{equation}
\begin{aligned}
 \Omega_\eta\Phi&=\eta_\h^\perp\cdot\nabla_\eta\Lambda(\xi-\eta)=\frac{1}{\abs{\xi-\eta}^3}\left[(\xi_3-\eta_3)\eta_\h^\perp\cdot(\xi_\h-\eta_\h)\right]=-\frac{(\xi_\h-\eta_\h)^\perp}{\abs{\xi-\eta}^3}\cdot[(\xi_3-\eta_3)\eta_\h]\\
 &=-\frac{(\xi_\h-\eta_\h)^\perp}{\abs{\xi-\eta}^3}\cdot\bar\sigma(\xi,\eta).
\end{aligned}
\end{equation}
Then from \eqref{eq:vf_sigma_0} we deduce the bounds \eqref{eq:vflobound_0}, since $$\max\{\abs{\sigma(\xi,\eta)\cdot(\xi_\h-\eta_\h)},\abs{\sigma(\xi,\eta)\cdot(\xi_\h-\eta_\h)^\perp}\}\geq \frac{1}{\sqrt{2}}\abs{\sigma(\xi,\eta)}\abs{\xi_\h-\eta_\h}.$$
\end{proof}

For an expanded version of this result with more details regarding also the repeated application of vector fields to the phase, we refer the reader to Lemma \ref{lem:vfsizes} in the appendix. Here we demonstrate another crucial relation between the size of the vector fields and the phase function. The following proposition shows that either we have a lower bound for $\bar\sigma$ (and by \eqref{eq:vflobound_0} thus also for $V_\eta\Phi$), or the phase is relatively large. In practice, this implies that \emph{either we can integrate by parts along a vector field $V\in\{S,\Omega\}$ or perform a normal form.} (We refer to Section \ref{sec:symbols} for the relevant symbol estimates.)
\begin{proposition}\label{prop:phasevssigma}
 Assume that $\abs{\Phi}\leq 2^{q_{\max}-2}$. Then in fact $2^{p_{\max}}\sim 1$, and $\abs{\bar\sigma}\gtrsim 2^{q_{\max}}2^{k_{\max}+k_{\min}}$.
\end{proposition}

\begin{proof}
 We distinguish cases based on the relative size of $q_\alpha:=\max\{q,q_1,q_2\}$ and $q_\beta:=\min\{q,q_1,q_2\}$.

 First, consider the case $\abs{q_\alpha-q_\beta}\leq 2$. Extending the above notation to write $\Lambda_\alpha:=\max\{\abs{\Lambda(\zeta)};\zeta\in\{\xi,\xi-\eta,\eta\}\}$ and $\Lambda_\beta:=\min\{\abs{\Lambda(\zeta)};\zeta\in\{\xi,\xi-\eta,\eta\}\}$, we claim that this implies that $\frac{\Lambda_\beta}{\Lambda_\alpha}<\frac{2}{3}$: for if on the contrary we had $\frac{\Lambda_\beta}{\Lambda_\alpha}\geq\frac{2}{3}$, then it would follow that there exist $a,b\in[\frac{2}{3},1]$ such that
 \begin{equation}
  \abs{\Phi}=\abs{\Lambda(\xi)\pm\Lambda(\xi-\eta)\pm\Lambda(\eta)}=\Lambda_\alpha\cdot \abs{1\pm a\pm b}\geq \frac{1}{3}\Lambda_\alpha,
 \end{equation}
 a contradiction to our assumption that $\abs{\Phi}\leq 2^{q_\alpha-2}$. To conclude, we note that $\frac{\Lambda_\beta}{\Lambda_\alpha}<\frac{2}{3}$ implies that $\Lambda_\alpha\sqrt{1-\Lambda_\beta^2}>\frac{3}{2}\Lambda_\beta\sqrt{1-\Lambda_\beta^2}\geq\frac{3}{2}\Lambda_\beta\sqrt{1-\Lambda_\alpha^2}$, and hence, using \eqref{eq:def_sigma}-\eqref{eq:def_sigma2}, $\abs{\bar\sigma}\gtrsim 2^{q_\alpha+p_\beta}2^{k_\alpha+k_\beta}$, as claimed.
 
 In case $\abs{q_\alpha-q_\beta}>2$, it is clear from its definition that $\abs{\bar\sigma}\gtrsim 2^{q_\alpha+p_\beta}2^{k_{\max}+k_{\min}}$. Moreover, it follows that $2^{p_\beta}\sim \sqrt{1-2^{2q_{\beta}}}\gtrsim 1$, which gives the claim since $p_{\max}=p_\beta$.
 \end{proof}

Next we see that the action of $\Ups$ on the phase is easily understood.
\begin{lemma}[$\Ups$ and the Phase]
 We have that
 \begin{equation}\label{eq:UpsPhi_bd}
  \abs{\Ups_\xi\Phi}\lesssim 2^p+2^{k-k_1}2^{p_1}.
 \end{equation}
\end{lemma}
\begin{proof}
We compute that
\begin{equation}
 \Upsilon_\xi\Lambda(\xi-\eta)=-\frac{\vert\xi\vert}{\vert\xi-\eta\vert}\sqrt{1-\Lambda^2(\xi-\eta)}\left[\Lambda(\xi)\Lambda(\xi-\eta)\frac{\xi_\h\cdot(\xi-\eta)_\h}{\vert\xi_\h\vert\vert(\xi-\eta)_\h\vert}+\sqrt{1-\Lambda^2(\xi)}\sqrt{1-\Lambda^2(\xi-\eta)}\right],
\end{equation}
so that with $\Upsilon_\xi\Lambda(\xi)=-\sqrt{1-\Lambda^2}(\xi)$ it follows that for a phase $\Phi=\Lambda(\xi)\pm\Lambda(\xi-\eta)\pm\Lambda(\eta)$ we get 
\begin{equation}\label{eq:UpsPhi}
\begin{aligned}
 &\Upsilon_\xi\Phi=-\sqrt{1-\Lambda^2(\xi)}\mp\frac{\vert\xi\vert}{\vert\xi-\eta\vert}\sqrt{1-\Lambda^2(\xi-\eta)}\left[\Lambda(\xi)\Lambda(\xi-\eta)\frac{\xi_\h\cdot(\xi-\eta)_\h}{\vert\xi_\h\vert\vert(\xi-\eta)_\h\vert}+\sqrt{1-\Lambda^2(\xi)}\sqrt{1-\Lambda^2(\xi-\eta)}\right],
\end{aligned} 
\end{equation}
from which the bound \eqref{eq:UpsPhi_bd} follows.
\end{proof}

\subsection{Vector Field ``Cross Terms''}\label{ssec:crossterms}
When integrating by parts along vector fields, we will encounter ``cross terms'' where a vector field $V\in\{S,\Omega\}$ in one variable is applied to a function of another variable. The following lemma shows how such expressions can be resolved in terms of the three vector fields $S,\Omega,\Ups$. We state it here for the axisymmetric setting, as is relevant for our purposes.

Let us introduce the notation
\begin{equation}
 \omega_c:=\frac{\eta_\h\cdot(\xi_\h-\eta_\h)}{\abs{\eta_\h}\abs{\xi_\h-\eta_\h}},\quad \omega_s:=\frac{\eta_\h\cdot(\xi_\h-\eta_\h)^\perp}{\abs{\eta_\h}\abs{\xi_\h-\eta_\h}}.
\end{equation}

\begin{lemma}\label{lem:VFcross}
 On \emph{axisymmetric functions}, we have that
 \begin{equation}
 \begin{aligned}
  S_{\eta}&= \frac{\abs{\eta}}{\abs{\xi-\eta}}\left[\omega_c\sqrt{1-\Lambda^2}(\eta)\sqrt{1-\Lambda^2}(\xi-\eta)-\Lambda(\xi-\eta)\Lambda(\eta)\right]S_{\xi-\eta}\\
  &\qquad+\frac{\abs{\eta}}{\abs{\xi-\eta}}\left[\omega_c\sqrt{1-\Lambda^2}(\eta)\Lambda(\xi-\eta)+\Lambda(\eta)\sqrt{1-\Lambda^2(\xi-\eta)}\right]\Upsilon_{\xi-\eta},\\
  \Omega_{\eta}&=-\frac{\abs{\eta}}{\abs{\xi-\eta}}\omega_s\cdot\sqrt{1-\Lambda^2(\eta)}\left[\Lambda(\xi-\eta)\Upsilon_{\xi-\eta}+\sqrt{1-\Lambda^2(\xi-\eta)}S_{\xi-\eta}\right].
 \end{aligned} 
 \end{equation}
 and symmetrically for the roles of $\eta$ and $\xi-\eta$ exchanged. 
 
 Moreover, there holds that
 \begin{equation}\label{eq:Ups_cross}
 \begin{aligned}
  \Upsilon_\xi&=\frac{\abs{\xi}}{\abs{\xi-\eta}}\left[\omega_c\Lambda(\xi)\sqrt{1-\Lambda^2(\xi-\eta)}-\sqrt{1-\Lambda^2(\xi)}\Lambda(\xi-\eta)\right]\cdot S_{\xi-\eta}\\
  &\qquad+\frac{\abs{\xi}}{\abs{\xi-\eta}}\left[\omega_c\Lambda(\xi)\Lambda(\xi-\eta)+\sqrt{1-\Lambda^2(\xi)}\sqrt{1-\Lambda^2(\xi-\eta)}\right]\cdot \Upsilon_{\xi-\eta}.
 \end{aligned} 
 \end{equation}
\end{lemma}

\begin{remark}
With the usual localizations we thus have the bounds
\begin{equation}\label{eq:crossvf_bd}
\begin{aligned}
 \abs{S_\eta f(\xi-\eta)}&\lesssim 2^{k_2-k_1}[(2^{p_2+p_1}+2^{q_2+q_1})\abs{Sf(\xi-\eta)}+(2^{p_2+q_1}+2^{p_1+q_2})\abs{\Ups f(\xi-\eta)}],\\
 \abs{\Omega_\eta f(\xi-\eta)}&\lesssim 2^{k_2-k_1}[2^{p_2+p_1}\abs{Sf(\xi-\eta)}+2^{p_2+q_1}\abs{\Ups f(\xi-\eta)}],
\end{aligned} 
\end{equation}
\end{remark}

\begin{proof}
We give next the full details, from which the case of axisymmetric functions follows by noting that if $f$ is axisymmetric, then $\Omega f=0$ (see also Remark \ref{rem:role_axisym}).

Then we have that by \eqref{eq:dxi3} and \eqref{eq:vfbasics} that
\begin{equation}
\begin{aligned}
 S_{\eta}&=\frac{\abs{\eta_\h}}{\abs{\xi_\h-\eta_\h}}\omega_s\cdot\Omega_{\xi-\eta}+\frac{\abs{\eta_\h}}{\abs{\xi_\h-\eta_\h}}\omega_c\cdot S_{\xi-\eta,\h}-\eta_3\partial_{\xi_3-\eta_3}\\
 &=\frac{\abs{\eta_\h}}{\abs{\xi_\h-\eta_\h}}\omega_s\cdot\Omega_{\xi-\eta}+\frac{\abs{\eta_\h}}{\abs{\xi_\h-\eta_\h}}\omega_c\cdot S_{\xi-\eta}\\
 &\qquad -\left[\frac{\abs{\eta_\h}}{\abs{\xi_\h-\eta_\h}}\omega_c\cdot(\xi_3-\eta_3)+\eta_3\right]\cdot\frac{1}{\abs{\xi-\eta}}[\Lambda(\xi-\eta) S_{\xi-\eta}-\sqrt{1-\Lambda^2(\xi-\eta)}\Upsilon_{\xi-\eta}]\\
 &=\frac{\abs{\eta_\h}}{\abs{\xi_\h-\eta_\h}}\omega_s\cdot\Omega_{\xi-\eta}+\frac{\abs{\eta}}{\abs{\xi-\eta}}\left[\omega_c\sqrt{1-\Lambda^2}(\eta)\sqrt{1-\Lambda^2}(\xi-\eta)-\Lambda(\xi-\eta)\Lambda(\eta)\right]S_{\xi-\eta}\\
 &\qquad +\frac{\abs{\eta}}{\abs{\xi-\eta}}\left[\omega_c\sqrt{1-\Lambda^2}(\eta)\Lambda(\xi-\eta)+\Lambda(\eta)\sqrt{1-\Lambda^2(\xi-\eta)}\right]\Upsilon_{\xi-\eta}.
\end{aligned} 
\end{equation}
Similarly, by \eqref{eq:Sperp} there holds that
\begin{equation}
\begin{aligned}
 \Omega_{\eta}&=\frac{\abs{\eta_\h}}{\abs{\xi_\h-\eta_\h}}\omega_c\cdot\Omega_{\xi-\eta}-\frac{\abs{\eta_\h}}{\abs{\xi_\h-\eta_\h}}\omega_s\cdot S_{\xi-\eta,\h}\\
 &=\frac{\abs{\eta_\h}}{\abs{\xi_\h-\eta_\h}}\omega_c\cdot\Omega_{\xi-\eta}-\frac{\abs{\eta_\h}}{\abs{\xi_\h-\eta_\h}}\omega_s\cdot\sqrt{1-\Lambda^2(\xi-\eta)}\left[\Lambda(\xi-\eta)\Upsilon_{\xi-\eta}+\sqrt{1-\Lambda^2(\xi-\eta)}S_{\xi-\eta}\right]\\
 &=\frac{\abs{\eta_\h}}{\abs{\xi_\h-\eta_\h}}\omega_c\cdot\Omega_{\xi-\eta}-\frac{\abs{\eta}}{\abs{\xi-\eta}}\omega_s\cdot\sqrt{1-\Lambda^2(\eta)}\left[\Lambda(\xi-\eta)\Upsilon_{\xi-\eta}+\sqrt{1-\Lambda^2(\xi-\eta)}S_{\xi-\eta}\right].
\end{aligned} 
\end{equation}

Moreover, there holds
\begin{equation}
\begin{aligned}
 \Upsilon_\xi&=\frac{\Lambda(\xi)\abs{\xi}}{\abs{\xi_\h-\eta_\h}}\omega_s\Omega_{\xi-\eta}+\frac{\abs{\xi}}{\abs{\xi-\eta}}\left[\omega_c\Lambda(\xi)\sqrt{1-\Lambda^2(\xi-\eta)}-\sqrt{1-\Lambda^2(\xi)}\Lambda(\xi-\eta)\right]\cdot S_{\xi-\eta}\\
 &\qquad+\frac{\abs{\xi}}{\abs{\xi-\eta}}\left[\omega_c\Lambda(\xi)\Lambda(\xi-\eta)+\sqrt{1-\Lambda^2(\xi)}\sqrt{1-\Lambda^2(\xi-\eta)}\right]\cdot \Upsilon_{\xi-\eta}.
\end{aligned} 
\end{equation}
\end{proof}

\section{A Bilinear Estimate}\label{sec:bilin}
Here our goal is to prove suitable bilinear estimates that will allow us to give control of the vector fields $S,\Omega$ in the $B$ norm.

\begin{proposition}\label{prop:bilin}
 Consider a bilinear term $\Q_\m(f,g)$ (as in \eqref{eq:def_Qm})
 \begin{equation}
  \widehat{\Q_\m(f,g)}(\xi)=\int_{\mathbb{R}^3}e^{it\Phi(\xi,\eta)}\m(\xi,\eta)\widehat{f}(\xi-\eta)\widehat{g}(\eta)d\eta,
 \end{equation}
 where $\Phi(\xi,\eta)=\pm\Lambda(\xi)\pm\Lambda(\xi-\eta)\pm\Lambda(\eta)$ is one of the phase functions, and $\m=m_i^{\mu\nu}$ is one of the bilinear symbols of axisymmetric, rotating 3d Euler as in Lemma \ref{lem:IERmult}.
 Let $k,k_1,k_2\in\mathbb{Z}$ and $p,p_1,p_2,q, q_1,q_2\le 0$, with
 \begin{equation*}
 f=P_{k_1,p_1,q_1}f,\qquad g=P_{k_2,p_2,q_2}g.
 \end{equation*}
 Then we have the following estimates:
 \begin{enumerate}[label=(\alph*)]
  \item\label{it:bilin1} If $p_2+\frac{q_2}{2}\leq p+\frac{q}{2}+C$, we get
  \begin{equation}\label{eq:bilin1}
   2^{-p-\frac{q}{2}}\norm{P_{k,p,q} \Q_\m(f,g)}_{L^2}\lesssim \ip{t}^{-1}2^{k+\frac{3}{2}k_{\max}}\norm{f}_D\norm{g}_B.
  \end{equation}
  \item\label{it:bilin2} If $p_2+\frac{q_2}{2}> p+\frac{q}{2}+C$, we have
 \begin{equation}\label{eq:bilin2}
 \begin{aligned}
   2^{-p-\frac{q}{2}}\norm{P_{k,p,q} \Q_\m(f,g)}_{L^2}&\lesssim \ip{t}^{-1} 2^{\frac{k}{2}+2k_{\max}}\Big[\norm{f}_B+\norm{Sf}_B+\norm{\Ups f}_{L^2}\Big]\\
   &\qquad\cdot\Big[\norm{g}_{L^2}+2^{k_1}\norm{g}_{H^{-1}}+\norm{Sg}_{L^2}\Big],
  \end{aligned}
  \end{equation}
  where $k_{\max}=\max\{k,k_1,k_2\}$.
 \end{enumerate}
\end{proposition}

While \ref{it:bilin1} just follows from the linear decay estimates, \ref{it:bilin2} needs more elaboration, and the following subsection is devoted to its proof. We refer also to computations regarding the phase function (Appendix \ref{apdx:phase-comp}) and an improved set size estimate for the bilinear terms (Appendix \ref{apdx:set_gain}).

Overall it is worth noting that the arguments rely on set size estimates combined with a \emph{single} integration by parts with $\Omega$ or $S$ and the structure of the multipliers $\m$, but \emph{do not make use of normal forms}.

\subsection{Proof of Proposition \ref{prop:bilin}}\label{ssec:bilin-proof}
As announced, the proof of the first assertion \ref{it:bilin1} is a direct consequence of the decay estimates:
\begin{proof}[Proof of \ref{it:bilin1}]
 We recall from Proposition \ref{DecayProp} that
 \begin{equation}
  \norm{e^{it\Lambda}P_{k_1,p_1,q_1}f}_{L^\infty}\lesssim t^{-1}2^{\frac{3}{2}k_1}\norm{f}_D,
 \end{equation}
 hence, if $p_2+\frac{q_2}{2}\leq p+\frac{q}{2}+C$, an $L^\infty\times L^2$ estimate suffices to conclude:
 \begin{equation}
 \begin{aligned}
  2^{-p-\frac{q}{2}}\norm{\Q_\m(P_{k_1,p_1,q_1}f,P_{k_2,p_2,q_2}g)}_{L^2}&\lesssim t^{-1}2^{\frac{3}{2}k_1}2^k\norm{f}_D 2^{-p-\frac{q}{2}}2^{p_2+\frac{q_2}{2}}\norm{g}_{B}\\
  &\lesssim t^{-1}2^{\frac{3}{2}k_1+k}\norm{f}_D\norm{g}_B.
 \end{aligned}
 \end{equation}
\end{proof}

It remains to prove assertion \ref{it:bilin2} of Proposition \ref{prop:bilin}. The basic strategy is as follows: Using a good estimate for the size of the vector fields (see \eqref{eq:vf_sigma} from Lemma \ref{lem:vfsizes}) enables us to perform one integration by parts in either $\Omega$ or $S$, which will transform $\Q_\m(f,g)$ into more favorable terms, which can then be estimates using Lemma \ref{lem:set_gain}.

To prepare the ground for this, let us recall from Lemma \ref{lem:vfsizes-mini} that\footnote{The crucial insight here will be that up to the sizes $2^k,2^{k_1},2^{k_2}$, we have $\abs{\bar\sigma}\sim_{k,k_1,k_2}2^{\max\{p,p_1,p_2\}+\max\{q,q_1,q_2\}}$, which allows for an integration by parts that can be partially compensated for by the size of the multipliers $\m$ -- compare \eqref{eq:simpmultbd}.}
\begin{equation}\label{eq:vf_lowbd}
\begin{split}
  \abs{S_\eta\Phi}+\abs{\Omega_\eta\Phi} \sim 2^{-2k_1+p_1}\abs{\bar\sigma(\xi,\eta)},\qquad
  \bar{\sigma}(\xi,\eta)=\xi_3\eta_\h-\eta_3\xi_\h=-\bar{\sigma}(\xi,\xi-\eta)=-\bar{\sigma}(\eta,\xi).
 \end{split}
\end{equation}
Moreover, if $p_2+\frac{q_2}{2}> p+\frac{q}{2}+C$, we recall a special case of \eqref{eq:vfquotient_apdx}, namely that whenever $V\in\{S,\Omega\}$ with $\abs{V_\eta\Phi}\sim2^{-2k_1+p_1}\abs{\bar\sigma(\xi,\eta)}$, then we also have
\begin{equation}\label{eq:vfquotient}
 \frac{\abs{V_\eta^2\Phi}}{\abs{V_\eta\Phi}}\lesssim 1+2^{k_2-k_1}+2^{p-p_1}[1+2^{k-k_1}],\qquad V\in\{S,\Omega\}.
\end{equation}

\begin{proof}[Proof of \ref{it:bilin2}]
For simplicity of notation let us abbreviate $I_\m(\xi,t):=P_{k,p,q} \Q_\m(P_{k_1,p_1,q_1}f,P_{k_2,p_2,q_2}g)$. In view of an integration by parts with one of the vector fields, for the scale $L:=2^{-2k_1+p_1}\abs{\bar\sigma(\xi,\eta)}$ we split our integral as
\begin{equation}\label{eq:ibp_split-m}
\begin{aligned}
  I_\m(\xi,t)&=I_\m^1(\xi,t)+I_\m^2(\xi,t),\\
  I_\m^j(\xi,t)&:=\varphi_{k,p,q}(\xi)\int_{\mathbb{R}^3}e^{it\Phi}\varphi_{k_1,p_1,q_1}(\xi-\eta)\varphi_{k_2,p_2,q_2}(\eta)\chi^j(\xi,\eta)\m(\xi,\eta)\widehat{f}(\xi-\eta)\widehat{g}(\eta)d\eta,\quad 1\leq j\leq 2,\\
  &\quad \text{with }\chi^1(\xi,\eta):=(1-\varphi(L^{-1}\Omega_\eta\Phi)),\quad\chi^2(\xi,\eta):=\varphi(L^{-1}\Omega_\eta\Phi).
\end{aligned}
\end{equation}
Hence on the support of $I_\m^j$ we can integrate by parts in $V_\eta=\Omega_\eta$ if $j=1$, or $V_\eta=S_\eta$ if $j=2$, which yields
\begin{equation}
\begin{aligned}
  I_\m^j&=\frac{i}{t}\varphi_{k,p,q}(\xi)\int_{\mathbb{R}^3}e^{it\Phi}V_\eta\left\{\frac{\m(\xi,\eta)}{V_\eta\Phi}\varphi_{k_1,p_1,q_1}(\xi-\eta)\varphi_{k_2,p_2,q_2}(\eta)\chi^j(\xi,\eta)\widehat{f}(\xi-\eta)\widehat{g}(\eta)\right\}d\eta\\
  &=:it^{-1}(I_{\m,1}^j+I_{\m,2}^j+I_{\m,3}^j),
\end{aligned}
\end{equation}
where
\begin{align}
  I_{\m,1}^j&=\varphi_{k,p,q}(\xi)\int_{\mathbb{R}^3}e^{it\Phi}V_\eta\left\{\frac{\m(\xi,\eta)}{V_\eta\Phi}\varphi_{k_1,p_1,q_1}(\xi-\eta)\varphi_{k_2,p_2,q_2}(\eta)\chi^j(\xi,\eta)\right\}\widehat{f}(\xi-\eta)\widehat{g}(\eta)d\eta,\label{eq:ibp_mult}\\
  I_{\m,2}^j&=\varphi_{k,p,q}(\xi)\int_{\mathbb{R}^3}e^{it\Phi}\frac{\m(\xi,\eta)}{V_\eta\Phi}\varphi_{k_1,p_1,q_1}(\xi-\eta)\varphi_{k_2,p_2,q_2}(\eta)\chi^j(\xi,\eta)V_\eta\left\{\widehat{f}(\xi-\eta)\right\}\widehat{g}(\eta)d\eta,\label{eq:ibp_cross}\\
  I_{\m,3}^j&=\varphi_{k,p,q}(\xi)\int_{\mathbb{R}^3}e^{it\Phi}\frac{\m(\xi,\eta)}{V_\eta\Phi}\varphi_{k_1,p_1,q_1}(\xi-\eta)\varphi_{k_2,p_2,q_2}(\eta)\chi^j(\xi,\eta)\widehat{f}(\xi-\eta)V_\eta\widehat{g}(\eta)d\eta.\label{eq:ibp_purevf}
\end{align}
We remark on a some useful estimates for these terms separately:
\begin{itemize}[leftmargin=*]
\item Pure Multipliers \eqref{eq:ibp_mult}: Abbreviating $\varphi_l=\varphi_{k_l;p_l,q_l}$, $l\in\{1,2\}$, we have
\begin{equation}
\begin{aligned}
 V_\eta\left\{\frac{\m(\xi,\eta)}{V_\eta\Phi}\varphi_1(\xi-\eta)\varphi_2(\eta)\chi^j(\xi,\eta)\right\}&=\frac{V_\eta\m}{V_\eta\Phi}\varphi_1\varphi_2\chi^j+\frac{V_\eta^2\Phi}{V_\eta\Phi}\frac{\m}{V_\eta\Phi}\varphi_1\varphi_2\chi^j+\frac{\m}{V_\eta\Phi}V_\eta(\varphi_1\varphi_2\chi^j).
\end{aligned} 
\end{equation}
The bounds \eqref{eq:vfonmult} in Lemma \ref{lem:vfmult} imply that $\abs{V_\eta\m}\lesssim (1+2^{p+k-p_1-k_1})\cdot C_\m +2^{k-p_1}\abs{V_\eta\Phi}$ for $C_\m=2^{k+p_{\max}+q_{\max}}$, so we have
\begin{equation}
 \abs{\frac{V_\eta\m}{V_\eta\Phi}\varphi_1\varphi_2\chi^j}
 \lesssim (1+2^{p+k-p_1-k_1})\frac{C_\m}{\abs{V_\eta\Phi}}+2^{k-p_1}.
\end{equation}
Similarly, from \eqref{eq:vfquotient} we obtain
\begin{equation}
 \abs{\frac{V_\eta^2\Phi}{V_\eta\Phi}\frac{\m}{V_\eta\Phi}\varphi_1\varphi_2\chi^j}\lesssim \{1+2^{k_2-k_1}+2^{p-p_1}[1+2^{k-k_1}]\}\frac{\abs{\m}}{\abs{V_\eta\Phi}},
\end{equation}
and by direct computation we have
\begin{equation}
 \abs{V_\eta(\varphi_1\varphi_2\chi^j)}\lesssim 1+\min\{2^{p_2+k_2}2^{-p_1-k_1},1+2^{p-p_1}2^{k-k_1}\}+\min\{2^{q_2+k_2}2^{-q_1-k_1},1+2^{q-q_1}2^{k-k_1}\}.
\end{equation}

In summary, we thus have
\begin{equation}\label{eq:bd_puremult}
\begin{aligned}
 &\abs{V_\eta\left\{\frac{\m(\xi,\eta)}{V_\eta\Phi}\varphi_{k_1,p_1,q_1}(\xi-\eta)\varphi_{k_2,p_2,q_2}(\eta)\chi^j(\xi,\eta)\right\}}\lesssim 2^{k-p_1}\\
 &\qquad\qquad +\left\{1+2^{k_2-k_1}+2^{p-p_1}[1+2^{k-k_1}]+\min\{2^{q_2+k_2}2^{-q_1-k_1},2^{q-q_1}2^{k-k_1}\}\right\}\frac{C_\m}{\abs{V_\eta\Phi}}.
\end{aligned} 
\end{equation}

\item Cross terms \eqref{eq:ibp_cross}: Here we recall from \eqref{eq:crossvf_bd} that for vector field ``cross terms'' with the above localizations we have
\begin{equation}\label{eq:vfcross_bd''}
 \abs{V_\eta f(\xi-\eta)}\lesssim 2^{k_2-k_1}\left(\abs{Sf(\xi-\eta)}+(2^{q_1+p_2}+2^{q_2+p_1})\abs{\Ups f(\xi-\eta)}\right), \quad V\in\{S,\Omega\}.
\end{equation}
Moreover, in view of future application where we need to recover $2^{p+p_1}$, we note that by the triangle inequality there holds that $2^{p_2+k_2}\lesssim 2^{p+k}+2^{p_1+k_1}$, so that in \eqref{eq:ibp_cross} we can estimate
\begin{equation}\label{eq:vfcross_bd'}
 \norm{V_\eta\left\{\widehat{f}(\xi-\eta)\right\}}_{L^2}\lesssim 2^{k_2-k_1}\norm{Sf}_{L^2}+\left(2^{q_1}\min\{2^{p_2+k_2-k_1},2^{p+k-k_1}+2^{p_1}\}+2^{q_2}2^{p_1+k_2-k_1}\right)\norm{\Ups f}_{L^2}.
\end{equation}

\end{itemize}

The rest of the proof consist in distinguishing cases and applying these estimates, together with Lemma \ref{lem:set_gain} for improved set size gain. Without loss of generality we can assume that either $p\ll 0$ or $q\ll 0$, else we are done as in scenario \ref{it:bilin1}.

\subsubsection*{Case $p\ll 0$.} Here we have $2^q\sim 1$, and we may assume without loss of generality that $p<\max\{p,p_1,p_2\}$, so that by \eqref{eq:simpmultbd} there holds
\begin{equation}
 C_\m\lesssim 2^{k+\max\{p_1,p_2\}}
\end{equation}
(else, i.e.\ if $p=\max\{p,p_1,p_2\}$ we are done with an $L^\infty\times L^2$ estimate as in part \ref{it:bilin1}).
\begin{enumerate}[leftmargin=*]
 \item Case $2^{p_1}\lesssim 2^{p_2}$: By assumption of case \ref{it:bilin2} we have $p<-C+p_2$, and hence $2^{p+q_2}\ll 2^{p_2+q}$, which implies that $\abs{\bar{\sigma}}\sim 2^{k+k_2+p_2}$. It follows that $\abs{V_\eta\Phi}\sim 2^{k+k_2-2k_1}2^{p_1+p_2}$, and hence
 \begin{equation}
  \frac{C_\m}{\abs{V_\eta\Phi}}\lesssim 2^{2k_1-k_2}2^{-p_1}.
 \end{equation}
 Together with the bounds \eqref{eq:bd_puremult} and the fact that $2^{q_2}\lesssim 2^{q_1}$ we invoke Lemma \ref{lem:set_gain} to obtain that
 \begin{equation}
 \begin{aligned}
  2^{-p}\norm{I_{m,1}^j}_{L^2}&\lesssim 2^{-p}\left[2^{2k_1-k_2}2^{-p_1}\left\{1+2^{k_2-k_1}+2^{p-p_1}[1+2^{k-k_1}]\right\}+2^{k-p_1}\right]\\
  &\qquad \cdot \min_\alpha\{2^{p_\alpha+k_\alpha}\}\min_\alpha\{2^{\frac{q_\alpha+k_\alpha}{2}}\} \cdot\norm{f}_{L^2}\norm{g}_{L^2}\\
  &\lesssim  \left\{2^{-p-p_1}[2^{2k_1-k_2}+2^{k_1}+2^{k}]+2^{-2p_1}[2^{2k_1-k_2}+2^{k+k_1-k_2}]\right\}\\
  &\qquad \cdot \min_\alpha\{2^{p_\alpha+k_\alpha}\}\min_\alpha\{2^{\frac{q_\alpha+k_\alpha}{2}}\} \cdot\norm{f}_{L^2}\norm{g}_{L^2}\\
  &\lesssim \left(2^k\cdot [2^{2k_1-k_2}+2^{k_1}+2^k]+2^{k_1}\cdot[2^{2k_1-k_2}+2^{k+k_1-k_2}]\right)\\
  &\qquad \cdot 2^{-p_1}\cdot\min_\alpha\{2^{\frac{q_\alpha+k_\alpha}{2}}\}\norm{f}_{L^2}\norm{g}_{L^2}.
 \end{aligned}
 \end{equation}

 Similarly, from \eqref{eq:vfcross_bd'} we deduce that
 \begin{equation}\label{eq:crossbd1}
 \begin{aligned}
  2^{-p}\norm{I_{m,2}^j}_{L^2}&\lesssim 2^{-p}\cdot 2^{2k_1-k_2}2^{-p_1}\cdot \Big[2^{k_2-k_1}\norm{Sf}_{L^2}+[2^{p+k-k_1}+2^{p_1}(1+2^{k_2-k_1})]\norm{\Ups f}_{L^2}\Big]\\
  &\qquad\cdot \min\{2^{p_\alpha+k_\alpha}\}\min\{2^{\frac{q_\alpha+k_\alpha}{2}}\}\norm{g}_{L^2}\\
  &\lesssim 2^{k+k_1-p_1}\min\{2^{\frac{q_\alpha+k_\alpha}{2}}\}\norm{Sf}_{L^2}\norm{g}_{L^2}+ (2^{k+2k_1-k_2}+2^{k+k_1})\min\{2^{\frac{q_\alpha+k_\alpha}{2}}\}\norm{\Ups f}_{L^2}\norm{g}_{L^2}.
 \end{aligned}
 \end{equation}

 Lastly, the simplest term is
 \begin{equation}
 \begin{aligned}
  2^{-p}\norm{I_{m,3}^j}_{L^2}&\lesssim 2^{-p}\cdot 2^{2k_1-k_2}2^{-p_1}\cdot \min\{2^{p_\alpha+k_\alpha}\}\min\{2^{\frac{q_\alpha+k_\alpha}{2}}\} \cdot\norm{f}_{L^2}\norm{Sg}_{L^2}\\
  &\lesssim 2^{2k_1+k-k_2}\cdot 2^{-p_1}\cdot\min\{2^{\frac{q_\alpha+k_\alpha}{2}}\}\norm{f}_{L^2}\norm{Sg}_{L^2}.
 \end{aligned}
 \end{equation}

 \item Case $p\ll p_2\ll p_1$: Since here $q_1\leq q_2$ it follows that $\abs{\bar\sigma}\sim 2^{k_1+k_2+p_1}$, and $2^{k_1}\ll2^k\sim 2^{k_2}$ (since also $\abs{\bar\sigma}\sim 2^{k+k_2}2^{p_2}$), hence
 \begin{equation}
  \frac{C_\m}{\abs{V_\eta\Phi}}\lesssim 2^{k_1}2^{-p_1}.
 \end{equation}
 As above it then follows that
 \begin{equation}
 \begin{aligned}
  2^{-p}\norm{I_{m,1}^j}_{L^2}&\lesssim 2^{-p}\left[2^{k_1}2^{-p_1}\left\{1+2^{k_2-k_1}+2^{p-p_1}[1+2^{k-k_1}]+2^{q_2-q_1}2^{k_2-k_1}\right\}+2^{k-p_1}\right]\\
  &\qquad \cdot \min\{2^{p_\alpha+k_\alpha}\}\min\{2^{\frac{q_\alpha+k_\alpha}{2}}\} \cdot\norm{f}_{L^2}\norm{g}_{L^2}\\
  &\lesssim  2^k2^{-p-p_1}2^{q_2-q_1} \cdot \min\{2^{p_\alpha+k_\alpha}\}\min\{2^{\frac{q_\alpha+k_\alpha}{2}}\} \cdot\norm{f}_{L^2}\norm{g}_{L^2}\\
  &\lesssim 2^{2k+\frac{k_1}{2}}\cdot 2^{-p_1-\frac{q_1}{2}}\norm{f}_{L^2}\norm{g}_{L^2}
 \end{aligned}
 \end{equation}
 as well as by \eqref{eq:vfcross_bd'} and $p_2\ll p_1$ that
 \begin{equation}\label{eq:crossbd2}
 \begin{aligned}
  2^{-p}\norm{I_{m,2}^j}_{L^2}&\lesssim 2^{-p}\cdot 2^{k_1}2^{-p_1}\left[2^{k_2-k_1}\norm{Sf}_{L^2}+\left(2^{p_2+k_2-k_1}+2^{k_2-k_1}2^{p_1}\right)\norm{\Ups f}_{L^2}\right] \\
  &\qquad\cdot \min\{2^{p_\alpha+k_\alpha}\}\min\{2^{\frac{q_\alpha+k_\alpha}{2}}\} \cdot\norm{g}_{L^2}\\
  &\lesssim 2^{-p}\cdot 2^k [2^{-p_1}\norm{Sf}_{L^2}+\norm{\Ups f}_{L^2}]\min\{2^{p_\alpha+k_\alpha}\} \cdot\min\{2^{\frac{q_\alpha+k_\alpha}{2}}\}\norm{g}_{L^2}\\
  &\lesssim 2^{2k}\min\{2^{\frac{q_\alpha+k_\alpha}{2}}\}\cdot [2^{-p_1}\norm{Sf}_{L^2}+\norm{\Ups f}_{L^2}]\norm{g}_{L^2},
 \end{aligned}
 \end{equation}
 and
 \begin{equation}
 \begin{aligned}
  2^{-p}\norm{I_{m,3}^j}_{L^2}&\lesssim 2^{-p}\cdot 2^{k_1}2^{-p_1}\cdot \min\{2^{p_\alpha+k_\alpha}\}\min\{2^{\frac{q_\alpha+k_\alpha}{2}}\} \cdot\norm{f}_{L^2}\norm{Sg}_{L^2}\\
  &\lesssim 2^{k_1+k}\cdot 2^{-p_1}\cdot\min\{2^{\frac{q_\alpha+k_\alpha}{2}}\}\norm{f}_{L^2}\norm{Sg}_{L^2}.
 \end{aligned}
 \end{equation}
\end{enumerate}

\subsubsection*{Case $q\ll 0$.} Here we have $2^p\sim 1$, and we may assume without loss of generality that $q<\max\{q,q_1,q_2\}$, so that by \eqref{eq:simpmultbd} there holds
\begin{equation}
 C_\m\lesssim 2^{k+\max\{q_1,q_2\}}
\end{equation}
(else, i.e.\ if $q=\max\{q,q_1,q_2\}$ we are done with an $L^\infty\times L^2$ estimate as in part \ref{it:bilin1}). The rest of the estimates proceeds in close analogy to the previous cases.
\begin{enumerate}[leftmargin=*]
 \item Case $2^{q_1}\lesssim 2^{q_2}$: By assumption of case \ref{it:bilin2} we have $q<-2C+q_2$, and hence $2^{p_2+q}\ll 2^{p+q_2}$, which implies that $\abs{\bar{\sigma}}\sim 2^{k+k_2+q_2}$. It follows by \eqref{eq:vf_lowbd} that $\abs{V_\eta\Phi}\sim 2^{k+k_2-2k_1}2^{p_1+q_2}$, and hence
 \begin{equation}
  \frac{C_\m}{\abs{V_\eta\Phi}}\lesssim 2^{2k_1-k_2}2^{-p_1}.
 \end{equation}
 Together with the bounds \eqref{eq:bd_puremult} we again invoke Lemma \ref{lem:set_gain} to obtain that
 \begin{equation}
 \begin{aligned}
  2^{-\frac{q}{2}}\norm{I_{m,1}^j}_{L^2}&\lesssim 2^{-\frac{q}{2}}\left[2^{2k_1-k_2}2^{-p_1}\left\{1+2^{k_2-k_1}+(2^{p-p_1}+2^{q-q_1})[1+2^{k-k_1}]\right\}+2^{k-p_1}\right]\\
  &\qquad \cdot \min\{2^{p_\alpha+k_\alpha}\}\min\{2^{\frac{q_\alpha+k_\alpha}{2}}\} \cdot\norm{f}_{L^2}\norm{g}_{L^2}\\
  &\lesssim  2^{\frac{k}{2}}[2^{2k_1}+2^{k_1+k_2}+2^{k+k_2}]\cdot\norm{f}_{L^2}\norm{g}_{L^2} + 2^{\frac{k}{2}}2^{k_1}[2^{k_1}+2^{k}]\cdot 2^{-p_1}\cdot\norm{f}_{L^2}\norm{g}_{L^2}\\
  &\qquad +2^{k+\frac{3k_1}{2}-k_2}[2^{k_1}+2^k]\cdot 2^{-p_1-\frac{q_1}{2}}\norm{f}_{L^2}\norm{g}_{L^2},
 \end{aligned}
 \end{equation}
 where we used also that $2^{p_2}\lesssim 2^{p_1}$.
 The term $I_{m,2}^j$ can be estimated as in \eqref{eq:crossbd2}, and gives
 \begin{equation}
 \begin{aligned}
  2^{-\frac{q}{2}}\norm{I_{m,2}^j}_{L^2}&\lesssim 2^{\frac{5}{2}k}\cdot [2^{-p_1}\norm{Sf}_{L^2}+\norm{\Ups f}_{L^2}]\norm{g}_{L^2}.
 \end{aligned}
 \end{equation}
 As before, the simplest term is
 \begin{equation}
 \begin{aligned}
  2^{-\frac{q}{2}}\norm{I_{m,3}^j}_{L^2}&\lesssim 2^{-\frac{q}{2}}\cdot 2^{2k_1-k_2}2^{-p_1}\cdot \min\{2^{p_\alpha+k_\alpha}\}\min\{2^{\frac{q_\alpha+k_\alpha}{2}}\} \cdot\norm{f}_{L^2}\norm{Sg}_{L^2}\\
  &\lesssim 2^{2k_1+\frac{k}{2}-k_2}\cdot 2^{-p_1}\cdot \min\{2^{p_\alpha+k_\alpha}\}\norm{f}_{L^2}\norm{Sg}_{L^2}.
 \end{aligned}
 \end{equation}

 \item Case $q+C\le q_2\ll q_1$: Here we have that $\abs{\bar\sigma}\sim 2^{k_1+k_2+q_1}$, so that\footnote{As before we note that here $2^{k_1}\ll2^k\sim 2^{k_2}$ since $2^{k+k_2+q_2}\sim\vert\bar{\sigma}\vert\sim 2^{k_1+k_2+q_1}$.}
 \begin{equation}
  \frac{C_\m}{\abs{V_\eta\Phi}}\lesssim 2^{k_1}2^{-p_1}.
 \end{equation}
 As above it then follows that
 \begin{equation}
 \begin{aligned}
  2^{-\frac{q}{2}}\norm{I_{m,1}^j}_{L^2}&\lesssim 2^{-\frac{q}{2}}\left[2^{k_1}2^{-p_1}\left\{1+2^{k_2-k_1}+(2^{p-p_1}+2^{q-q_1})[1+2^{k-k_1}]\right\}+2^{k-p_1}\right]\\
  &\qquad \cdot \min\{2^{p_\alpha+k_\alpha}\}\min\{2^{\frac{q_\alpha+k_\alpha}{2}}\} \cdot\norm{f}_{L^2}\norm{g}_{L^2}\\
  &\lesssim   2^{\frac{3k}{2}}2^{k_1}\cdot 2^{-p_1}\cdot\norm{f}_{L^2}\norm{g}_{L^2},
 \end{aligned}
 \end{equation}
 as well as by \eqref{eq:vfcross_bd''}
 \begin{equation}
 \begin{aligned}
  2^{-\frac{q}{2}}\norm{I_{m,2}^j}_{L^2}&\lesssim 2^{\frac{3}{2}k+k_1}\left[\norm{Sf}_{L^2}+ \norm{\Ups f}_{L^2}\right]\norm{g}_{L^2}
 \end{aligned}
 \end{equation}
 and finally we have
 \begin{equation}
 \begin{aligned}
  2^{-\frac{q}{2}}\norm{I_{m,3}^j}_{L^2}&\lesssim 2^{-\frac{q}{2}}\cdot 2^{k_1}2^{-p_1}\cdot \min\{2^{p_\alpha+k_\alpha}\}\min\{2^{\frac{q_\alpha+k_\alpha}{2}}\} \cdot\norm{f}_{L^2}\norm{Sg}_{L^2}\\
  &\lesssim 2^{k_1+\frac{3k}{2}-k_2}\cdot 2^{-p_1}\cdot \min\{2^{p_\alpha+k_\alpha}\}\norm{f}_{L^2}\norm{Sg}_{L^2}.
 \end{aligned}
 \end{equation}
 
\end{enumerate}
In conclusion, we thus have that
\begin{equation}
\begin{aligned}
 t2^{-p-\frac{q}{2}}\norm{I^j_{\m,1}}_{L^2}&\lesssim t^{-1}\Big[(2^{3k_1-k_2}+2^{2k_1+k-k_2}+2^{k+k_1}+2^{2k})\min\{2^{\frac{q_\alpha+k_\alpha}{2}}\}\\
 &\qquad\quad +2^{\frac{k}{2}}(2^{2k_1}+2^{k_1+k_2}+2^{k+k_2}+2^{k+k_1}+2^{k+2k_1-k_2}+2^{2k+k_1-k_2})\Big]2^{-p_1}\norm{f}_{L^2}\norm{g}_{L^2}\\
 &\quad +t^{-1}(1+2^{k_1-k_2})[2^{k_1}+2^k]\cdot 2^{k+\frac{k_1}{2}}2^{-p_1-\frac{q_1}{2}}\norm{f}_{L^2}\norm{g}_{L^2},\\
 t2^{-p-\frac{q}{2}}\norm{I^j_{\m,2}}_{L^2}&\lesssim t^{-1} 2^{\frac{3}{2}k}(2^k+2^{k_1})\left[2^{-p_1}\norm{Sf}_{L^2}+ (2^{k_1-k_2}+1)\norm{\Ups f}_{L^2}\right]\norm{g}_{L^2},\\
 t2^{-p-\frac{q}{2}}\norm{I^j_{\m,2}}_{L^2}&\lesssim t^{-1}\Big[2^{k_1-k_2}[2^{k}+2^{k_1}]\cdot [2^k\min\{2^{\frac{q_\alpha+k_\alpha}{2}}\}+ 2^{\frac{k}{2}}\min\{2^{p_\alpha+k_\alpha}\}]\Big]\cdot 2^{-p_1}\norm{f}_{L^2}\norm{Sg}_{L^2}.
\end{aligned}
\end{equation}
One checks directly that this can be bounded by \eqref{eq:bilin2}, and thus finishes the proof of Proposition \ref{prop:bilin}.
\end{proof}

\section{Polynomial Time Propagation of the Vector Fields in the $B$ Norm}\label{sec:B_VFprop}
\begin{proposition}\label{prop:B_VFprop}
Assume that $U_\pm$ are solutions to \eqref{eq:IER_disp} with initial data satisfying \eqref{eq:assump_id}, with corresponding profiles $\U_\pm$, and satisfy for $0\le t\le T^{\ast}$,
\begin{align}
 \norm{U_\pm(t)}_{H^{2N_0}\cap H^{-1}}+\norm{S^aU_\pm(t)}_{L^2}&\le 2\varepsilon_1,\qquad 0\le a\le 4N,\label{eq:ass_many_Sob}\\
 \norm{S^a\U_\pm(t)}_{B}&\le 2\varepsilon_1,\qquad 0\le a\le 2N,\label{eq:ass_many_vfs}\\
 \norm{\Ups S^a\U_\pm(t)}_{L^2}&\le  2\varepsilon_1,\qquad 0\le a\le N+3,\label{eq:ass_few_vfs}
\end{align}
then for $0\leq t\leq T^\ast$ there holds that
\begin{equation}\label{eq:VFprop_bd}
 \sup_{a\le 2N}\norm{S^a\U_\pm(t)}_{B}\leq \varepsilon_0+C\varepsilon_1^2\ip{t}^{\frac{8}{N_0}}.
\end{equation}
In particular, for $T^\ast=\eps_1^{-\M}$ and $N_0=16\M$ we have if $\eps_1$ is sufficiently small that
\begin{equation}\label{eq:VFprop_btstrap}
 \sup_{a\le 2N}\norm{S^a\U_\pm(t)}_{B}\leq 2\eps_1.
\end{equation}
\end{proposition}

We note that \eqref{eq:VFprop_btstrap} follows directly from \eqref{eq:VFprop_bd}, which we establish in the remainder of this section.

\subsection{Reduction of the Proof to a Bilinear Estimate}
Let
\begin{equation}
 \widehat{\Q_\m(f,g)}(\xi,t)=\int_{\mathbb{R}^3}e^{it\Phi}\mathfrak{m}(\xi,\eta)\widehat{f}(\xi-\eta)\widehat{g}(\eta)d\eta,\quad \Phi=\pm\Lambda(\xi)+\Lambda(\xi-\eta)\pm\Lambda(\eta).
\end{equation}
To see what kinds of terms arise for vector fields applied to $\Q_\m$, we observe that if $V\in\{S,\Omega\}$ we have
\begin{equation}
 V_\xi\Phi=-V_\eta\Phi,
\end{equation}
which implies
\begin{equation}
 \begin{aligned}
 &\int_{\mathbb{R}^3}V_\xi[e^{it\Phi}]\mathfrak{m}(\xi,\eta)\widehat{f}(\xi-\eta)\widehat{g}(\eta)d\eta=\int_{\mathbb{R}^3}itV_\xi\Phi e^{it\Phi}\mathfrak{m}(\xi,\eta)\widehat{f}(\xi-\eta)\widehat{g}(\eta)d\eta\\
 &\quad =\int_{\mathbb{R}^3}(-itV_\eta\Phi) e^{it\Phi}\mathfrak{m}(\xi,\eta)\widehat{f}(\xi-\eta)\widehat{g}(\eta)d\eta=\int_{\mathbb{R}^3} e^{it\Phi}V_\eta[\mathfrak{m}(\xi,\eta)\widehat{f}(\xi-\eta)\widehat{g}(\eta)]d\eta.
 \end{aligned}
\end{equation}
Since $(V_\xi+V_\eta)\hat{f}(\xi-\eta)=(V\hat{f})(\xi-\eta)$ it follows that
\begin{equation}\label{eq:vf_distribute}
 V[\Q_\mathfrak{m}(f,g)(t)]=\Q_\mathfrak{m}(Vf,g)(t)+\Q_\mathfrak{m}(f,Vg)(t)+\Q_{(V_\xi+V_\eta)\mathfrak{m}}(f,g)(t).
\end{equation}
Furthermore, recall that by Lemma \ref{lem:vfmult} we have $(\Omega_\xi+\Omega_\eta)\m=0$ and $(S_\xi+S_\eta)\m=\m$. As a result, to prove bounds for $2N$ vector fields on the output it suffices to control bilinear expressions with no more than $N$ vector fields on one of the inputs. This is the content of the following:

\begin{lemma}\label{lem:red_VFprop}
Under the assumptions of Proposition \ref{prop:B_VFprop}, for $m\in\N$ such that $2^m\leq t\leq 2^{m+1}$ and letting
\begin{equation}
 I_{k_1,k_2}^{k,p,q}(s):=P_{k,p,q}\Q_\m(S^aP_{k_1}f(s),S^bP_{k_2}g(s)),
\end{equation}
there holds that
\begin{equation}\label{eq:sum_bound}
2^{-p-\frac{q}{2}}\norm{\sum_{k_1}\sum_{k_2} \int_{s=0}^t\tau_m(s)I_{k_1,k_2}^{k,p,q}(s)ds}_{L^2}\lesssim \varepsilon_1^2 2^{\frac{8}{N_0} m},\\
\end{equation}
where $0\le a+b\le 2N$, $0\le a\le N$ and $0\le b\le 2N$.
\end{lemma}

\begin{proof}

We start with a crude estimate that uses the above multiplier bound and the energy estimates (see also \eqref{eq:vfHNinterpol}):
\begin{equation*}
\begin{split}
 \norm{P_{k,p,q}\Q_\m(P_{k_1}F,P_{k_2}G)}_{L^2}&\lesssim 2^{\frac{3}{2}k}2^{p+\frac{q}{2}}2^k \norm{F}_{L^2}\norm{G}_{L^2}\\
 &\lesssim 2^{\frac{3}{2}k}2^{p+\frac{q}{2}}2^{k}\min\{2^{-N_0k_1^+}\norm{P_{k_1}F}_{H^{N_0}},2^{k_1}\norm{F}_{H^{-1}}\}\\
 &\qquad\times\min\{2^{-N_0k_2^+}\norm{P_{k_2}G}_{H^{N_0}},2^{k_2}\norm{G}_{H^{-1}}\}.
\end{split}
\end{equation*}
As a result, writing $\overline{k_{1,2}}:=\max\{k_1^+,k_2^+\}$ and $\underline{k_{1,2}}:=\min\{k_1^-,k_2^-\}$, we get that
\begin{equation*}
 2^{-p-\frac{q}{2}}\norm{\int_{s=0}^t\tau_m(s)I_{k_1,k_2}^{k,p,q}(s)ds}_{L^2}\lesssim 2^m 2^{2k}2^{-N_0\cdot\overline{k_{1,2}}+\underline{k_{1,2}}}\varepsilon_1^2.
\end{equation*}
Consequently, we obtain an acceptable contribution whenever
\begin{equation*}
\begin{split}
\max\{k_1,k_2\}\geq 2m/N_0,\text{ or }\min\{k_1,k_2\}\leq-2m,\text{ or } k\leq -m.
\end{split}
\end{equation*}
In what follows we may thus assume that $-m<k<2m/N_0$ and $-2m<k_1,k_2<2m/N_0$. 

Under the same conditions as in Lemma \ref{lem:red_VFprop}, instead of the summation in \eqref{eq:sum_bound} it will thus suffice to prove the bounds
\begin{equation}
 2^{-p-\frac{q}{2}}\sup_{k_1,k_2\in (-2m,2m/N_0)}\norm{\int_{s=0}^t\tau_m(s)I_{k_1,k_2}^{k,p,q}(s)ds}_{L^2}\lesssim \varepsilon_1^2 2^{\frac{6}{N_0}m}.
\end{equation}

In order to go further, we quantify the parameters in the $B$-norms of $F$, $G$. To this end, for $p_i,q_i\leq 0$, $i\in\{1,2\}$, we write
\begin{equation}
\begin{aligned}
 F_1&:=P_{k_1,p_1,q_2}S^aP_{k_1}f(s),\quad G_2(s):=P_{k_2,p_2,q_2}S^bP_{k_2}g(s),\\
 J_{k_1,p_1,q_1,k_2,p_2,q_2}^{k,p,q}&:=P_{k,p,q}\Q_\m(F_1(s),G_2(s)),\\
\end{aligned}
\end{equation}
so that
\begin{equation}
 I_{k_1,k_2}^{k,p,q}=\sum_{p_1,q_1\leq 0}\sum_{p_2,q_2\leq 0}J_{k_1,p_1,q_1,k_2,p_2,q_2}^{k,p,q}.
\end{equation}
Since we have
\begin{equation}
 \norm{J_{k_1,p_1,q_1,k_2,p_2,q_2}^{k,p,q}}_{L^2}\leq 2^{\frac{3}{2}k+p+\frac{q}{2}}2^{p_1+\frac{q_1}{2}}\norm{F_1}_{B}2^{p_2+\frac{q_2}{2}}\norm{G_2}_{B},
\end{equation}
as above it suffices to bound
\begin{equation}
 2^{-p-\frac{q}{2}}\sup_{\substack{k_1,k_2\in (-2m,2m/N_0)\\p_1,p_2\in (-2m,0]\\q_1,q_2\in (-4m,0]}}\norm{\int_{s=0}^t\tau_m(s)J_{k_1,p_1,q_1,k_2,p_2,q_2}^{k,p,q}(s)ds}_{L^2}\lesssim \varepsilon_1^2 2^{\frac{5}{N_0} m}.
\end{equation}

But this is precisely achieved by the bilinear estimate above in Proposition \ref{prop:bilin}: Using it we deduce via the bootstrap assumption and \eqref{eq:btstrap_D} that for $k_1,k_2\leq 2m/N_0$ we have
\begin{equation}
\begin{aligned}
 2^{-p-\frac{q}{2}}\norm{\tau_m(s)J_{k_1,p_1,q_1,k_2,p_2,q_2}^{k,p,q}(s)}_{L^2}&\lesssim 2^{-m}2^{\frac{5k_{\max}}{2}}\norm{F_1}_D\norm{G_2}_B\\
 &\quad +2^{-m} 2^{\frac{5k_{\max}}{2}}\Big[\norm{F_1}_B+\norm{SF_1}_B+\norm{\Ups F_1}_{L^2}\Big]\\
 &\qquad\qquad \cdot\Big[\norm{G_2}_{L^2}+2^{k_1}\norm{G_2}_{H^{-1}}+\norm{SG_2}_{L^2}\Big]\\
 &\lesssim 2^{-m+\frac{5m}{N_0}}\varepsilon_1^2.
\end{aligned} 
\end{equation}

This concludes the proof of Lemma \ref{lem:red_VFprop}, and with it the demonstration of Proposition \ref{prop:B_VFprop}.
\end{proof}

\section{More Bilinear Estimates}\label{sec:more_bilin}
In this section we prove bilinear estimates which will be instrumental for the propagation of $\Ups\V^{N+3}$ in $L^2$. For this bootstrap, the most difficult terms arise when $\Ups$ hits the phase and thereby produces an additional growth factor of $s$ -- see also \eqref{eq:Upsterms_k}. 
Lemmas \ref{lem:2ibp}--\ref{lem:bdry} provide the bilinear estimates necessary to deal with this in case several integrations by parts are available.

\begin{notation}
In what follows, we shall work with the following convention: when we write estimates as a sum of terms expressed in the sizes of our various localizations, this shall denote a sum of terms that can be bounded by the corresponding expressions. We note that this is consistent with uses of the triangle inequality and other algebraic manipulations in the context of $L^\infty$ type estimates on the multipliers and in particular for our improved set size estimates Lemma \ref{lem:set_gain}. For example, it is legitimate to write $2^{p+k}\sim\abs{\xi_\h}=\abs{\xi_\h-\eta_\h+\eta_\h}\leq \abs{\xi_\h-\eta_\h}+\abs{\eta_\h}\lesssim 2^{p_1+k_1}+2^{p_2+k_2}$ and similar.

To simplify notation in the proofs of bilinear estimate, we recall that we denote by $\Sz$ the set size gain from Lemma \ref{lem:set_gain}, i.e.
\begin{equation}
 \Sz=\min\{2^{p+k},2^{p_1+k_1},2^{p_2+k_2}\}\cdot\min\{2^{\frac{q+k}{2}},2^{\frac{q_1+k_1}{2}},2^{\frac{q_2+k_2}{2}}\}.
\end{equation}
\end{notation}

\subsection{Two integrations by parts}\label{ssec:2ibp}
Here we give a bilinear estimate for when two subsequent integrations by parts can be carried out. For the propagation of $\Ups V^{N+3} \U_{\pm}$ in $L^2$ this is relevant when neither input has the maximal amount of vector fields.

To make this precise, we introduce localizations for the vector fields along which we integrate by parts: With $\bar\chi:=1-\chi$ a smooth cutoff function that vanishes around the origin (cf.\ Section \ref{ssec:loc}),  for $L>0$ we denote by
\begin{equation}\label{eq:vfloc_not}
  \bar\chi_{V_\eta}=\bar\chi(L^{-1}2^{2k_1-p_1}V_\eta\Phi),\qquad \bar\chi_{V_{\xi-\eta}}=\bar\chi(L^{-1}2^{2k_2-p_2}V_{\xi-\eta}\Phi),\qquad V\in\{S,\Omega\}.
 \end{equation}
We note that by Lemma \ref{lem:vfsizes-mini}, on the support of each of these cutoffs we have that $\abs{\bar\sigma}\gtrsim L$. In particular, it follows (again by Lemma \ref{lem:vfsizes-mini}) that on the support of $\bar\chi_{V_\eta}$ we have that $\vert V'_{\xi-\eta}\Phi\vert\gtrsim 2^{-2k_2+p_2}L$ for $V'=S$ or $V'=\Omega$, and symmetrically for $\bar\chi_{V_{\xi-\eta}}$. Hence upon appropriate localization, several integrations by parts are possible.

To keep track of these localizations in our bilinear expressions, for a multiplier $\mathfrak{n}$ we shall then denote
\begin{equation}
 \Q^{V_\eta}_{\mathfrak{n}}[f,g]=\Q_{\mathfrak{n}\cdot\bar\chi_{V_\eta}}[f,g],\qquad \Q^{V,V'}_{\mathfrak{n}}[f,g]=\Q_{\mathfrak{n}\cdot\bar\chi_{V_\eta}\cdot\bar\chi_{V'_{\xi-\eta}}}[f,g],\qquad V,V'\in\{S,\Omega\}.
\end{equation}

\begin{remark}
 For integrations by parts along vector fields the cutoffs in \eqref{eq:vfloc_not} introduce no additional difficulties over $\frac{1}{V_\eta\Phi}$ resp.\ $\frac{1}{V'_{\xi-\eta}\Phi}$, which we account for in full detail in our estimates. To avoid further burdening the presentation, in the below proofs we will thus omit the corresponding terms of vector fields applied to the cutoffs. The careful reader can check the relevant estimates without further complications.
 
 In line with this, in the below proofs or when it is clear from the context we may also sometimes drop the superscripts for the bilinear terms in order to simplify the notation.
\end{remark}
 
Next we will show:
\begin{lemma}\label{lem:2ibp}
 Let $L:=2^{k_{\max}+k_{\min}+q_{\max}}$, $V\in\{S,\Omega\}$, and assume that $k_1\geq k_2$.
 Then there holds that for any $f=P_{k_1,p_1,q_1}f$ and $g=P_{k_2,p_2,q_2}g$ we have the estimate
 \begin{equation}
 \begin{aligned}
  \norm{P_{k,p,q}\Q^{V_\eta}_{s\m\Ups\Phi}[f,g]}_{L^2}&\lesssim s^{-1}\cdot 2^{2k_{\max}}2^{\frac{k}{2}}\Big[(\norm{f}_B+\norm{Sf}_B+\norm{\Ups f}_{L^2})(\norm{g}_B+\norm{Sg}_B+\norm{\Ups g}_{L^2})\\
  &\qquad\qquad\quad +\norm{f}_B(\norm{\Ups Sg}_{L^2}+\norm{S^2g}_{B})+(\norm{\Ups Sf}_{L^2}+\norm{S^2f}_{B})\norm{g}_B\Big].
 \end{aligned} 
 \end{equation}
\end{lemma}

The basic strategy of proof is straightforward: we perform two integrations by parts,  first along $V_\eta$, then along $V'_{\xi-\eta}$. The terms are then estimated via an $L^1\times L^\infty$ in Fourier space, as in Lemma \ref{lem:set_gain}. This gives a gain of $s^{-2}$, with a loss of localization sizes. These can then be compensated for by the $L^1$ set size gain on the Fourier side (see Lemma \ref{lem:set_gain}) and our $B$ norm estimates.

To provide some orientation for the estimates to come, we note that the most delicate part of them is when $k=k_{\min}$, and we observe that there it is essential that a loss of $2^{k_{\min}}$ only comes with at most one power of each $2^{-p_j}$, $j=1,2$, which can then be compensated for using the above considerations. Similarly, it is crucial to perform the integrations by parts in two different directions, in order not too lose too many powers of the same $2^{-p_j}$, $j=1,2$. (The situation when $k_2=k_{\min}$ is a little better, since on the one hand integration by parts along $V'_{\xi-\eta}$ gains back at least $2^{k_2}$, but also since the set size gain recovers both $2^{k_2+p_2}$ at once, unlike the case of $k=k_{\min}$, where the $2^{k}$ and $2^{p_j}$, $j=1,2$ need to be regained separately.)

\begin{proof}
 We rely on the estimates of Lemma \ref{lem:ibp_mult_bds}. After one integration by parts in $V_\eta$ we obtain
 \begin{equation}\label{eq:1ibp}
  \Q^{V_\eta}_{s\m\Ups\Phi}[f,g]=\Q^{V_\eta}_{V_\eta\left(\frac{\m\Ups\Phi}{V_\eta\Phi}\right)}[f,g]+\Q^{V_\eta}_{\frac{\m\Ups\Phi}{V_\eta\Phi}}[V_\eta f,g]+\Q^{V_\eta}_{\frac{\m\Ups\Phi}{V_\eta\Phi}}[f,Vg],
 \end{equation}
and a second integration by parts in $V'_{\xi-\eta}$ gives the following terms to be controlled:
\begin{align}
 \Q^{V,V'}_{V_\eta\left(\frac{\m\Ups\Phi}{V_\eta\Phi}\right)}[f,g]&=\frac{1}{s}\left[ \Q^{V,V'}_{V'_{\xi-\eta}\left(\frac{1}{V'_{\xi-\eta}\Phi}V_\eta\left(\frac{\m\Ups\Phi}{V_\eta\Phi}\right)\right)}[f,g]+\Q^{V,V'}_{\frac{1}{V'_{\xi-\eta}\Phi}V_\eta\left(\frac{\m\Ups\Phi}{V_\eta\Phi}\right)}[V' f,g]+\Q^{V,V'}_{\frac{1}{V'_{\xi-\eta}\Phi}V_\eta\left(\frac{\m\Ups\Phi}{V_\eta\Phi}\right)}[f,V'_{\xi-\eta}g] \right],\nonumber\\
 \label{eq:2ibp_term-1}\\
 \Q^{V,V'}_{\frac{\m\Ups\Phi}{V_\eta\Phi}}[V_\eta f,g]&=\frac{1}{s}\left[\Q^{V,V'}_{V'_{\xi-\eta}\left(\frac{1}{V'_{\xi-\eta}\Phi}\frac{\m\Ups\Phi}{V_\eta\Phi}\right)}[V_\eta f,g]+\Q^{V,V'}_{\frac{1}{V'_{\xi-\eta}\Phi}\frac{\m\Ups\Phi}{V_\eta\Phi}}[V'V_\eta f,g]+\Q^{V,V'}_{\frac{1}{V'_{\xi-\eta}\Phi}\frac{\m\Ups\Phi}{V_\eta\Phi}}[V_\eta f,V'_{\xi-\eta}g]\right],\label{eq:2ibp_term-2}\\
 \Q^{V,V'}_{\frac{\m\Ups\Phi}{V_\eta\Phi}}[f,Vg]&=\frac{1}{s}\left[\Q^{V,V'}_{V'_{\xi-\eta}\left(\frac{1}{V'_{\xi-\eta}\Phi}\frac{\m\Ups\Phi}{V_\eta\Phi}\right)}[f,Vg]+\Q^{V,V'}_{\frac{1}{V'_{\xi-\eta}\Phi}\frac{\m\Ups\Phi}{V_\eta\Phi}}[V'f,Vg]+\Q^{V,V'}_{\frac{1}{V'_{\xi-\eta}\Phi}\frac{\m\Ups\Phi}{V_\eta\Phi}}[f,V'_{\xi-\eta}Vg] \right].\label{eq:2ibp_term-3}
\end{align}
We recall from Section \ref{ssec:crossterms} that we have the bound
\begin{equation}\label{eq:cross_recall}
 \abs{V_\eta f(\xi-\eta)}\lesssim 2^{k_2-k_1}\left(\abs{Sf(\xi-\eta)}+(2^{q_1+p_2}+2^{q_2+p_1})\abs{\Ups f(\xi-\eta)}\right),
\end{equation}
and symmetrically for $V'_{\xi-\eta}g(\eta)$. 
Next we establish the following two lemmas, which together give the claim via \eqref{eq:1ibp}, thus concluding the proof of Lemma \ref{lem:2ibp}.
\end{proof}

\begin{lemma}\label{lem:cl1}
 Under the assumptions of Lemma \ref{lem:2ibp}, there holds that
 \begin{equation}\label{eq:cl1}
  \norm{\Q^{V,V'}_{V_\eta\left(\frac{\m\Ups\Phi}{V_\eta\Phi}\right)}[f,g]}_{L^2}\lesssim s^{-1}\cdot 2^{2k_{\max}}2^{\frac{k}{2}}\norm{(1,S)f}_B[\norm{(1,S)g}_B+\norm{\Ups g}_{L^2}].
 \end{equation}
\end{lemma}

\begin{lemma}\label{lem:cl23}
 Under the assumptions of Lemma \ref{lem:2ibp}, there holds that
 \begin{align}
  \norm{\Q^{V,V'}_{\frac{\m\Ups\Phi}{V_\eta\Phi}}[V_\eta f,g]}_{L^2}&\lesssim s^{-1}\cdot  2^{2k_{\max}}2^{\frac{k}{2}}[\norm{(1,S)Sf}_B+\norm{\Ups S f}_{L^2}]\cdot [\norm{(1,S)g}_B+\norm{\Ups g}_{L^2}]\label{eq:cl2},\\
  \norm{\Q^{V,V'}_{\frac{\m\Ups\Phi}{V_\eta\Phi}}[f,Sg]}_{L^2}&\lesssim s^{-1}\cdot  2^{2k_{\max}}2^{\frac{k}{2}}\norm{(1,S)f}_B[\norm{(1,S)Sg}_B+\norm{\Ups Sg}_{L^2}]\label{eq:cl3}.
\end{align}
\end{lemma}

(As announced, in the following proofs we drop the superscripts $V,V'$ for simplicity of notation.)
\begin{proof}[Proof of Lemma \ref{lem:cl1}]
It suffices to bound the three bilinear expressions in \eqref{eq:2ibp_term-1}. To this end, from \eqref{eq:vfquotient_apdx} it follows with $k_2\leq k_1$ (and \eqref{eq:freq_triang}) that
\begin{equation}\label{eq:vfquot'}
 \abs{\frac{(V'_{\xi-\eta})^2\Phi}{V'_{\xi-\eta}\Phi}}\lesssim 2^{k_1-k_2}(1+2^{p_1-p_2})
 \lesssim
 \begin{cases}
  (1+2^{k-k_2}2^{p-p_2}),&k=k_{\min},\\
  2^{k_1-k_2}(1+2^{p_1-p_2}),&k_2=k_{\min}.
 \end{cases}
\end{equation}
We then have
\begin{equation}
 \abs{V'_{\xi-\eta}\left(\frac{1}{V'_{\xi-\eta}\Phi}V_\eta\left(\frac{\m\Ups\Phi}{V_\eta\Phi}\right)\right)}\lesssim \abs{\frac{1}{V'_{\xi-\eta}\Phi}}\abs{\frac{(V'_{\xi-\eta})^2\Phi}{V'_{\xi-\eta}\Phi}}\abs{V_\eta\left(\frac{\m\Ups\Phi}{V_\eta\Phi}\right)}+\abs{\frac{1}{V'_{\xi-\eta}\Phi}}\abs{V'_{\xi-\eta}V_\eta\left(\frac{\m\Ups\Phi}{V_\eta\Phi}\right)},
\end{equation}
which we estimate next. To provide some orientation, we note that the most delicate part of these estimates is when $k=k_{\min}$, and we observe that there it is essential that a loss of $2^{k_{\min}}$ only comes with at most one power of each $2^{-p_j}$, $j=1,2$. 

With \eqref{eq:bds_1ibp} and \eqref{eq:vfquot'} there holds that
\begin{equation}\label{ADDED1}
\begin{aligned}
 &\abs{\frac{1}{V'_{\xi-\eta}\Phi}}\abs{\frac{(V'_{\xi-\eta})^2\Phi}{V'_{\xi-\eta}\Phi}}\abs{V_\eta\left(\frac{\m\Ups\Phi}{V_\eta\Phi}\right)}\lesssim 2^{2k_2-p_2}\abs{\bar\sigma}^{-1}\cdot 2^{k_1-k_2}(1+2^{p_1-p_2})\cdot \abs{V_\eta\left(\frac{\m\Ups\Phi}{V_\eta\Phi}\right)}\\
 &\qquad\lesssim 2^{-q_{\max}}(1+2^{p-p_1})
 \begin{cases}
   \left[2^{k_{\max}-k-p_2}(1+2^{p_1-p_2})\right]\left(2^{k_{\max}}+2^{k}2^{p-p_1}\right) ,&k=k_{\min},\\
    \left[2^{-p_2}(1+2^{p_1-p_2})\right]\left(2^{2k_{\max}-k_2}+2^{k_{\max}}[1+2^{p_2-p_1}]\right),&k_2=k_{\min},
 \end{cases}\\
 &\qquad\lesssim 2^{-q_{\max}}(1+2^{p-p_1})
 \begin{cases}
    2^{2k_{\max}-k-p_2}(1+2^{k-k_2}2^{p-p_2})+2^{k_{\max}-p_2}2^{p-p_1}(1+2^{p_1-p_2}),&k=k_{\min},\\
    2^{k_{\max}-p_2}(1+2^{p_1-p_2})\left(2^{k_{\max}-k_2}+1+2^{p_2-p_1}\right),&k_2=k_{\min},
 \end{cases}\\
 &\qquad\lesssim 2^{-q_{\max}}(1+2^{p-p_1})
 \begin{cases}
    2^{2k_{\max}-k-p_2}+2^{k_{\max}-p_2}[1+2^{p_1-p_2}+2^{p-p_1}+2^{p-p_2}],&k=k_{\min},\\
    2^{k_{\max}-p_2}(1+2^{p_1-p_2})\left(1+2^{k_{\max}-k_2}+2^{p_2-p_1}\right),&k_2=k_{\min}
 \end{cases}
\end{aligned} 
\end{equation}
and with \eqref{eq:bds_2ibp} we obtain that
\begin{equation}
\begin{aligned}
 &\abs{\frac{1}{V'_{\xi-\eta}\Phi}}\abs{V'_{\xi-\eta}V_\eta\left(\frac{\m\Ups\Phi}{V_\eta\Phi}\right)}\lesssim 2^{2k_2-p_2}\abs{\bar\sigma}^{-1}\cdot\abs{V'_{\xi-\eta}V_\eta\left(\frac{\m\Ups\Phi}{V_\eta\Phi}\right)}\\
 &\qquad\lesssim 2^{-q_{\max}}
 \begin{cases}
   (2^{k_2-k-p_2})(2^{k_{\max}}+2^{k}2^{p-p_1})[1+2^{p-p_1}+2^{p-p_2}],&k=k_{\min},\\
   2^{k_2-k_{\max}-p_2}\Big(2^{2k_{\max}-k_2-p_2}(1+2^{p_2-p_1}+ 2^{k_{\max}-k_2})(2^{p_1}+2^{p})&\\
   \quad +(1+2^{p_2-p_1})(2^{k_{\max}}+2^{k_2}2^{p_2-p_1})\Big),&k_2=k_{\min}
  \end{cases}\\
  &\qquad\lesssim 2^{-q_{\max}}
 \begin{cases}
   2^{2k_2-k-p_2}(1+2^{p-p_1})+2^{k_2-p_2}2^{p-p_1}[1+2^{p-p_1}+2^{p-p_2}],&k=k_{\min},\\
   2^{k_{\max}-2p_2}(1+2^{k_{\max}-k_2}+2^{p_2-p_1})(2^{p_1}+2^{p})+2^{k_2}(2^{-p_2}+2^{-p_1}+2^{p_2-2p_1}),&k_2=k_{\min},
  \end{cases}
\end{aligned}
\end{equation}
where we used in the second inequality that in case $k=k_{\min}$ there holds that $2^{p-p_2}\lesssim 2^{p-p_1}+2^{k-k_2}2^{2p-p_1-p_2}$, as in \eqref{eq:freq_triang}.  

Using that $\Sz\lesssim 2^{\frac{q+k}{2}}2^{k+p}$ when there is a loss of $k_{\min}=k$ and $\Sz\lesssim 2^{\frac{q+k}{2}}2^{k_2+p_2}$ when $k_{\min}=k_2$, and $\Sz\lesssim 2^{\frac{q+k}{2}}2^{k_j+p_j}$ when there is a loss of $2^{-p_1-p_2-p_j}$, $j=1,2$, we thus have (using also \eqref{eq:freq_triang} for a term from \eqref{ADDED1})
\begin{equation}
 \abs{V'_{\xi-\eta}\left(\frac{1}{V'_{\xi-\eta}\Phi}V_\eta\left(\frac{\m\Ups\Phi}{V_\eta\Phi}\right)\right)}\cdot\Sz\lesssim 2^{2k_{\max}}2^{\frac{k}{2}}2^{-\frac{q_{\max}}{2}}\cdot 2^{-p_1-p_2}.
\end{equation}
By Lemma \ref{lem:set_gain} this yields directly that
\begin{equation}
 \norm{\Q_{V'_{\xi-\eta}\left(\frac{1}{V'_{\xi-\eta}\Phi}V_\eta\left(\frac{\m\Ups\Phi}{V_\eta\Phi}\right)\right)}[f,g]}_{L^2}\lesssim \abs{V'_{\xi-\eta}\left(\frac{1}{V'_{\xi-\eta}\Phi}V_\eta\left(\frac{\m\Ups\Phi}{V_\eta\Phi}\right)\right)}\cdot\Sz\cdot\norm{f}_{L^2}\norm{g}_{L^2}\lesssim 2^{\frac{k}{2}}\norm{f}_B\norm{g}_B. 
\end{equation}

Note that by \eqref{eq:bds_1ibp} and \eqref{eq:freq_triang}
\begin{equation}\label{eq:ibp_frac1}
\begin{aligned}
 \abs{\frac{1}{V'_{\xi-\eta}\Phi}V_\eta\left(\frac{\m\Ups\Phi}{V_\eta\Phi}\right)}&\lesssim 2^{2k_2-p_2}\abs{\bar\sigma}^{-1}(1+2^{p-p_1})\cdot 
 \begin{cases}
   2^{k_{\max}}+2^{k}2^{p-p_1},&k=k_{\min},\\
   2^{2k_{\max}-k_2}+2^{k_{\max}}[1+2^{p_2-p_1}],&k_2=k_{\min},
 \end{cases}\\
 &\lesssim 2^{-q_{\max}}(1+2^{p-p_1})
 \begin{cases}
   2^{2k_{\max}-k-p_2}+2^{k_{\max}-p_2}2^{p-p_1},&k=k_{\min},\\
   2^{k_{\max}}2^{-p_2}(1+2^{p_2-p_1}), &k_2=k_{\min}.
   \end{cases}\\
 &\lesssim 2^{-q_{\max}}
 \begin{cases}
   2^{2k_{\max}-k-p_2}(1+2^{p-p_1})+2^{k_{\max}-p_2}2^{p-p_1}(1+2^{p-p_1}),&k=k_{\min},\\
   2^{k_2-p_1}+2^{k_{\max}-p_2}(1+2^{p_2-p_1})+2^{k_2}2^{p_2-2p_1}, &k_2=k_{\min}.
  \end{cases}
\end{aligned} 
\end{equation}
Hence it follows just as above with Lemma \ref{lem:set_gain} that
\begin{equation}
 \norm{\Q_{\frac{1}{V'_{\xi-\eta}\Phi}V_\eta\left(\frac{\m\Ups\Phi}{V_\eta\Phi}\right)}[V' f,g]}_{L^2}\lesssim 2^{2k_{\max}}2^{\frac{k}{2}}\norm{V'f}_B\norm{g}_B.
\end{equation}
Similarly, by \eqref{eq:cross_recall} we will show that
\begin{equation}\label{eq:ibp_1stcross}
\begin{aligned}
 \norm{\Q_{\frac{1}{V'_{\xi-\eta}\Phi}V_\eta\left(\frac{\m\Ups\Phi}{V_\eta\Phi}\right)}[f,V'_{\xi-\eta} g]}_{L^2}&\lesssim \abs{\frac{1}{V'_{\xi-\eta}\Phi}V_\eta\left(\frac{\m\Ups\Phi}{V_\eta\Phi}\right)}\norm{f}_{L^2} \cdot 2^{k_1-k_2}\Sz[\norm{Sg}_{L^2}+(2^{p_1}+2^{p_2})\norm{\Ups g}_{L^2}]\\
 &\lesssim 2^{2k_{\max}}2^{\frac{k}{2}}\norm{f}_B\left[\norm{Sg}_B+\norm{\Ups g}_{L^2}\right].
\end{aligned}
\end{equation}
Here, with \eqref{eq:ibp_frac1} the term with $Sg$ can be estimated directly as above. It is more delicate to control the $\Ups$ contribution: we observe that there is at most a loss of one order of $2^{-p_2}$ in \eqref{eq:ibp_frac1}, hence for the term with $2^{p_2}\Ups g$ we have the bound
\begin{equation}
\begin{aligned}
  &\abs{\frac{1}{V'_{\xi-\eta}\Phi}V_\eta\left(\frac{\m\Ups\Phi}{V_\eta\Phi}\right)}\norm{f}_{L^2}\cdot 2^{k_1-k_2}\cdot 2^{p_2}\Sz\cdot\norm{\Ups g}_{L^2}\\
  &\qquad \lesssim\norm{f}_{L^2}\norm{\Ups g}_{L^2}2^{-q_{\max}}\cdot\Sz
  \begin{cases}
   2^{2k_{\max}-k}(1+2^{p-p_1})+2^{k_{\max}}2^{p-p_1}(1+2^{p-p_1}),&k=k_{\min},\\
   2^{k_{\max}+p_2-p_1}+2^{2k_{\max}-k_2}(1+2^{p_2-p_1}), &k_2=k_{\min}.
  \end{cases}\\
  &\qquad \lesssim 2^{2k_{\max}}2^{\frac{k}{2}}\norm{f}_B\norm{\Ups g}_{L^2},
\end{aligned}  
\end{equation}
where we used that $\Sz\lesssim 2^{\frac{q_{\max}+k}{2}}\min\{2^{p+k},2^{p_1+k_1}\}$.
On the other hand, for the term with $2^{p_1}\Ups g$ we have
\begin{equation*}
\begin{aligned}
  &\abs{\frac{1}{V'_{\xi-\eta}\Phi}V_\eta\left(\frac{\m\Ups\Phi}{V_\eta\Phi}\right)}\norm{f}_{L^2}\cdot 2^{k_1-k_2}\cdot 2^{p_1}\Sz\cdot\norm{\Ups g}_{L^2}\\
  &\qquad \lesssim\norm{f}_{L^2}\norm{\Ups g}_{L^2}2^{-q_{\max}}\cdot\Sz
  \begin{cases}
   2^{2k_{\max}-k+p_1-p_2}(1+2^{p-p_1})+2^{k_{\max}-p_2}2^{p-p_1}(2^{p_1}+2^{p}),&k=k_{\min},\\
   2^{k_{\max}}+2^{k_{\max}-p_2}(2^{p_1}+2^{p_2}), &k_2=k_{\min}.
   \end{cases}\\
   &\qquad \lesssim\norm{f}_{L^2}\norm{\Ups g}_{L^2}2^{-q_{\max}}\cdot\Sz
  \begin{cases}
   2^{k_{\max}}(2^{k_{\max}-k}+2^{p-p_2})(1+2^{p-p_1})+2^{k_{\max}-p_2}2^{p-p_1}(2^{p_1}+2^{p}),&k=k_{\min},\\
   2^{k_{\max}}+2^{k_{\max}-p_2}(2^{p_1}+2^{p_2}), &k_2=k_{\min}.
   \end{cases}\\
  &\qquad \lesssim 2^{2k_{\max}}2^{\frac{k}{2}}\norm{f}_B\norm{\Ups g}_{L^2},
\end{aligned}  
\end{equation*}
where we used $\Sz\lesssim 2^{\frac{q_{\max}+k}{2}}\min\{2^{p+k},2^{p_2+k_2}\}$.

\end{proof}

\begin{proof}[Proof of Lemma \ref{lem:cl23}]
As before, we will simply bound the bilinear expressions in \eqref{eq:2ibp_term-2} and \eqref{eq:2ibp_term-3}. To begin, note that 
\begin{equation}
 \abs{V'_{\xi-\eta}\left(\frac{1}{V'_{\xi-\eta}\Phi}\frac{\m\Ups\Phi}{V_\eta\Phi}\right)}\lesssim \abs{\frac{1}{V'_{\xi-\eta}\Phi}}\abs{\frac{(V'_{\xi-\eta})^2\Phi}{V'_{\xi-\eta}\Phi}}\abs{\frac{\m\Ups\Phi}{V_\eta\Phi}}+\abs{\frac{1}{V'_{\xi-\eta}\Phi}}\abs{V'_{\xi-\eta}\left(\frac{\m\Ups\Phi}{V_\eta\Phi}\right)},
\end{equation}
which we can bound as follows: by \eqref{eq:vfquot'} and \eqref{eq:bds_1ibp}, followed by \eqref{eq:freq_triang}, there holds that
\begin{equation}
\begin{aligned}
 &\abs{\frac{1}{V'_{\xi-\eta}\Phi}}\abs{\frac{(V'_{\xi-\eta})^2\Phi}{V'_{\xi-\eta}\Phi}}\abs{\frac{\m\Ups\Phi}{V_\eta\Phi}}\lesssim 2^{2k_2-p_2}\abs{\bar\sigma}^{-1}\cdot 2^{k_1-k_2}(1+2^{p_1-p_2})\cdot 2^{2k_1-k_2}[2^{p-p_1}+2^{k-k_1}]\\
 &\qquad\lesssim 2^{-q_{\max}}
 \begin{cases}
   2^{2k_2-k-p_2}(1+2^{p_1-p_2})2^{p-p_1}+2^{k_2-p_2}(1+2^{p_1-p_2}) ,&k=k_{\min},\\
   2^{2k_{\max}-k_2-p_2}(1+2^{p_1-p_2})(1+2^{p-p_1}),&k_2=k_{\min},
 \end{cases}\\
 &\qquad \lesssim 2^{-q_{\max}}
 \begin{cases}
   2^{2k_2-k-p_2}2^{p-p_1}+2^{k_2-p_2}2^{p-p_2}2^{p-p_1}+2^{k_2-p_2}(1+2^{p_1-p_2}) ,&k=k_{\min},\\
   2^{2k_{\max}-k_2-p_2}(1+2^{p_1-p_2})(1+2^{p-p_1}),&k_2=k_{\min},
 \end{cases}
\end{aligned}
\end{equation}
and similarly by \eqref{eq:bds_1ibp'}
\begin{equation}\label{eq:bds_1ibp''}
\begin{aligned}
 \abs{\frac{1}{V'_{\xi-\eta}\Phi}}\abs{V'_{\xi-\eta}\left(\frac{\m\Ups\Phi}{V_\eta\Phi}\right)}&\lesssim 2^{2k_2-p_2}\abs{\bar\sigma}^{-1}\cdot 
 \begin{cases}
   2^{k_{\max}}(2^{p-p_1}+2^{p-p_2})+2^{k}(1+2^{p_1-p_2}),&k=k_{\min},\\
   2^{k_1}2^{p_2-p_1}+2^{2k_1-k_2}[1+2^{k_1-k_2-p_2}(2^{p_1}+2^{p})],&k_2=k_{\min},
 \end{cases}\\
 &\lesssim 2^{-q_{\max}} 
 \begin{cases}
   2^{2k_{\max}-k-p_2}(2^{p-p_1}+2^{p-p_2})+2^{k_{\max}-p_2}(1+2^{p_1-p_2}),&k=k_{\min},\\
   2^{2k_{\max}-k_2-2p_2}(2^{p_1}+2^{p})+2^{k_{\max}}(2^{-p_2}+2^{-p_1}),&k_2=k_{\min},
 \end{cases}\\
 &\lesssim 2^{-q_{\max}} 
 \begin{cases}
   2^{2k_{\max}-k-p_2}2^{p-p_1}+2^{k_{\max}-p_2}(1+2^{p_1-p_2}+2^{p-p_2}2^{p-p_1}),&k=k_{\min},\\
   2^{2k_{\max}-k_2-p_2}(2^{p_1-p_2}+2^{p-p_2})+2^{k_{\max}}(2^{-p_2}+2^{-p_1}),&k_2=k_{\min},
 \end{cases}
\end{aligned}
\end{equation}
where we have used that in case $k=k_{\min}$, by \eqref{eq:freq_triang} there holds $2^{p-p_2}\lesssim 2^{p-p_1}+2^{k-k_2+2p-p_2-p_1}$. We observe that there is only a loss of one order of $2^{-p_1}$ in \eqref{eq:bds_1ibp''}, and we are thus in an analogous position as for \eqref{eq:ibp_1stcross}. Using that
\begin{alignat}{2}
 2^{-k}(2^{p_2-p_1}+2^{p_1-p_2})\Sz&\lesssim 2^\frac{q+k}{2}2^p,\qquad&&\textnormal{when }k=k_{\min},\\
 2^{-k_2}(2^{p-p_1}+2^{p_1-p})\Sz&\lesssim 2^\frac{q+k}{2}2^{p_2},\qquad&& \textnormal{when }k_2=k_{\min},
\end{alignat}
we can then proceed as detailed there to show that
\begin{equation}
 \norm{\Q_{V'_{\xi-\eta}\left(\frac{1}{V'_{\xi-\eta}\Phi}\frac{\m\Ups\Phi}{V_\eta\Phi}\right)}[V_\eta f,g]}_{L^2}\lesssim 2^{2k_{\max}}2^{\frac{k}{2}}[\norm{Sf}_B+\norm{\Ups f}_{L^2}]\norm{g}_B,
\end{equation}
and similarly (but easier)
\begin{equation}
 \norm{\Q_{V'_{\xi-\eta}\left(\frac{1}{V'_{\xi-\eta}\Phi}\frac{\m\Ups\Phi}{V_\eta\Phi}\right)}[f,Vg]}_{L^2}\lesssim 2^{2k_{\max}}2^{\frac{k}{2}}\norm{f}_B\norm{Vg}_B.
\end{equation}

\medskip
Finally, we note that by \eqref{eq:bds_1ibp}
\begin{equation}\label{eq:ibp_frac2}
 \abs{\frac{1}{V'_{\xi-\eta}\Phi}\frac{\m\Ups\Phi}{V_\eta\Phi}}\lesssim 2^{2k_2-p_2}\abs{\bar\sigma}^{-1}\cdot 2^{2k_1-k_2}[2^{p-p_1}+2^{k-k_1}]\lesssim
 2^{-q_{\max}} 
 \begin{cases}
   2^{2k_2-k-p_2}2^{p-p_1}+2^{k_2-p_2},&k=k_{\min},\\
   2^{k_{\max}-p_2}(1+2^{p-p_1}),&k_2=k_{\min},
 \end{cases}
\end{equation}
We observe that \eqref{eq:ibp_frac2} only has losses of one order in $2^{-p_1}$ and $2^{-p_2}$, so that we have by \eqref{eq:cross_recall} and arguments as above that
\begin{equation}
\begin{aligned}
 \norm{\Q_{\frac{1}{V'_{\xi-\eta}\Phi}\frac{\m\Ups\Phi}{V_\eta\Phi}}[V_\eta f,V'_{\xi-\eta}g]}_{L^2}&\lesssim 2^{2k_{\max}}2^{\frac{k}{2}}(\norm{Sf}_B+\norm{\Ups f}_{L^2})(\norm{Sg}_B+\norm{\Ups g}_{L^2}).
\end{aligned}
\end{equation}
Here products of $S$ with itself or $\Ups$ can be estimated as before, whereas for the $\Ups$-$\Ups$ interaction we note that the only new difficulty arises when $k=k_{\min}$ and we have a factor of the form $2^{2k_2-k+p+p_i-p_j}$, $i,j\in\{1,2\}$. We then invoke once again \eqref{eq:freq_triang} to bound this as $2^{2k_2-k+p}+2^{k_2+2p-p_j}$, which we can then control as usual with $\Sz$. 
Similarly, using \eqref{eq:ibp_frac2} with in addition \eqref{eq:crossvf} yields that
\begin{equation}
 \norm{\Q_{\frac{1}{V'_{\xi-\eta}\Phi}\frac{\m\Ups\Phi}{V_\eta\Phi}}[V'V_\eta f,g]}_{L^2}\lesssim 2^{2k_{\max}}2^{\frac{k}{2}} [\norm{(1,V')Sf}_B+\norm{(1,V')\Ups f}_{L^2}]\norm{g}_B.
\end{equation}
The two remaining terms are similar (in fact easier) and give
\begin{equation}
 \norm{\Q_{\frac{1}{V'_{\xi-\eta}\Phi}\frac{\m\Ups\Phi}{V_\eta\Phi}}[V'f,Vg]}_{L^2}\lesssim  2^{2k_{\max}}2^{\frac{k}{2}}\norm{V'f}_B\norm{Vg}_B,
\end{equation}
and
\begin{equation}
 \norm{\Q_{\frac{1}{V'_{\xi-\eta}\Phi}\frac{\m\Ups\Phi}{V_\eta\Phi}}[f,V'_{\xi-\eta}Vg]}_{L^2}\lesssim 2^{2k_{\max}}2^{\frac{k}{2}}\norm{f}_B[\norm{\Ups Vg}_{L^2}+\norm{SVg}_{B}],
\end{equation}
concluding the proof of Lemma \ref{lem:cl23}.
\end{proof}

\subsection{``Extreme'' terms and integration by parts}
Here we consider scenarios where normal forms are not possible and two integrations by parts as in Section \ref{ssec:2ibp} are not feasible due to the extra vector fields required. Since this will only happen when one of the inputs has very few vector fields, we can then resort to \eqref{eq:Linfdecay}, i.e.\ to the dispersive decay estimate $\norm{e^{\pm is\Lambda}P_{k,p,q}f}_{L^\infty}\lesssim s^{-1} 2^{\frac{3}{2}k}\norm{f}_D$.

\begin{lemma}\label{lem:bdry}
 Let $L:=2^{k_{\max}+k_{\min}+q_{\max}}$, $V\in\{S,\Omega\}$, and assume that $k_1\geq k_2$.
 Then there holds that for any $f=P_{k_1,p_1,q_1}f$, $F=P_{k_1,p_1,q_1}F$ and $g=P_{k_2,p_2,q_2}g$, $G=P_{k_2,p_2,q_2}G$ we have the estimates
 \begin{equation}\label{eq:allV1}
 \begin{aligned}
  2^{-\frac{5}{2}k_{\max}^+}\norm{P_{k,p,q}\Q^{V_\eta}_{s\m\Ups\Phi}[F,g]}_{L^2}&\lesssim s^{-1}\Big[\norm{(1,S)F}_B+\norm{\Ups F}_{L^2}\Big]\cdot [\norm{(1,S)g}_D+\norm{g}_B]\\
  &\quad+s^{-1}\norm{(1,S)F}_B \Big[\norm{(1,S,S^2)g}_B+\norm{(1,S)\Ups g}_{L^2}\Big]
 \end{aligned}
 \end{equation}
 and
 \begin{equation}\label{eq:allV2}
 \begin{aligned} 
  2^{-\frac{7}{2}k_{\max}^+}\norm{P_{k,p,q}\Q^{V_\eta}_{s\m\Ups\Phi}[f,G]}_{L^2}&\lesssim s^{-1}\Big[\norm{(1,S)Sf}_B+\norm{\Ups S f}_{L^2}\Big]\cdot \Big[\norm{(1,S)G}_B+\norm{\Ups G}_{L^2}\Big]\\
  &\quad +s^{-1}\norm{(1,S)f}_D\Big[\norm{S G}_B+\norm{\Ups G}_{L^2}+\norm{S G}_{H^{-1}}\Big],
 \end{aligned}
 \end{equation}
 where $k_{\max}^+:=\max\{k_{\max},0\}$.
\end{lemma}

The main point here is that unlike the estimates in Lemma \ref{lem:2ibp}, here for either one of the two inputs (either $F$ in \eqref{eq:allV1} or $G$ in \eqref{eq:allV2}), we require only \emph{one} additional vector field $S$ in the $B$ norm and \emph{one} $\Upsilon$ in $L^2$. This makes these estimates compatible with a bootstrap argument that propagates a certain number of vector fields $S$ and one order of $\Upsilon$, as is carried out below in Section \ref{sec:Ups_propagation}.

On a technical level, to have access to the dispersive decay we require a version of H\"older's inequality \eqref{ProdRule2} and symbol estimates for the multipliers that arise after integration by parts. These are detailed in Section \ref{sec:symbols}. We recall here for convenience the algebra property
\begin{equation}\label{eq:alg_prop2}
 \norm{m_1\cdot m_2}_{\W}\lesssim\norm{m_1}_{\W}\norm{m_2}_{\W},
\end{equation}
which implies that for expressions that can be factored in $\xi,\xi-\eta,\eta$ the bounds established as absolute values under localization also hold as symbol bounds in $\W$, a fact we shall make frequent use of.
In particular, from \eqref{eq:simpmultbd}, \eqref{eq:Ups_mult} and \eqref{eq:UpsPhi_bd} we have that
\begin{equation}\label{eq:simp_symb}
 \norm{\m}_{\W}\lesssim C_\m,\quad\norm{\Ups_\xi\m}_{\W}\lesssim C_{\Ups\m},\quad \norm{\Ups_\xi\Phi}_{\W}\lesssim 2^{p}+2^{k-k_1}2^{p_1}.
\end{equation}
For expressions involving the inverse of a vector field on the phase we invoke Lemma \ref{lem:vf_symb}, which states that also such expressions have the expected symbol bounds, i.e.\
\begin{equation}\label{eq:vf_symb}
  \norm{\frac{1}{V_\eta\Phi}\cdot \bar\chi_{V_\eta}}_{\W}\lesssim L^{-1}\cdot 2^{2k_1-p_1}.
\end{equation} 
From this and the algebra property \eqref{eq:alg_prop2} we can deduce for example that
 \begin{equation}\label{eq:1ibp_symb}
  \norm{\frac{\m \Ups_\xi\Phi}{V_\eta\Phi}\cdot\bar\chi_{V_\eta}}_{\W}\lesssim 2^{2k_1-k_2}[2^{p-p_1}+2^{k-k_1}]
  \lesssim 2^{2k_{\max}-k_2}(1+2^{p-p_1}).
 \end{equation}
 
\begin{proof}
 We begin with the proof of \eqref{eq:allV1}. 
 \paragraph{Case $2^{p_2}\lesssim 2^{p_1}$.} Here we integrate by parts as in \eqref{eq:1ibp}, obtaining
 \begin{equation}
  \Q_{s\m\Ups\Phi}[F,g]=\Q_{V_\eta\left(\frac{\m\Ups\Phi}{V_\eta\Phi}\right)}[F,g]+\Q_{\frac{\m\Ups\Phi}{V_\eta\Phi}}[V_\eta F,g]+\Q_{\frac{\m\Ups\Phi}{V_\eta\Phi}}[F,Vg].
 \end{equation}
 Using the bound \eqref{eq:1ibp_symb} and the dispersive decay estimate $\norm{e^{is\Lambda}\varphi}_{L^\infty}\lesssim s^{-1} 2^{\frac{3}{2}k_2}\norm{\varphi}_D$, we obtain with H\"older's inequality \eqref{ProdRule2} for the last term that
 \begin{equation}
 \begin{aligned}
  \norm{\Q_{\frac{\m\Ups\Phi}{V_\eta\Phi}}[F,Vg]}_{L^2}&\lesssim \norm{\frac{\m\Ups\Phi}{V_\eta\Phi}\bar\chi_{V_\eta}}_{\W}\norm{F}_{L^2}\norm{e^{it\Lambda}Vg}_{L^\infty}\\
  &\lesssim 2^{2k_{\max}-k_2}(1+2^{p-p_1})\norm{F}_{L^2}\norm{e^{it\Lambda}Vg}_{L^\infty}\lesssim s^{-1} 2^{\frac{5}{2}k_{\max}}\norm{F}_B\norm{Vg}_D.
 \end{aligned} 
 \end{equation}
 Similarly, we have from \eqref{eq:cross_recall} that
 \begin{equation}
  \norm{\Q_{\frac{\m\Ups\Phi}{V_\eta\Phi}}[V_\eta F,g]}_{L^2}\lesssim s^{-1}2^{\frac{5}{2}k_{\max}}[\norm{VF}_B+(1+2^{p+p_2-p_1})\norm{\Ups F}_{L^2}]\cdot \norm{g}_D,
 \end{equation}
 which suffices since $2^{p_2-p_1}\lesssim 1$ by assumption. Next, we note that the first term can be treated as in Lemma \ref{lem:cl1}, i.e.\ we can apply \eqref{eq:cl1} to obtain after another integration by parts that
 \begin{equation}
  \norm{\Q_{V_\eta\left(\frac{\m\Ups\Phi}{V_\eta\Phi}\right)}[F,g]}_{L^2}\lesssim s^{-1}\cdot  2^{2k_{\max}}2^{\frac{k}{2}}\norm{(1+V)F}_B[\norm{(1+S)g}_B+\norm{\Ups g}_{L^2}].
 \end{equation}

 \paragraph{Case $2^{p_1}\ll 2^{p_2}$.} On the other hand, integrating by parts in $V'_{\xi-\eta}$ yields
 \begin{equation}\label{eq:1ibpV'}
  \Q_{s\m\Ups\Phi}[F,g]=\Q_{V'_{\xi-\eta}\left(\frac{\m\Ups\Phi}{V'_{\xi-\eta}\Phi}\right)}[F,g]+\Q_{\frac{\m\Ups\Phi}{V'_{\xi-\eta}\Phi}}[V' F,g]+\Q_{\frac{\m\Ups\Phi}{V'_{\xi-\eta}\Phi}}[F,V'_{\xi-\eta}g].
 \end{equation}
 In analogy with \eqref{eq:1ibp_symb}, we use $k_2\leq k_1$, $2^{p_1}\ll 2^{p_2}$, \eqref{eq:simp_symb} and the algebra property \eqref{eq:alg_prop2} to conclude that
 \begin{equation}\label{eq:1ibp''}
  \norm{\frac{\m\Ups_\xi\Phi}{V'_{\xi-\eta}\Phi}\bar\chi_{V'_{\xi-\eta}}}_{\W}\lesssim 2^{2k_2-p_2}L^{-1}C_\m(2^p+2^{p_1}2^{k-k_1})\lesssim  2^{k_2-p_2}(2^{p}+2^{p_1})\lesssim 2^{k_2}(1+2^{p-p_2}).
 \end{equation}
 From this it follows that
 \begin{equation}\label{eq:1ibp'''}
 \begin{aligned}
  \norm{V'_{\xi-\eta}\left(\frac{\m\Ups_\xi\Phi}{V'_{\xi-\eta}\Phi}\right)\bar\chi_{V'_{\xi-\eta}}}_{\W}& \lesssim\norm{\frac{V'_{\xi-\eta}(\m\Ups_\xi\Phi)}{V'_{\xi-\eta}\Phi}\bar\chi_{V'_{\xi-\eta}}}_{\W} + \norm{\frac{(V'_{\xi-\eta})^2\Phi}{V'_{\xi-\eta}\Phi}\bar\chi_{V'_{\xi-\eta}}}_{\W}\norm{\frac{\m\Ups_\xi\Phi}{V'_{\xi-\eta}\Phi}\bar\chi_{V'_{\xi-\eta}}}_{\W}\\
  &\lesssim 2^{k_{\max}}(1+2^{p-p_2}),
 \end{aligned}
 \end{equation}
 since by \eqref{eq:simp_symb}, \eqref{eq:vfonUpsPhi} and \eqref{eq:V'onmult}
 \begin{equation}
 \begin{aligned}
  \norm{\frac{V'_{\xi-\eta}(\m\Ups_\xi\Phi)}{V'_{\xi-\eta}\Phi}\bar\chi_{V'_{\xi-\eta}}}_{\W}&\lesssim 2^{2k_2-p_2}L^{-1}C_\m\left[(1+2^{k_1-k_2}2^{p_1-p_2})(2^p+2^{k-k_1}2^{p_1})+2^{k-k_1+p_1}\right]\\
  &\lesssim 2^{k_2-p_2}\left[2^p+2^{p_1}+2^{k_{\max}-k_2}2^{p+p_1-p_2}+2^{k-k_2}2^{2p_1-p_2}\right]\\
  &\lesssim 2^{k_{\max}}(1+2^{p-p_2}),
 \end{aligned} 
 \end{equation}
 and via \eqref{eq:vfquot'} and \eqref{eq:1ibp''}
 \begin{equation}
  \norm{\frac{(V'_{\xi-\eta})^2\Phi}{V'_{\xi-\eta}\Phi}\bar\chi_{V'_{\xi-\eta}}}_{\W}\norm{\frac{\m\Ups_\xi\Phi}{V'_{\xi-\eta}\Phi}\bar\chi_{V'_{\xi-\eta}}}_{\W} \lesssim 2^{k_1-k_2}(1+2^{p_1-p_2})\cdot 2^{k_2}(1+2^{p-p_2})\lesssim 2^{k_{\max}}(1+2^{p-p_2}).
 \end{equation}
 Hence with a direct $L^2\times L^\infty$ estimate in physical space we obtain
 \begin{equation}
  \norm{\Q_{V'_{\xi-\eta}\left(\frac{\m\Ups\Phi}{V'_{\xi-\eta}\Phi}\right)}[F,g]}_{L^2}\lesssim s^{-1}2^{k_{\max}+\frac{3}{2}k_2}\norm{F}_B\norm{g}_D,
 \end{equation}
 and similarly from \eqref{eq:1ibp''}
 \begin{equation}
  \norm{\Q_{\frac{\m\Ups\Phi}{V'_{\xi-\eta}\Phi}}[V' F,g]}_{L^2}\lesssim s^{-1}2^{k_{\max}+\frac{3}{2}k_2}\norm{V F}_B\norm{g}_D.
 \end{equation}
 Finally, for the last term in \eqref{eq:1ibpV'} we integrate by parts once more in $V_\eta$ (analogously to \eqref{eq:cl2} of Lemma \ref{lem:cl23}, with the order reversed) and obtain
 \begin{equation}
 \begin{aligned}
  \Q_{\frac{\m\Ups\Phi}{V'_{\xi-\eta}\Phi}}[F,V'_{\xi-\eta}g]&=s^{-1}\Big(\Q_{V_{\eta}\left(\frac{1}{V'_{\xi-\eta}\Phi}\frac{\m\Ups\Phi}{V_\eta\Phi}\right)}[F, V'_{\xi-\eta}g]+\Q_{\frac{1}{V'_{\xi-\eta}\Phi}\frac{\m\Ups\Phi}{V_\eta\Phi}}[V_\eta F,g]\\
  &\qquad\qquad +\Q_{\frac{1}{V'_{\xi-\eta}\Phi}\frac{\m\Ups\Phi}{V_\eta\Phi}}[F,VV'_{\xi-\eta}g]\Big).
 \end{aligned} 
 \end{equation}
 To bound the relevant multipliers we note that since $2^{p_1}\ll 2^{p_2}$ there holds that $2^{q_{\max}}\sim 1$, and we have that 
 \begin{equation}\label{eq:freq_restrict}
  2^{p_2}2^{k_1+k_2}\sim\abs{\bar\sigma}\gtrsim 2^{k_{\max}+k_{\min}}\quad\Rightarrow\quad
  \begin{cases}
   2^{k}\lesssim 2^{p_2+k_{\max}},&k=k_{\min},\\
   1\lesssim 2^{p_2},&k_2=k_{\min},
  \end{cases}
 \end{equation}
 and hence 
 \begin{equation*}
  \vert V'_{\xi-\eta}\Phi\vert \sim 2^{-2k_2+p_2}\abs{\bar\sigma}\sim 2^{-k_2+k_1+2p_2}.
 \end{equation*}
 From Lemma \ref{lem:ibp_mult_bds} it thus follows that
 \begin{equation}\label{eq:bds_2ibp'}
  \begin{aligned}
  \abs{V_{\eta}\left(\frac{1}{V'_{\xi-\eta}\Phi}\frac{\m\Ups\Phi}{V_\eta\Phi}\right)}&\lesssim \abs{\frac{1}{V'_{\xi-\eta}\Phi}}\abs{V_{\eta}\left(\frac{\m\Ups\Phi}{V_\eta\Phi}\right)}+\abs{\frac{1}{V'_{\xi-\eta}\Phi}}\abs{\frac{V_\eta V'_{\xi-\eta}\Phi}{V'_{\xi-\eta}\Phi}}\abs{\frac{\m\Ups\Phi}{V_\eta\Phi}}\\
  &\lesssim 
  (1+2^{p-p_1})\begin{cases}
   2^{k_{\max}-2p_2}+2^{k_{\max}-p_2}2^{p-p_1} ,&k=k_{\min},\\
   2^{k_{\max}}+2^{k_2}[1+2^{p_2-p_1}],&k_2=k_{\min},
  \end{cases}
 \end{aligned} 
 \end{equation}
 where we used that by \eqref{eq:bds_1ibp} and \eqref{eq:freq_restrict} there holds
 \begin{equation}
 \begin{aligned}
  \abs{\frac{1}{V'_{\xi-\eta}\Phi}}\abs{V_{\eta}\left(\frac{\m\Ups\Phi}{V_\eta\Phi}\right)}&\lesssim  2^{-2p_2}(1+2^{p-p_1})
  \begin{cases}
   2^{k_{\max}}+2^{k}2^{p-p_1},&k=k_{\min},\\
   2^{k_{\max}}+2^{k_2}[1+2^{p_2-p_1}],&k_2=k_{\min},
  \end{cases}\\
  &\lesssim (1+2^{p-p_1})
  \begin{cases}
   2^{k_{\max}-2p_2}+2^{k_{\max}-p_2}2^{p-p_1} ,&k=k_{\min},\\
   2^{k_{\max}}+2^{k_2}[1+2^{p_2-p_1}],&k_2=k_{\min},
  \end{cases}
 \end{aligned}
 \end{equation}
 and by \eqref{eq:bds_1ibp} and by (the symmetric version of) \eqref{eq:mixvf_quot} we have
 \begin{equation}
 \begin{aligned}
  \abs{\frac{1}{V'_{\xi-\eta}\Phi}}\abs{\frac{V_\eta V'_{\xi-\eta}\Phi}{V'_{\xi-\eta}\Phi}}\abs{\frac{\m\Ups\Phi}{V_\eta\Phi}}&\lesssim 2^{k_2-k_1-2p_2}\cdot 2^{2k_1-k_2}[2^{p-p_1}+2^{k-k_1}]\\
  &\lesssim 2^{k_{\max}-2p_2}[2^{p-p_1}+1].
 \end{aligned}
 \end{equation}
 With the usual set size estimates for $\Sz\lesssim 2^{\frac{k_{\max}+q_{\max}}{2}}\cdot\min\{2^{k_1+p_1},2^{k_2+p_2}\}$, (the symmetric version of) \eqref{eq:cross_recall} and since by assumption $2^{p_1-p_2}\ll 1$ we thus invoke \eqref{eq:bds_2ibp'} to deduce that
 \begin{equation}
 \begin{aligned}
  \norm{\Q_{V_{\eta}\left(\frac{1}{V'_{\xi-\eta}\Phi}\frac{\m\Ups\Phi}{V_\eta\Phi}\right)}[F, V'_{\xi-\eta}g]}_{L^2}&\lesssim 2^{\frac{5}{2}k_{\max}}2^{-p_2}(1+2^{p-p_1})\norm{F}_{L^2}\cdot [\norm{Vg}_{L^2}+(2^{p_1}+2^{p_2})\norm{\Ups g}_{L^2}] \\
  &\lesssim 2^{\frac{5}{2}k_{\max}}\norm{F}_B \cdot [\norm{Vg}_{B}+\norm{\Ups g}_{L^2}].
 \end{aligned} 
 \end{equation}
 Since moreover by \eqref{eq:bds_1ibp} we have
 \begin{equation}
  \abs{\frac{1}{V'_{\xi-\eta}\Phi}\frac{\m\Ups\Phi}{V_\eta\Phi}}\lesssim 2^{k_2-k_1-2p_2}2^{2k_{\max}-k_2}(1+2^{p-p_1})\lesssim 2^{k_{\max}}2^{-2p_2}(1+2^{p-p_1})
 \end{equation}
 the remaining two terms are direct and give by \eqref{eq:cross_recall} and $\Sz\lesssim 2^{\frac{k_{\min}}{2}}\min\{2^{k_2+p_2},2^{k_1+p_1}\}$
 \begin{equation}
 \begin{aligned}
  \norm{\Q_{\frac{1}{V'_{\xi-\eta}\Phi}\frac{\m\Ups\Phi}{V_\eta\Phi}}[V_\eta F,g]}_{L^2}&\lesssim 2^{k_{\max}}2^{-2p_2}(1+2^{p-p_1})\cdot\Sz\cdot [\norm{SV^N f}_{L^2}+(2^{p_1}+2^{p_2})\norm{\Ups F}_{L^2}]\norm{g}_{L^2}\\
  &\lesssim 2^{\frac{5}{2}k_{\max}}[\norm{V F}_{L^2}+\norm{\Ups F}_{L^2}]\cdot\norm{g}_B,
 \end{aligned}
 \end{equation}
 where we used that since $\Sz\lesssim 2^{\frac{k_{\min}}{2}}\min\{2^{k_2+p_2},2^{k_1+p_1}\}$ there holds that $2^{-2p_2}(1+2^{p-p_1})\cdot\Sz(2^{p_1}+2^{p_2})\lesssim 2^{\frac{3}{2}k_{\max}}2^{-p_2}$.
 Similarly, (using the symmetric version of Lemma \ref{lem:crossvf} with the roles of $\xi-\eta$ and $\eta$ exchanged) we obtain using $\Sz\lesssim 2^{\frac{k_{\min}}{2}}2^{k_2+p_2}$ that
 \begin{equation}
 \begin{aligned}
  \norm{\Q_{\frac{1}{V'_{\xi-\eta}\Phi}\frac{\m\Ups\Phi}{V_\eta\Phi}}[F,VV'_{\xi-\eta}g]}_{L^2} &\lesssim 2^{2k_{\max}-k_2}2^{-2p_2}(1+2^{p-p_1})\cdot\Sz\cdot\norm{F}_{L^2}\\
  &\qquad\qquad\cdot [\norm{(1,V)Sg}_{L^2}+(2^{p_1}+2^{p_2})\norm{(1,V)\Ups g}_{L^2}] \\
  &\lesssim 2^{\frac{5}{2}k_{\max}}\norm{F}_{B}\cdot [\norm{(1,V)Sg}_{B}+\norm{(1,V)\Ups g}_{L^2}],
 \end{aligned} 
 \end{equation}
 having used that by assumption $2^{p_1-p_2}\ll 1$.

\medskip
\noindent It remains to deal with \eqref{eq:allV2}:
\paragraph{Case $2^{p_2}\ll 2^{p_1}$} Here we begin with an integration by parts in $V_\eta$, which gives
\begin{equation}
 \Q_{s\m\Ups\Phi}[f,G]=\Q_{V_\eta\left(\frac{\m\Ups\Phi}{V_\eta\Phi}\right)}[f,G]+\Q_{\frac{\m\Ups\Phi}{V_\eta\Phi}}[V_\eta f,G]+\Q_{\frac{\m\Ups\Phi}{V_\eta\Phi}}[f, VG].
\end{equation}
The first and second terms are directly amenable to another integration by parts as in Lemma \ref{lem:2ibp}, so that we obtain from \eqref{eq:cl1} in Lemma \ref{lem:cl1}
\begin{equation}
 \norm{\Q_{V_\eta\left(\frac{\m\Ups\Phi}{V_\eta\Phi}\right)}[f,G]}_{L^2}\lesssim s^{-1}\cdot 2^{2k_{\max}}2^{\frac{k}{2}}\norm{(1,S)f}_B[\norm{(1,S)G}_B+\norm{\Ups G}_{L^2}],
\end{equation} 
and from \eqref{eq:cl2} in Lemma \ref{lem:cl23}
\begin{equation}
 \norm{\Q_{\frac{\m\Ups\Phi}{V_\eta\Phi}}[V_\eta f,G]}_{L^2}\lesssim s^{-1}\cdot 2^{2k_{\max}}2^{\frac{k}{2}}[\norm{(1,S)Vf}_B+\norm{\Ups V' f}_{L^2}]\cdot [\norm{(1,S)G}_B+\norm{\Ups G}_{L^2}].
\end{equation}
Finally, to treat the last term we note that if $2^p\lesssim 2^{p_1}$, we can appeal to \eqref{eq:1ibp_symb} to conclude that
\begin{equation}
\begin{aligned}
 \norm{\Q_{\frac{\m\Ups\Phi}{V_\eta\Phi}}[f, VG}_{L^2} &\lesssim s^{-1}2^{2k_{\max}+\frac{3}{2}k_1-k_2}(1+2^{p-p_1})\norm{f}_D\norm{V G}_{L^2}\\
 &\lesssim s^{-1} 2^{2k_{\max}+\frac{3}{2}k_1}\norm{f}_D\norm{V G}_{H^{-1}}.
\end{aligned}
\end{equation}
On the other hand, if $2^p\gg 2^{p_1}$, we note that we have $\abs{\bar\sigma}\sim 2^{q_{\max}+p+k+k_1}$, so that we can also bound
\begin{equation}
\begin{aligned}
 \norm{\frac{\m\Ups\Phi}{V_\eta\Phi}\bar\chi_{V_\eta}}_{\W} &\lesssim 2^{k+p_{\max}+q_{\max}}\cdot (2^p+2^{k-k_1}2^{p_1})\cdot 2^{2k_1-p_1} 2^{-q_{\max}-p-k-k_1}\lesssim 2^{k_1-p_1}(1+2^{k-k_1}2^{p_1-p})\\
 &\lesssim 2^{-p_1}(2^{k_1}+2^k).
\end{aligned} 
\end{equation}
Since $2^{-p_1}\ll 2^{-p_2}$ we can thus estimate
\begin{equation}
\begin{aligned}
 \norm{\Q_{\frac{\m\Ups\Phi}{V_\eta\Phi}}[f, V G}_{L^2} &\lesssim s^{-1}2^{-p_1}(2^{k_1}+2^k)2^{\frac{3}{2}k_1}\norm{f}_D\norm{V G}_{L^2}\\
 &\lesssim s^{-1} 2^{\frac{5}{2}k_{\max}}\norm{f}_D\norm{V G}_{B}.
\end{aligned}
\end{equation}

\paragraph{Case $2^{p_1}\lesssim 2^{p_2}$} 
Then we integrate instead by parts in $V'_{\xi-\eta}$, thus obtaining
\begin{equation}
 \Q_{s\m\Ups\Phi}[f,G]=\Q_{V'_{\xi-\eta}\left(\frac{\m\Ups\Phi}{V'_{\xi-\eta}\Phi}\right)}[f,G]+\Q_{\frac{\m\Ups\Phi}{V'_{\xi-\eta}\Phi}}[V' f,G]+\Q_{\frac{\m\Ups\Phi}{V'_{\xi-\eta}\Phi}}[f,V'_{\xi-\eta}G].
\end{equation}
Now we recall that by \eqref{eq:1ibp''} there holds that
\begin{equation}
 \norm{\frac{\m\Ups_\xi\Phi}{V'_{\xi-\eta}\Phi}\bar\chi_{V_\eta}}_{\W}\lesssim  2^{k_2}(1+2^{p-p_2}),
\end{equation}
so that we can directly estimate
\begin{equation}
 \norm{\Q_{\frac{\m\Ups\Phi}{V'_{\xi-\eta}\Phi}}[V' f,G]}_{L^2}\lesssim s^{-1}2^{\frac{3}{2}k_1+k_2}\norm{V'f}_D\norm{G}_B,
\end{equation}
and together with \eqref{eq:crossvf}
\begin{equation}
 \norm{\Q_{\frac{\m\Ups\Phi}{V'_{\xi-\eta}\Phi}}[f,V'_{\xi-\eta}G]}_{L^2}\lesssim s^{-1}2^{\frac{5}{2}k_1}(1+2^{p-p_2})\norm{f}_D\cdot[\norm{SG}_{L^2}+(2^{p_1}+2^{p_2})\norm{\Ups G}_{L^2}],
\end{equation}
which suffices since by assumption $2^{p_1-p_2}\lesssim 1$. Finally, from \eqref{eq:1ibp'''} we recall that
\begin{equation} 
 \norm{V'_{\xi-\eta}\left(\frac{\m\Ups\Phi}{V'_{\xi-\eta}\Phi}\right)\bar\chi_{V'_{\xi-\eta}}}_{\W}\lesssim 2^{k_{\max}}(1+2^{p-p_2}),
\end{equation}
from which we easily deduce that
\begin{equation}
 \norm{\Q_{V'_{\xi-\eta}\left(\frac{\m\Ups\Phi}{V'_{\xi-\eta}\Phi}\right)}[f,G]}_{L^2}\lesssim s^{-1}2^{k_{\max}}(1+2^{p-p_2})2^{\frac{3}{2}k_1}\norm{f}_D\norm{G}_{L^2}\lesssim s^{-1}2^{\frac{5}{2}k_{\max}}\norm{f}_D\norm{G}_B.
\end{equation}

\end{proof}

\section{Propagating $\Ups V^{N+3} f$ in $L^2$}\label{sec:Ups_propagation}
In this section we will demonstrate how to propagate $\Ups$ on the profiles in $L^2$.
\begin{proposition}\label{prop:Ups_prop}
Assume that $U_\pm$ are solutions to \eqref{eq:IER_disp} with initial data satisfying \eqref{eq:assump_id}, assume that for $0\le t\le T^{\ast}$,
\begin{align}
 \norm{U_\pm(t=0)}_{H^{2N_0}\cap H^{-1}}+\norm{S^aU_\pm(t=0)}_{L^2\cap H^{-1}}&\le2\eps_1,\qquad 0\le a\le 4N,\label{eq:ass_many_Sob-1}\\
 \norm{S^a \U_\pm(t)}_{B}&\le 2\eps_1,\qquad 0\le a\le 2N,\label{eq:ass_many_vfs-1}\\
 \norm{\Ups S^a \U_\pm(t)}_{L^2}&\le  2\eps_1,\qquad 0\le a\le N+3.\label{eq:ass_few_vfs-1}
\end{align}
Then for $0\leq t\leq T^\ast$ there holds that
\begin{equation}\label{eq:Ups_btstrap_bd}
 \sup_{a\le N+3}\norm{\Ups S^a \U_\pm(t)}_{L^2}\leq \varepsilon_0+C\varepsilon_1^2\ip{t}^{\frac{20}{N_0}}.
\end{equation}
In particular, for $T^\ast=\eps_1^{-\M}$ and $N_0=25\M$ we have (if $\eps_1$ is sufficiently small) that
\begin{equation}\label{eq:Ups_btstrap}
 \sup_{a\le N+3}\norm{\Ups S^a \U_\pm(t)}_{L^2}\le \eps_1.
\end{equation}
\end{proposition}

The rest of this section is devoted to the proof of \eqref{eq:Ups_btstrap_bd}. 
 First, we show how in Section \ref{ssec:red_Ups}, Lemma \ref{lem:red_Upsprop} how to reduce the claim to bilinear estimates. These are then given Lemma \ref{lem:UPS_red_split}, which is proved in Sections \ref{ssec:Upsm}--\ref{ssec:UpsPhase}. We precede these arguments with some normal form estimates that hold in the setting of Proposition \ref{prop:Ups_prop} -- see Section \ref{ssec:NFs}.

We recall here that by \eqref{eq:btstrap_D}, under the bootstrap assumptions we have that $\norm{S^a \U_\pm}_D\lesssim \eps_1$, $a\leq N$. In particular, by the dispersive estimate \eqref{eq:Linfdecay} it thus follows that
\begin{equation}\label{eq:disp_decay}
 \norm{P_{k,p,q}e^{\pm it\Lambda} S^a \U_{\pm}(t)}_{L^\infty}\lesssim 2^{\frac{3}{2}k}\ip{t}^{-1}\norm{S^a \U_\pm}_D\lesssim 2^{\frac{3}{2}k}\ip{t}^{-1}\eps_1.
\end{equation}

\subsection{Normal Form Estimates}\label{ssec:NFs}

We start with a simple bound for normal forms, that complements our $B$ norm bound obtained via the basic bilinear estimates in Section \ref{sec:B_VFprop}.
\begin{lemma}\label{lem:VNFs}
Under the assumptions of Proposition \ref{prop:Ups_prop}, we have that
\begin{align}
\norm{P_k\partial_t (S^M \U_{\pm}(t))}_{L^2}&\lesssim \ip{t}^{-1+\frac{3}{N_0}}\eps_1^2,\quad M\leq 2N,\label{eq:VNFL2}\\
\norm{P_k e^{it\Lambda}\partial_t (S^M \U_{\pm}(t))}_{L^\infty}&\lesssim \ip{t}^{-2+\frac{5}{N_0}}\eps_1^2,\quad M\leq N.\label{eq:VNFLinfty}
 \end{align}
\end{lemma}

\begin{proof}
 We have that
 \begin{equation}
  P_{k}\partial_t (S^M\U_{\pm}(t))=\sum_{\substack{\m \textnormal{ as in}\\\textnormal{Lemma \ref{lem:IERmult}}}}\sum_{k_2\leq k_1}\sum_{\substack{a+b\leq M}} \Q_\m(S^aP_{k_1}f_1,S^b P_{k_2}f_2),
 \end{equation} 
 where $f_1,f_2\in\{\U_+,\U_-\}$ are chosen in accordance with \eqref{eq:IER_disp_Duham}, with the corresponding symbols $\m$ and phase $\Phi$.
 By the set size estimates in Lemma \ref{lem:set_gain} there holds that
 \begin{equation}
 \begin{aligned}
  \norm{\Q_\m(S^aP_{k_1}f_1,S^b P_{k_2}f_2)}_{L^2}\lesssim 2^k 2^{\frac{3}{2}k_{\min}}&\min\{2^{-k_1\cdot N_0}\norm{S^aP_{k_1}f_1}_{H^{N_0}},2^{k_1}\norm{S^aP_{k_1}f_1}_{H^{-1}}\}\\
  &\cdot\min\{2^{-k_2\cdot N_0}\norm{S^aP_{k_2}f_2}_{H^{N_0}},2^{k_2}\norm{S^aP_{k_2}f_2}_{H^{-1}}\}.
 \end{aligned} 
 \end{equation}
 Since $k_1\geq k_2$ we have that $2^k\lesssim 2^{k_1}$ and $2^{k_{\max}}\lesssim 2^{k_1}$, so that with the assumption \eqref{eq:ass_many_Sob-1} the claim \eqref{eq:VNFL2} holds for $2^{k_{\max}}\gtrsim t^{\frac{1}{N_0}}$ or $2^{k_{\min}}\lesssim t^{-\frac{1}{2}}$, where $k_{\min}=k$ or $k_{\min}=k_2$. Furthermore, localizing $f_j=P_{k_j,p_j,q_j}f_j$, $j=1,2$, we have that
 \begin{equation}
 \begin{aligned}
  \norm{\Q_\m(S^a f_1,S^b f_2)}_{L^2}\lesssim \sum_{k_j,p_j,q_j,j=1,2}2^k 2^{\frac{3}{2}k_{\min}}&2^{p_1+\frac{q_1}{2}}\norm{S^af_1}_{B}\cdot 2^{p_2+\frac{q_2}{2}}\norm{S^b f_2}_{B},
 \end{aligned} 
 \end{equation}
 which gives \eqref{eq:VNFL2} if $2^{\min_{j=1,2}\{p_j\}}\leq t^{-1}$ or $2^{\min_{j=1,2}\{q_j\}}\leq t^{-2}$. 

 Since we have that 
 \begin{equation*}
  \norm{\Q_\m(S^aP_{k_1}f_1,S^b P_{k_2}f_2)}_{L^\infty}\lesssim\norm{\mathcal{F}\Q_\m(S^aP_{k_1}f_1,S^b P_{k_2}f_2)}_{L^1}\lesssim 2^{\frac{3}{2}k}\norm{\Q_\m(S^aP_{k_1}f_1,S^b P_{k_2}f_2)}_{L^2},
 \end{equation*}
 these reductions are completely analogous for the $L^\infty$ estimate. 
 
 Thus we can conclude the proof of the claim by estimating that for 
 \begin{equation}\label{eq:some_restrict}
  t^{-\frac{1}{2}}\lesssim 2^{k_j}\lesssim t^{\frac{1}{N_0}},\quad t^{-1}\lesssim 2^{p_j},\quad t^{-2}\lesssim 2^{q_j},\quad j=1,2,
 \end{equation}
 there holds with $\norm{\m}_{\W}\lesssim 2^k$ that
 \begin{equation}
 \begin{aligned}
  \norm{\Q_\m(S^a f_1,S^b f_2)}_{L^2}&\lesssim 2^k\norm{S^{\leq M} f_1}_{L^2}\norm{e^{\pm it\Lambda}S^{\leq\frac{M}{2}}f_2}_{L^\infty}\\
  &\lesssim \ip{t}^{-1}2^{\frac{5}{2}k}\norm{S^{\leq M} f_1}_{L^2}\norm{e^{\pm it\Lambda}S^{\leq\frac{M}{2}}f_2}_{D}\lesssim \ip{t}^{-1+\frac{5}{2N_0}}\eps_1^2,
 \end{aligned} 
 \end{equation}
 having used the dispersive decay estimate \eqref{eq:disp_decay}, and similarly that for $M\leq N$
 \begin{equation}
 \begin{aligned}
  \norm{e^{it\Lambda}\Q_\m(S^a f_1,S^b f_2)}_{L^\infty}&\lesssim 2^k\norm{e^{\pm it\Lambda}S^{\leq M} f_1}_{L^\infty}\norm{e^{\pm it\Lambda}S^{\leq\frac{M}{2}}f_2}_{L^\infty}\\
  &\lesssim \ip{t}^{-2}2^{4k}\norm{S^{\leq M} f_1}_{D}\norm{S^{\leq\frac{M}{2}}f_2}_{D}\lesssim \ip{t}^{-2+\frac{4}{N_0}}\eps_1^2.
 \end{aligned} 
 \end{equation}
\end{proof}

Next we give a decomposition lemma for the normal forms involving $\Ups$:
\begin{lemma}\label{lem:UpsNF}
 Under the assumptions of Proposition \ref{prop:Ups_prop}, we have that 
\begin{equation}
 \norm{P_{k,p,q}e^{\pm it\Lambda}\partial_t (\Ups S^{M}) \U_{\pm}(t)}_{L^\infty}\lesssim \ip{t}^{-1+\frac{8}{N_0}},\qquad M\leq N.
\end{equation}
\end{lemma}
\begin{proof}
With a suitable choice of $f_j\in\{\U_+,\U_-\}$ as in \eqref{eq:IER_disp_Duham}, we localize $f_j=P_{k_j,p_j,q_j}f_j$, $j=1,2$ and then have (up to summation over symbols $\m$ as in Lemma \ref{lem:IERmult}) that
\begin{equation}
\begin{aligned}
   e^{\pm it\Lambda}\partial_t \Ups S^{M} \U_{\pm}(t)&=\sum_{k_j,p_j,q_j,\,j=1,2}\sum_{\substack{a+b\leq M}} e^{\pm it\Lambda}\partial_t \Ups_\xi B_\m(S^af_1,S^bf_2)\\
  &=\sum_{k_j,p_j,q_j,\,j=1,2}\sum_{\substack{a+b\leq M}}e^{\pm it\Lambda}\Q_{\Ups_\xi\m}(S^af_1,S^bf_2)+e^{\pm it\Lambda}\Q_{\m}(\Ups_\xi S^af_1,S^bf_2)\\
  &\hspace{5cm} +e^{\pm it\Lambda}i\Q_{t\m \Ups_\xi\Phi}(S^af_1,S^bf_2),
\end{aligned} 
\end{equation}
and we make the convention that $k_2\leq k_1$. Next we define $A^j_M$, $j=1,2$, via
\begin{align}
 e^{\mp it\Lambda}A^1_M(\U_{\pm};t)&:=P_{k,p,q}\sum_{\substack{k_j,p_j,q_j,\\j=1,2}} \sum_{\substack{a+b\leq M}}\Big[\mathfrak{1}_{\{2^{p_2}\lesssim 2^{p_1}\}}\Q_{\Ups_\xi\m}(S^af_1,S^bf_2)+i\Q_{t\m \Ups_\xi\Phi}(S^af_1,S^bf_2)\Big],\label{eq:A1M}\\
 e^{\mp it\Lambda}A^2_M(\U_{\pm};t)&:=P_{k,p,q}\sum_{\substack{k_j,p_j,q_j,\\j=1,2}}\sum_{\substack{a+b\leq M}}\Big[\mathfrak{1}_{\{2^{p_1}\ll 2^{p_2}\}}\Q_{\Ups_\xi\m}(S^af_1,S^bf_2)+\Q_{\m}(\Ups_\xi S^af_1,S^bf_2)\Big].\label{eq:A2M}
\end{align}
To estimate $A^1_M$, we recall the bounds \eqref{eq:Ups_mult}, \eqref{eq:simp_symb} and \eqref{eq:UpsPhi_bd}, namely
\begin{equation}\label{eq:Ups_mult_recall}
 \abs{\Ups_\xi\m}\lesssim 2^{k}[1+2^{q+p_2-p_1}],\qquad \abs{\m\Ups_\xi\Phi}\lesssim 2^{k_{\max}}.
\end{equation}
Using that $\norm{A_M^1}_{L^\infty}\lesssim 2^{\frac{3}{2}k}\norm{A_M^1}_{L^2}$, as in the proof of the above Lemma \ref{lem:VNFs}, thanks to the energy and $B$ norm estimates we can reduce in analogy with \eqref{eq:some_restrict} to showing that if $m\in\N$ is such that $2^m\leq t\leq 2^{m+1}$, then there holds that
\begin{equation}\label{eq:A1_red}
 \sup_{\substack{k_j\in(-2m,\frac{m}{N_0}),\\p_j\in(-m,0],q_j\in(-2m,0]}}\norm{\mathfrak{1}_{\{2^{p_2}\lesssim 2^{p_1}\}}\Q_{\Ups_\xi\m}(S^af_1,S^bf_2)}_{L^\infty}+\norm{\Q_{t\m \Ups_\xi\Phi}(S^af_1,S^bf_2)}_{L^\infty}\lesssim 2^{(-1+\frac{4}{N_0})m}\eps_1^2.
\end{equation}
This now follows directly from the dispersive estimate \eqref{eq:Linfdecay} and \eqref{eq:Ups_mult_recall}: With the frequency localizations of \eqref{eq:A1_red}, the two terms in $A^1_M$ can be bounded as
\begin{equation}
\begin{aligned}
 \norm{\mathfrak{1}_{\{2^{p_2}\lesssim 2^{p_1}\}} \Q_{\Ups_\xi\m}(S^af_1,S^bf_2)}_{L^\infty}&\lesssim 2^{k_{\max}}\norm{e^{it\Lambda}S^af_1}_{L^\infty}\norm{e^{it\Lambda}S^bf_2}_{L^\infty}\\
 &\lesssim t^{-2}2^{4k_{\max}}\norm{S^af_1}_{D}\norm{S^bf_2}_{D}\lesssim t^{-2+\frac{4}{N_0}}\eps_1^2,
\end{aligned} 
\end{equation}
and similarly
\begin{equation}
 \norm{\Q_{t\m \Ups_\xi\Phi}(S^af_1,S^bf_2)}_{L^\infty}\lesssim t^{-1}2^{4k_{\max}}\norm{S^af_1}_{D}\norm{S^bf_2}_{D}\lesssim t^{-1+\frac{4}{N_0}}\eps_1^2.
\end{equation}
On the other hand, for the terms in $A^2_M$ we will estimate in $L^2$ and use that $\norm{A_M^2}_{L^\infty}\lesssim 2^{\frac{3}{2}k}\norm{A_M^2}_{L^2}$. We re-sum the first argument of the second one in \eqref{eq:A2M}, which then equals
\begin{equation}
 \sum_{k_2,p_2,q_2}\sum_{\substack{a+b\leq M}}\Q_{\m}(\sum_{k_1\geq k_2}\sum_{p_1,q_1}\Ups_\xi S^af_1,S^bf_2).
\end{equation}
Writing $\bar{f}_1=\sum_{k_1\geq k_2}\sum_{p_1,q_1}f_1$, as above we can reduce to considering the frequency restrictions as in \eqref{eq:A1_red}, and it thus suffices to show that
\begin{equation}
\begin{aligned}
 \norm{\Q_{\m}(\Ups_\xi S^a\bar{f}_1,S^bf_2)}_{L^2}&\lesssim 2^{k_{\max}} \norm{\Ups_\xi S^a\bar{f}_1}_{L^2}\norm{e^{it\Lambda}S^bf_2}_{L^\infty}\\
 &\lesssim t^{-1}2^{\frac{5}{2}k_{\max}}\left[\norm{S^{a+1}\bar{f}_1}_{L^2}+\norm{\Ups S^a\bar{f}_1}_{L^2}\right]\cdot\norm{S^bf_2}_{D}\\
 &\lesssim \ip{t}^{-1+\frac{3}{N_0}}\eps_1^2.
\end{aligned}
\end{equation}
where we used \eqref{eq:Ups_cross}. 

Finally, for the first term in \eqref{eq:A2M} we invoke \eqref{eq:Ups_mult_recall} and Lemma \ref{lem:set_gain} with set size bound $2^{k_1+p_1}2^{\frac{k_2+q_2}{2}}$ to see that
\begin{equation}
\begin{aligned}
 \norm{\Q_{\Ups_\xi\m}(S^af_1,S^bf_2)}_{L^2}\lesssim 2^{k+k_1+\frac{k_2}{2}}&\min\{2^{-k_1\cdot N_0}\norm{S^af_1}_{H^{N_0}},2^{k_1}\norm{S^af_1}_{H^{-1}},2^{p_1+\frac{q_1}{2}}\norm{S^af_1}_{B}\}\\
 &\cdot \min\{2^{-k_2\cdot N_0}\norm{S^bf_2}_{H^{N_0}},2^{k_2}\norm{S^bf_2}_{H^{-1}},2^{p_2+\frac{q_2}{2}}\norm{S^bf_2}_{B}\}.
\end{aligned} 
\end{equation}
This shows that for $m\in\N$ such that $2^m\leq t\leq 2^{m+1}$ and $J_m:=(-2m,\frac{2m}{N_0})\times (-2m,0]\times(-4m,0]$ we have
\begin{equation}
\begin{aligned}
 \sum_{(k_j,p_j,q_j)\in (J_m)^c,\, j=1,2}\norm{\Q_{\Ups_\xi\m}(S^af_1,S^bf_2)}_{L^2}\lesssim \ip{t}^{-1+\frac{6}{N_0}}\eps_1^2.
\end{aligned}
\end{equation}
Hence it suffices to show that
\begin{equation}
 \sup_{(k_j,p_j,q_j)\in (J_m),\, j=1,2}\norm{\Q_{\Ups_\xi\m}(S^af_1,S^bf_2)}_{L^2}\lesssim \ip{t}^{-1+\frac{5}{N_0}}\eps_1^2,
\end{equation}
which we will show by an integration by parts in $V_{\xi-\eta}$. To this end, we note that if $2^{p_1}\ll2^{p_2}$ then $\abs{\bar\sigma}\sim 2^{p_2}2^{k_1+k_2}$, and thus by \eqref{eq:vflobound},
\begin{equation}\label{eq:vfsize_recall}
 \abs{V_{\xi-\eta}\Phi}\sim 2^{k_1-k_2+2p_2},
\end{equation}
and we may integrate by parts in $V_{\xi-\eta}$ to obtain
\begin{equation}\label{eq:Upsmult-ibp'}
 \Q^{V_{\xi-\eta}}_{\Ups_\xi\m}(S^af_1,S^bf_2)=t^{-1}\left[\Q^{V_{\xi-\eta}}_{V_{\xi-\eta}\left(\frac{\Ups_\xi\m}{V_{\xi-\eta}\Phi}\right)}(S^af_1,S^bf_2)+\Q^{V_{\xi-\eta}}_{\frac{\Ups_\xi\m}{V_{\xi-\eta}\Phi}}(V_{\xi-\eta}S^af_1,S^bf_2)+\Q^{V_{\xi-\eta}}_{\frac{\Ups_\xi\m}{V_{\xi-\eta}\Phi}}(S^af_1,S^{b+1}f_2)\right].
\end{equation}
To estimate these terms we note that by \eqref{eq:Ups_mult_recall}, \eqref{eq:vfsize_recall} and $2^{p_1}\ll2^{p_2}$ there holds that
\begin{equation}\label{eq:Upsmfracphi}
 \abs{\frac{\Ups_\xi\m}{V_{\xi-\eta}\Phi}}\lesssim 2^{k_2}2^{-p_2-p_1},
\end{equation}
and we compute that 
\begin{equation}\label{eq:VUpsmfracphi}
 \abs{V_{\xi-\eta}\left(\frac{\Ups_\xi\m}{V_{\xi-\eta}\Phi}\right)}\lesssim\abs{\frac{V_{\xi-\eta}^2\Phi}{V_{\xi-\eta}\Phi}}\abs{\frac{\Ups_\xi\m}{V_{\xi-\eta}\Phi}}+\abs{\frac{V_{\xi-\eta}(\Ups_\xi\m)}{V_{\xi-\eta}\Phi}}\lesssim 2^{k_{\max}}[1+2^{-2p_1}]
\end{equation}
as follows: By \eqref{eq:vfquot'} there holds
\begin{equation}
 \abs{\frac{V_{\xi-\eta}^2\Phi}{V_{\xi-\eta}\Phi}}\abs{\frac{\Ups_\xi\m}{V_{\xi-\eta}\Phi}}\lesssim 2^{k_1-k_2}(1+2^{p_1-p_2})\cdot 2^{k_2}2^{-p_2-p_1}\lesssim 2^{k_{\max}}[1+2^{-p_2-p_1}],
\end{equation}
and we recalled that by Lemma \ref{lem:VUpsmult}
\begin{equation}
 \abs{V_{\xi-\eta}\Ups_\xi\m}\lesssim C_{\Ups\m} 2^{k_1-k_2}[1+2^{p_1-p_2}]\lesssim 2^{k+k_1-k_2}(1+2^{p_2-p_1})(1+2^{p_1-p_2})\lesssim 2^{2k_{\max}-k_2}(1+2^{p_2-p_1}),
\end{equation}
so that
\begin{equation}
 \abs{\frac{V_{\xi-\eta}(\Ups_\xi\m)}{V_{\xi-\eta}\Phi}}\lesssim 2^{k_{\max}}2^{-2p_2}(1+2^{p_2-p_1}).
\end{equation}
For the terms in \eqref{eq:Upsmult-ibp'} we thus have the bounds
\begin{equation}
\begin{aligned}
 \norm{\Q^{V_{\xi-\eta}}_{V_{\xi-\eta}\left(\frac{\Ups_\xi\m}{V_{\xi-\eta}\Phi}\right)}(S^af_1,S^bf_2)}_{L^2}&\lesssim 2^{k_{\max}}[1+2^{-2p_1}]\norm{S^af_1}_{L^2}\norm{S^bf_2}_{L^2}\cdot 2^{k_1+p_1}2^{\frac{k_{\max}}{2}}\\
 &\lesssim 2^{\frac{5}{2}k_{\max}}\norm{S^a f_1}_B\norm{S^bf_2}_{L^2}\lesssim \ip{t}^{\frac{5}{N_0}}\eps_1^2,
\end{aligned} 
\end{equation}
and similarly
\begin{equation}
 \norm{\Q^{V_{\xi-\eta}}_{\frac{\Ups_\xi\m}{V_{\xi-\eta}\Phi}}(S^af_1,S^{b+1}f_2)}_{L^2}\lesssim 2^{\frac{5}{2}k_{\max}}\norm{S^af_1}_{L^2}\norm{S^{b+1}f_2}_{B}\lesssim \ip{t}^{\frac{5}{N_0}}\eps_1^2,
\end{equation}
as well as by the symmetric version of \eqref{eq:cross_recall}
\begin{equation}
\begin{aligned}
 \norm{\Q^{V_{\xi-\eta}}_{\frac{\Ups_\xi\m}{V_{\xi-\eta}\Phi}}(V_{\xi-\eta}S^af_1,S^bf_2)}_{L^2}&\lesssim 2^{k_2-p_2-p_1}2^{k_1-k_2}[\norm{S^{a+1} f_1}_{L^2}+(2^{p_2}+2^{p_1})\norm{\Ups S^a f_1}_{L^2}]\norm{S^b f_2}_{L^2}\\
 &\qquad \qquad \cdot \min\{2^{k_2+p_2},2^{k_1+p_1}\}2^{\frac{k_{\max}}{2}}\\
 &\lesssim 2^{\frac{5}{2}k_{\max}}[\norm{S^{a+1}f_1}_B+\norm{\Ups S^af_1}_{L^2}]\norm{S^bf_2}_B\lesssim \ip{t}^{\frac{5}{N_0}}\eps_1^2.
\end{aligned} 
\end{equation}
\end{proof}

\subsection{Reduction to Bilinear Estimates}\label{ssec:red_Ups}

To control $\Ups S^{N+3}\U_\pm$, we note that by construction the vector fields distribute across the nonlinearity (see \eqref{eq:vf_distribute} and \eqref{eq:vfsumonmult}). From \eqref{eq:IER_disp_Duham} we then see that
\begin{equation}
\begin{aligned}
 \Ups_\xi S^{N+3}f&=\Ups_\xi S^{N+3}f_0+\sum_{\substack{\m \textnormal{ as in}\\\textnormal{Lemma \ref{lem:IERmult}}}}\sum_{\substack{a+b\leq N+3}}\Ups_\xi B_\m(S^af_1,S^bf_2)
\end{aligned} 
\end{equation}
where $f,f_1,f_2\in\{\U_+,\U_-\}$ are suitably chosen in accordance with \eqref{eq:IER_disp_Duham}, $f_0$ are the corresponding initial data and $B_\m$ includes the corresponding phase functions $\Phi=\pm\Lambda(\xi)\pm\Lambda(\xi-\eta)\pm\Lambda(\eta)$. The initial data are controlled in $L^2$ by assumption. To bound the bilinear terms we localize in frequency as $f=\sum_k P_k f$ and $f_j=\sum_{k_j}P_{k_j}f_j$, $j=1,2$, and note that we may arrange the summation in such a way that $k_2\leq k_1$, obtaining that 
\begin{equation}
\begin{aligned}
 \Ups_\xi B_\m(S^af_1,S^bf_2)=\sum_{k}\sum_{k_2\leq k_1}\,  P_k\Ups_\xi B_\m(S^aP_{k_1}f_1,S^bP_{k_2}f_2).
\end{aligned} 
\end{equation}
Abbreviating $f_j=P_{k_j}f_j$, $j=1,2$, it thus suffices to control appropriately terms of the form
\begin{equation}\label{eq:Upsterms_k}
\begin{aligned}
 P_k\Ups_\xi B_\m(S^af_1,S^bf_2)&=P_kB_{\Ups_\xi\m}(S^af_1,S^bf_2)+P_kB_{\m}(\Ups_\xi S^af_1,S^bf_2)+iP_kB_{s\m \Ups_\xi\Phi}(S^af_1,S^bf_2),\\
 a+b&\leq N+3,\quad k_2\leq k_1.
\end{aligned} 
\end{equation}
Next we show that if the frequencies involved are sufficiently large (or small), these terms can be bounded in $L^2$ by more crude arguments.

\begin{lemma}\label{lem:red_Upsprop}
Under the assumptions of Proposition \ref{prop:Ups_prop}, let $m\in\N$ be such that $2^m\leq t\leq 2^{m+1}$, and localize $f=P_{k,p,q}f$ and $f_j=P_{k_j,p_j,q_j}$, $j=1,2$. Let $0\le a+b\le N+3$, as above. To prove Proposition \ref{prop:Ups_prop} it suffices to show that
\begin{equation}\label{eq:UPS_red}
 \sup_{\substack{k_1,k_2\in (-2(1+\delta)m,2m/N_0)\\p_1,p_2\in (-2m,0]\\q_1,q_2\in (-4m,0]}}\norm{ \int_{s=0}^t\tau_m(s)\Ups_\xi \Q_\m(S^af_1,S^bf_2)(s)ds}_{L^2}\lesssim \varepsilon_1^2 2^{\frac{10}{N_0}\cdot m}.
\end{equation}
\end{lemma}

Clearly, the restrictions in \eqref{eq:UPS_red} also imply bounds for $k,p,q$  by the triangle inequality.

The proof shows how to obtain the reduction of \eqref{eq:Ups_btstrap_bd} to \eqref{eq:UPS_red}. This will be done by estimating separately the terms in \eqref{eq:Upsterms_k}. For the first and last terms in \eqref{eq:Upsterms_k} this proceeds as in the proof of Lemma \ref{lem:red_VFprop}. However, for the second term an extra argument is needed, since we do not control $H^{N_0}$ energy estimates of $\Ups S^a f_1$.

\begin{proof}[Proof of Lemma \ref{lem:red_Upsprop}]
We note first that since $k_2\leq k_1$, to prove the frequency restrictions it suffices to bound $k_1$ from above and $k_2$ and $k$ from below, and we have that
\begin{equation}
 2^{k-k_1}\lesssim 1.
\end{equation}
Furthermore, by the triangle inequality there holds that $2^{k+p}\lesssim 2^{k_1+p_1}+2^{k_2+p_2}$ and $2^{k+q}\lesssim 2^{k_1+q_1}+2^{k_2+q_2}$, so that it is enough to prove that with
\begin{equation}
\begin{aligned}
 J:=\{(k,p,q),(k_j,p_j,q_j),j=1,2:&\;k_2\leq k_1,\,k_1,k_2\not\in (-2(1+\delta)m,2m/N_0),\\
 &\hbox{ or }\min\{p_1,p_2\}\leq -2m,\hbox{ or }\min\{q_1,q_2\}\leq -4m,\}
\end{aligned}
\end{equation}
we have the following bounds for the terms in \eqref{eq:Upsterms_k}:
\begin{equation}\label{eq:Upssum_red}
\begin{aligned}
 &\sum_{J}\norm{ \int_{s=0}^t\tau_m(s)P_{k,p,q}\A(s)ds}_{L^2}\lesssim \varepsilon_1^2 \cdot 2^{\frac{20}{N_0}\cdot m},\\
 &\qquad \A\in \{\Q_{\Ups_\xi\m}(S^af_1,S^bf_2),\Q_{\m}(\Ups_\xi S^af_1,S^bf_2),\Q_{s\m \Ups_\xi\Phi}(S^af_1,S^bf_2)\},
\end{aligned} 
\end{equation}
where $f_j=P_{k_j,p_j,q_j}$, $j=1,2$, and $0\le a+b\le N+3$.

To bound the last term in \eqref{eq:Upssum_red} we recall that by \eqref{eq:UpsPhi_bd} there holds that
\begin{equation}
  \abs{\m\Ups_\xi\Phi}\lesssim 2^k,
\end{equation}
so that we have the estimate
\begin{equation}
\begin{aligned}
 &\norm{P_k\Q_{s\m \Ups_\xi\Phi}(S^af_1,S^bf_2)}_{L^2}\lesssim s2^{\frac{3}{2}k_{\min}}\norm{S^af_1}_{L^2}\norm{S^bf_2}_{L^2}\\
 &\qquad\lesssim s2^{\frac{3}{2}k_{\min}}\min\{2^{-k_1\cdot N_0}\norm{S^af_1}_{H^{N_0}},2^{k_1}\norm{S^af_1}_{H^{-1}}\}\cdot\min\{2^{-k_2\cdot N_0}\norm{S^bf_2}_{H^{N_0}},2^{k_2}\norm{S^bf_2}_{H^{-1}}\}.
\end{aligned}
\end{equation}
We see that this gives an acceptable contribution if $2^{-k_1\cdot N_0}\lesssim 2^{-2m}$ or if $2^{\frac{3}{2}k_{\min}}\lesssim 2^{-2m}$. Furthermore, localizing also in $p,q$ and using the set size bound from Lemma \ref{lem:set_gain} we obtain
\begin{equation}
\begin{aligned}
 &\norm{P_{k,p,q}\Q_{s\m \Ups_\xi\Phi}(S^a P_{k_1,p_1,q_1}f_1,S^bP_{k_2,p_2,q_2}f_2)}_{L^2}\lesssim 2^{\frac{3}{2}k+p+\frac{q}{2}}2^{p_1+\frac{q_1}{2}}\norm{S^a P_{k_1,p_1,q_1}f_1}_{B} 2^{p_2+\frac{q_2}{2}}\norm{S^b P_{k_2,p_2,q_2}f_2}_{B},
\end{aligned}
\end{equation}
which gives an acceptable contribution if $2^{p_{\min}+\frac{q_{\min}}{2}}\lesssim 2^{-2m}$.

The analogous arguments apply for the first term of \eqref{eq:Upssum_red}. We localize directly as $f_j=P_{k_j,p_j,q_j}$, $j=1,2$, and recall that by \eqref{eq:Ups_mult} there holds that
\begin{equation}
 \abs{\Ups_\xi\m}\lesssim 2^{k}[1+2^{q+p_2-p_1}].
\end{equation}
If $2^{p_1}\gtrsim 2^{q+p_2}$ we can thus proceed as above, and otherwise we estimate that
\begin{equation}
\begin{aligned}
 \norm{P_{k,p,q}\Q_{\Ups_\xi\m}(S^af_1,S^bf_2)}_{L^2}&\lesssim 2^{k+q+p_2-p_1}\cdot 2^{k_1+p_1}\min\{2^{\frac{k+q}{2}},2^{\frac{k_2+q_2}{2}}\}\cdot\min\{2^{k_2}\norm{S^bf_2}_{H^{-1}},2^{p_2+\frac{q_2}{2}}\norm{S^bf_2}_B\}\\
 &\qquad\cdot\min\{2^{-k_1\cdot N_0}\norm{S^af_1}_{H^{N_0}},2^{p_1+\frac{q_1}{2}}\norm{S^af_1}_B\},
\end{aligned}
\end{equation}
which gives an acceptable contribution to \eqref{eq:Upssum_red}.

The second term of \eqref{eq:Upssum_red} is more involved, since we do not have energy estimates available for $\Ups S^a f_1$. We note that by \eqref{eq:Ups_cross} there holds that (since $f_1$ is axisymmetric)
\begin{equation}\label{eq:Ups_cross_axisymm'}
\begin{aligned}
 \Ups_\xi S^a f_1(\xi-\eta)&=\Gamma_S\cdot (S^{a+1} f_1)(\xi-\eta)+\Gamma_\Ups \cdot (\Ups S^a f_1)(\xi-\eta),\\
 \Gamma_S&=\frac{\abs{\xi}}{\abs{\xi-\eta}}\left[\omega_c\Lambda(\xi)\sqrt{1-\Lambda^2(\xi-\eta)}-\sqrt{1-\Lambda^2(\xi)}\Lambda(\xi-\eta)\right],\\
 \Gamma_\Ups&=\frac{\abs{\xi}}{\abs{\xi-\eta}}\left[\omega_c\Lambda(\xi)\Lambda(\xi-\eta)+\sqrt{1-\Lambda^2(\xi)}\sqrt{1-\Lambda^2(\xi-\eta)}\right],
\end{aligned} 
\end{equation}
and thus
\begin{equation}\label{eq:2ndterm-split'}
 \Q_{\m}(\Ups_\xi S^af_1,S^bf_2)=Q_{\m}(\Gamma_S S^{a+1}f_1,S^bf_2)+Q_{\m}(\Gamma_\Ups \Ups S^af_1,S^bf_2).
\end{equation}
We note that $\abs{\Gamma_S}+\abs{\Gamma_\Ups}\lesssim 1$, so that the first term in \eqref{eq:2ndterm-split'} can be estimated by set size and energy estimates just as above: we have
\begin{equation}\label{eq:2ndterm-split-simple}
\begin{aligned}
 \norm{B_{\m}(\Gamma_S S^{a+1}f_1,S^bf_2)}_{L^2}\lesssim s\cdot 2^k\cdot 2^{\frac{3}{2}k_{\min}}\cdot &\min\{2^{-k_1\cdot N_0}\norm{S^{a+1} f_1}_{H^{N_0}},2^{k_1}\norm{S^{a+1} f_1}_{H^{-1}}\}\\
 &\cdot\min\{2^{-k_2\cdot N_0}\norm{S^b f_2}_{H^{N_0}},2^{k_2}\norm{S^b f_2}_{H^{-1}}\},
\end{aligned} 
\end{equation}
which gives an acceptable contribution.

For the second term in \eqref{eq:2ndterm-split'} we have that
\begin{equation}
 \norm{Q_{\m}(\Gamma_\Ups \Ups S^af_1,S^bf_2)}_{L^2}\lesssim  2^k\cdot 2^{\frac{3}{2}k_{\min}}\norm{\Ups S^a f_1}_{L^2}\min\{2^{-k_2\cdot N_0}\norm{S^b f_2}_{H^{N_0}},2^{k_2}\norm{S^b f_2}_{H^{-1}}\},
\end{equation}
which gives an acceptable contribution if $2^{k_2}\gtrsim (2^ks)^{\frac{1}{N_0}}$ or $2^{k_2}\lesssim (2^ks)^{-\frac{2}{5}}$. We may thus assume that
\begin{equation}
 (2^ks)^{-\frac{2}{5}}\lesssim 2^{k_2}\lesssim (2^ks)^{\frac{1}{N_0}}.
\end{equation}
We localize the inputs in $p_j,q_j$, $j=1,2$, and with \eqref{eq:ups} we then write 
\begin{equation}\label{eq:2ndterm-split2}
 \Q_{\m}(\Gamma_\Ups \Ups S^af_1,S^bf_2)=\Q_{\m}(\Gamma_\Ups \frac{1}{\sqrt{1-\Lambda^2}} S^{a+1}f_1,S^bf_2)-\Q_{\m}(\Gamma_\Ups \frac{1}{\sqrt{1-\Lambda^2}}\abs{\xi-\eta}\partial_{\xi_3-\eta_3} S^af_1,S^bf_2).
\end{equation}
For the first term we have that
\begin{equation}
\begin{aligned}
 \norm{\Q_{\m}(\Gamma_\Ups \frac{1}{\sqrt{1-\Lambda^2}} S^{a+1}f_1,S^bf_2)}_{L^2}&\lesssim 2^k\cdot 2^{p_1+k_1}2^{\frac{k_2}{2}}\min\{2^{-p_1}2^{-k_1\cdot N_0}\norm{S^{a+1} f_1}_{H^{N_0}},2^{\frac{q_1}{2}}\norm{S^{a+1} f_1}_B\}\norm{S^b f_2}_{L^2}\\
 &\lesssim 2^k\cdot 2^{k_1+\frac{k_2}{2}}\min\{2^{-k_1\cdot N_0}\norm{S^{a+1} f_1}_{H^{N_0}},2^{p_1+\frac{q_1}{2}}\norm{S^{a+1} f_1}_B\}\norm{S^b f_2}_{L^2},
\end{aligned} 
\end{equation}
which gives an acceptable contribution to \eqref{eq:Upssum_red}.

For the second term in \eqref{eq:2ndterm-split2} we have by \eqref{eq:Ups_cross_axisymm'} that
\begin{equation}
 \bar\Gamma_\Ups:=\Gamma_\Ups \frac{1}{\sqrt{1-\Lambda^2}}\abs{\xi-\eta}=\frac{\abs{\xi}}{\sqrt{1-\Lambda^2(\xi-\eta)}}\left[\omega_c\Lambda(\xi)\Lambda(\xi-\eta)+\sqrt{1-\Lambda^2(\xi)}\sqrt{1-\Lambda^2(\xi-\eta)}\right],
\end{equation}
with $\abs{\bar\Gamma_\Ups}\lesssim 2^{k-p_1}$, and we integrate by parts in $\eta_3$ to obtain that
\begin{equation}\label{eq:2ndterm-split-ibp}
\begin{aligned}
 \Q_{\m}(\bar\Gamma_\Ups\partial_{\xi_3-\eta_3} S^af_1,S^bf_2)&
 =\int_\eta \widehat{S^a f_1}(\xi-\eta)\partial_{\eta_3}\left(e^{is\Phi}\m \cdot \bar\Gamma_\Ups\right) \widehat{S^b f_2}(\eta) d\eta\\ 
 &\qquad +\int_\eta \widehat{S^a f_1}(\xi-\eta)e^{is\Phi}\m \cdot \bar\Gamma_\Ups \cdot\partial_{\eta_3}\widehat{S^b f_2}(\eta)d\eta.
\end{aligned} 
\end{equation}
We now compute that
\begin{equation}\label{eq:dGamUps}
 \abs{\partial_{\eta_3}\left(e^{is\Phi}\m \cdot \bar\Gamma_\Ups\right)}\lesssim s 2^{2k-k_2-p_1}
\end{equation}
as follows: We directly have that $\abs{\partial_{\eta_3}\left(e^{is\Phi}\m\right)}\lesssim 2^{k-k_2}$. Moreover, using that $\partial_{\eta_3}\omega_c=0$ and 
\begin{equation}
 \partial_{\eta_3}\left(\sqrt{1-\Lambda^2(\xi-\eta)}\right)=\frac{\Lambda(\xi-\eta)}{\sqrt{1-\Lambda^2(\xi-\eta)}}\cdot \partial_{\eta_3}\Lambda(\xi-\eta)=\frac{\Lambda(\xi-\eta)}{\sqrt{1-\Lambda^2(\xi-\eta)}}\frac{1-\Lambda^2(\xi-\eta)}{\abs{\xi-\eta}},
\end{equation}
we obtain $\abs{\partial_{\eta_3}\bar\Gamma_\Ups}\lesssim 2^{k-k_2-p_1}$.

Hence by \eqref{eq:dGamUps} the first term in \eqref{eq:2ndterm-split-ibp} can be dealt with with a mix of energy and $B$ norm estimates, as before in \eqref{eq:2ndterm-split-simple}. For the second term we use \eqref{eq:dxi3} to obtain that this equals
\begin{equation}\label{eq:2ndterm-split-ibp-split}
\begin{aligned}
 \Q_{\m\bar\Gamma_\Ups}(S^a f_1,\abs{\eta}^{-1}\Lambda(\eta) S^{b+1} f_2)+\Q_{\m\bar\Gamma_\Ups}(S^a f_1,\abs{\eta}^{-1}\sqrt{1-\Lambda^2(\eta)}\Ups_\eta S^b f_2).
\end{aligned}
\end{equation}
Using that by assumption there holds that $2^{-k_2}\lesssim (s2^k)^{\frac{2}{5}}$, we can control this as before, having now access to the energy estimates and $B$ norm on $S^a f_1$: there holds that
\begin{equation}
\begin{aligned}
 \norm{\Q_{\m\bar\Gamma_\Ups}(S^a f_1,\abs{\eta}^{-1}\sqrt{1-\Lambda^2(\eta)}\Ups_\eta S^b f_2)}_{L^2}\lesssim (s2^k)^{\frac{2}{5}}&\cdot 2^{p_1+k_1}2^{k_2}\cdot 2^{k-p_1}\\
 &\cdot \min\{2^{-k_1\cdot N_0}\norm{S^{a+1} f_1}_{H^{N_0}},2^{p_1}\norm{S^{a+1} f_1}_B\}\norm{S^b f_2}_{L^2},
\end{aligned} 
\end{equation}
which gives an acceptable contribution to \eqref{eq:Upssum_red}. The first term in \eqref{eq:2ndterm-split-ibp-split} is similar, but easier. 
\end{proof}

The rest of this section is devoted to the proof of \eqref{eq:UPS_red}. As in the proof of Lemma \ref{lem:red_Upsprop}, we use the decomposition into three terms as in \eqref{eq:Upsterms_k} to see that it suffices to prove
the following: 
\begin{lemma}\label{lem:UPS_red_split}
Under the assumptions of Proposition \ref{prop:Ups_prop}, let $m\in\N$ be such that $2^m\leq t\leq 2^{m+1}$, and $f=P_{k,p,q}f$, $f_j=P_{k_j,p_j,q_j}$, $j=1,2$. Then for $0\le a+b\le N+3$ there holds that
\begin{equation}\label{eq:UPS_red_split}
\begin{aligned}
 \sup_{\substack{k_1,k_2\in (-2(1+\delta)m,2m/N_0)\\p_1,p_2\in (-2m,0]\\q_1,q_2\in (-4m,0]}}&\norm{ \int_{s=0}^t\tau_m(s) P_{k,p,q}\A(s)ds}_{L^2}\lesssim \varepsilon_1^2 2^{\frac{10}{N_0}\cdot m},\\
 &\A\in \{\Q_{\Ups_\xi\m}(S^af_1,S^bf_2),\Q_{\m}(\Ups_\xi S^af_1,S^bf_2),\Q_{s\m \Ups_\xi\Phi}(S^af_1,S^bf_2)\}.
\end{aligned}
\end{equation}
\end{lemma}

Unlike for the reductions in Lemma \ref{lem:red_Upsprop}, here it will be important to also have normal forms at our disposal (see also Section \ref{ssec:NFs}). The cases of $\Ups$ on the multiplier or on the first input $S^a f_1$ (i.e.\ the first two terms in \eqref{eq:UPS_red_split}) can be dealt with relatively directly, and are detailed in Sections \ref{ssec:Upsm} and \ref{ssec:UpsVf1}, respectively. Finally, the case when $\Ups$ hits the phase produces another loss of $s$. The crucial groundwork to recover this has already been laid with the bilinear estimates developed in Section \ref{sec:more_bilin}, and the final details are given in Section \ref{ssec:UpsPhase}.

To unburden the notation, in what follows we shall denote the above, localized expressions in \eqref{eq:UPS_red_split} by
\begin{equation}
 \B_{\mathfrak{n}}(F,S^bf_2)=\int_{s=0}^t\tau_m(s) P_{k,p,q}\Q_\mathfrak{n}(F,S^bf_2)(s)ds,
\end{equation}
where $\mathfrak{n}\in\{\m,\Ups_\xi\m,s\m\Ups_\xi\Phi\}$ and $F\in\{S^af_1,\Ups_\xi S^a f_1\}$, thus dropping the explicit dependency on $m\in\N$ in the notation.

\subsection{Proof of Lemma \ref{lem:UPS_red_split} -- $\Ups$ on $\m$}\label{ssec:Upsm}
We show here that under the assumptions of Proposition \ref{prop:Ups_prop}, we have that
\begin{equation}
 \sup_{\substack{k_1,k_2\in (-2(1+\delta)m,2m/N_0)\\p_1,p_2\in (-2m,0]\\q_1,q_2\in (-4m,0]}}\norm{\B_{\Ups_\xi\m}(S^af_1,S^bf_2)}_{L^2}\lesssim 2^{\frac{5}{N_0}\cdot m}\eps_1^2,\qquad a+b\le N+3,\, k_2\leq k_1.
\end{equation} 

We begin by recalling from \eqref{eq:simp_symb} and Lemma \ref{lem:Ups_mult} that
\begin{equation}
 \norm{\Upsilon_\xi\m}_{\W}\lesssim 2^{k}[1+2^{q+p_2-p_1}]
\end{equation}
It thus follows that if $b=\min\{a,b\}$ or if $2^{p_2-p_1}\lesssim 1$ that
\begin{equation}
\begin{aligned}
 \norm{\B_{\Ups_\xi\m}(S^af_1,S^bf_2)}_{L^2}&\lesssim s\cdot 2^{k_{\max}}\norm{S^{\max\{a,b\}}f_1}_{B}\norm{e^{it\Lambda}S^{\min\{a,b\}} f_2}_{L^\infty}\\
 &\lesssim 2^{\frac{5}{2}k_{\max}}\norm{S^{\max\{a,b\}}f_1}_{B}\norm{S^{\min\{a,b\}} f_2}_{D}\lesssim \ip{s}^{\frac{5}{N_0}}\eps_1^2.
\end{aligned}
\end{equation}

Hence we can assume now that $b>a$ and $2^{p_1}\ll 2^{p_2}$. We are then in the same situation as in the proof of Lemma \ref{lem:UpsNF}, starting from \eqref{eq:Upsmult-ibp'}, and can proceed as done there. For the sake of completeness we recall here the main points, which are found in detail in the aforementioned proof of Lemma \ref{lem:UpsNF}. Since $2^{p_1}\ll 2^{p_2}$ we have that $\abs{\bar\sigma}\sim 2^{p_2}2^{k_1+k_2}$, and thus by \eqref{eq:vf_sigma} there holds $\abs{V_{\xi-\eta}\Phi}\sim 2^{k_1-k_2+2p_2}$, and we may integrate by parts in $V_{\xi-\eta}$ to obtain
\begin{equation}\label{eq:Upsmult-ibp}
 \B^{V_{\xi-\eta}}_{\Ups_\xi\m}(S^af_1,S^bf_2)=\B^{V_{\xi-\eta}}_{s^{-1}V_{\xi-\eta}\left(\frac{\Ups_\xi\m}{V_{\xi-\eta}\Phi}\right)}(S^af_1,S^bf_2)+\B^{V_{\xi-\eta}}_{s^{-1}\frac{\Ups_\xi\m}{V_{\xi-\eta}\Phi}}(V_{\xi-\eta}S^af_1,S^bf_2)+\B^{V_{\xi-\eta}}_{s^{-1}\frac{\Ups_\xi\m}{V_{\xi-\eta}\Phi}}(S^af_1,S^{b+1}f_2).
\end{equation}
To estimate these terms we recall from \eqref{eq:Upsmfracphi} and \eqref{eq:VUpsmfracphi} that
\begin{equation}
 \abs{\frac{\Ups_\xi\m}{V_{\xi-\eta}\Phi}}\lesssim 2^{k_2}2^{-p_2-p_1},\quad \abs{V_{\xi-\eta}\left(\frac{\Ups_\xi\m}{V_{\xi-\eta}\Phi}\right)}\lesssim 2^{k_{\max}}[1+2^{-2p_1}].
\end{equation}
For the terms in \eqref{eq:Upsmult-ibp} we thus have the bounds
\begin{equation}
\begin{aligned}
 \norm{\B^{V_{\xi-\eta}}_{s^{-1}V_{\xi-\eta}\left(\frac{\Ups_\xi\m}{V_{\xi-\eta}\Phi}\right)}(S^af_1,S^bf_2)}_{L^2}&\lesssim 2^{k_{\max}}[1+2^{-2p_1}]\norm{S^af_1}_{L^2}\norm{S^bf_2}_{L^2}\cdot 2^{k_1+p_1}2^{\frac{k_{\max}}{2}}\\
 &\lesssim 2^{\frac{5}{2}k_{\max}}\norm{S^a f_1}_B\norm{S^bf_2}_{L^2}\lesssim \ip{s}^{\frac{5}{N_0}}\eps_1^2,
\end{aligned} 
\end{equation}
and similarly
\begin{equation}
 \norm{\B^{V_{\xi-\eta}}_{s^{-1}\frac{\Ups_\xi\m}{V_{\xi-\eta}\Phi}}(S^af_1,S^{b+1}f_2)}_{L^2}\lesssim 2^{\frac{5}{2}k_{\max}}\norm{S^af_1}_{L^2}\norm{S^{b+1}f_2}_{B}\lesssim \ip{s}^{\frac{5}{N_0}}\eps_1^2,
\end{equation}
as well as
\begin{equation}
\begin{aligned}
 \norm{\B^{V_{\xi-\eta}}_{s^{-1}\frac{\Ups_\xi\m}{V_{\xi-\eta}\Phi}}(V_{\xi-\eta}S^af_1,S^bf_2)}_{L^2}&\lesssim 2^{k_{\max}+k_1-k_2}2^{-p_2-p_1}[\norm{S^{a+1} f_1}_{L^2}+(2^{p_2}+2^{p_1})\norm{\Ups S^a f_1}_{L^2}]\norm{S^b f_2}_{L^2}\\
 &\qquad \qquad \cdot \min\{2^{k_2+p_2},2^{k_1+p_1}\}2^{\frac{k_{\min}}{2}}\\
 &\lesssim 2^{\frac{5}{2}k_{\max}}[\norm{S^{a+1}f_1}_B+\norm{\Ups S^af_1}_{L^2}]\norm{S^bf_2}_B\lesssim \ip{s}^{\frac{5}{N_0}}\eps_1^2.
\end{aligned} 
\end{equation}

\subsection{Proof of Lemma \ref{lem:UPS_red_split} -- $\Ups$ on $S^af_1$}\label{ssec:UpsVf1}
We show here that under the assumptions of Proposition \ref{prop:Ups_prop}, we have that
\begin{equation}
 \sup_{\substack{k_1,k_2\in (-2(1+\delta)m,2m/N_0)\\p_1,p_2\in (-2m,0]\\q_1,q_2\in (-4m,0]}}\norm{\B_{\m}(\Ups_\xi S^af_1,S^bf_2)}_{L^2}\lesssim 2^{\frac{10}{N_0}\cdot m}\eps_1^2,\qquad a+b\le N+3,\, k_2\leq k_1.
\end{equation} 

We note that by \eqref{eq:Ups_cross} there holds that (since $f_1$ is axisymmetric)
\begin{equation}\label{eq:Ups_cross_axisymm}
\begin{aligned}
 \Ups_\xi S^a f_1(\xi-\eta)&=\Gamma_S\cdot (S^{a+1} f_1)(\xi-\eta)+\Gamma_\Ups \cdot (\Ups S^a f_1)(\xi-\eta),\\
 \Gamma_S&=\frac{\abs{\xi}}{\abs{\xi-\eta}}\left[\omega_c\Lambda(\xi)\sqrt{1-\Lambda^2(\xi-\eta)}-\sqrt{1-\Lambda^2(\xi)}\Lambda(\xi-\eta)\right],\\
 \Gamma_\Ups&=\frac{\abs{\xi}}{\abs{\xi-\eta}}\left[\omega_c\Lambda(\xi)\Lambda(\xi-\eta)+\sqrt{1-\Lambda^2(\xi)}\sqrt{1-\Lambda^2(\xi-\eta)}\right],
\end{aligned} 
\end{equation}
and thus
\begin{equation}\label{eq:2ndterm-split}
 \B_{\m}(\Ups_\xi S^af_1,S^bf_2)=\B_{\m}(\Gamma_S S^{a+1}f_1,S^bf_2)+\B_{\m}(\Gamma_\Ups \Ups S^af_1,S^bf_2).
\end{equation}

We note that since $\abs{\Gamma_S}+\abs{\Gamma_\Ups}\lesssim 1$, we can directly estimate the first of the two terms that arise, namely by \eqref{eq:simp_symb}
\begin{equation}
\begin{aligned}
 \norm{\B_{\m}(\Gamma_S S^{a+1}f_1,S^bf_2)}_{L^2}&\lesssim s\cdot 2^{k_{\max}}\norm{S^{\max\{a,b\}+1}f_1}_{L^2}\norm{e^{it\Lambda}S^{\min\{a,b\}+1} f_2}_{L^\infty}\\
 &\lesssim 2^{\frac{5}{2}k_{\max}}\norm{S^{\max\{a,b\}+1}f_1}_{L^2}\norm{S^{\min\{a,b\}+1} f_2}_{D}\lesssim \ip{s}^{\frac{5}{N_0}}\eps_1^2.
\end{aligned} 
\end{equation}
If $a\geq b$ the second term can be dealt with similarly, so that
\begin{equation}
 \norm{\B_{\m}(\Gamma_\Ups \Ups S^af_1,S^bf_2)}_{L^2}\lesssim 2^{\frac{5}{2}k_{\max}}\norm{\Ups S^a f_1}_{L^2}\norm{S^bf_2}_D\lesssim \ip{s}^{\frac{5}{N_0}}\eps_1^2,\qquad a\geq b.
\end{equation}
Hence we may assume for the rest of this subsection that $a<b$ so that $a< (N+3)/2\le N$.
\paragraph{Case $\abs{\Phi}\gtrsim 2^{q_{\max}}$.}
Here we make use of a normal form transformation. We note that by Lemma \ref{lem:phasesymb_bd} there holds 
\begin{equation}\label{eq:mfracphi_symb}
 \norm{\frac{\m}{\Phi}\bar\chi(2^{-q_{\max}}\Phi)}_{\W} \lesssim 2^k.
\end{equation}
In analogy with the notation of Section \ref{sec:more_bilin} we denote by a superscript $\Phi$ on the bilinear expressions $\B$ and $\Q$ the localization in $\Phi$, i.e.\
\begin{equation}
 \B^\Phi_{\m}(f,g)=\B_{\m\cdot \chi(2^{-q_{\max}}\Phi)}(f,g).
\end{equation}
Then we have that
\begin{equation}
\begin{aligned}
 \B^{\Phi}_{\m}(\Gamma_\Ups \Ups S^af_1,S^bf_2)&=-i\Q^{\Phi}_{\frac{\m}{\Phi}}(\Gamma_\Ups \Ups S^af_1,S^bf_2)(t)+i\Q^{\Phi}_{\frac{\m}{\Phi}}(\Gamma_\Ups \Ups S^af_1,S^bf_2)(0)\\
 &\qquad+i\B^{\Phi}_{\frac{\m}{\Phi}}(\Gamma_\Ups\partial_t \Ups S^af_1,S^bf_2)+i\B^{\Phi}_{\frac{\m}{\Phi}}(\Gamma_\Ups \Ups S^af_1,\partial_t S^bf_2),
\end{aligned} 
\end{equation}
and the boundary terms are easily controlled:
\begin{equation}
 \norm{\Q^{\Phi}_{\frac{\m}{\Phi}}(\Gamma_\Ups \Ups S^af_1,S^bf_2)(t)}_{L^2}+\norm{\Q^{\Phi}_{\frac{\m}{\Phi}}(\Gamma_\Ups \Ups S^af_1,S^bf_2)(0)}_{L^2}\lesssim 2^{\frac{3}{2}k_{\max}}\norm{\Ups S^a f_1}_{L^2}\norm{S^b f_2}_{L^2}\lesssim \ip{s}^{\frac{3}{N_0}}\eps_1^2.
\end{equation}
Next, by Lemma \ref{lem:VNFs} we have that
\begin{equation}
\begin{aligned}
 \norm{\B^{\Phi}_{\frac{\m}{\Phi}}(\Gamma_\Ups \Ups S^af_1,\partial_t S^bf_2)}_{L^2}&\lesssim s\cdot2^{\frac{3}{2}k_{\max}}\norm{\Ups S^af_1}_{L^\infty_t L^2}\norm{\partial_t S^bf_2}_{L^\infty_t L^2}\lesssim \ip{s}^{\frac{9}{N_0}}\eps_1^3
\end{aligned} 
\end{equation}
Similarly, with Lemma \ref{lem:UpsNF} and \eqref{eq:mfracphi_symb} an $L^\infty\times L^2$ estimate gives that
\begin{equation}
\begin{aligned}
 &\norm{\B^{\Phi}_{\frac{\m}{\Phi}}(\Gamma_\Ups\partial_t \Ups S^af_1,S^bf_2)}_{L^2}\lesssim 
 \int_0^t\norm{e^{it\Lambda}\partial_t \Ups S^af_1(s)}_{L^\infty} ds \cdot \norm{S^bf_2}_{L^\infty_t L^2}\\
 &\qquad \lesssim\int_0^t\ip{s}^{-1+\frac{8}{N_0}}\eps_1^2ds \cdot \eps_1\lesssim \ip{s}^{\frac{8}{N_0}}\eps_1^3.
\end{aligned} 
\end{equation}

\paragraph{Case $\abs{\Phi}\ll 2^{q_{\max}}$.}
Here by Proposition \ref{prop:phasevssigma} we have that $\abs{\bar\sigma}\gtrsim 2^{q_{\max}}2^{k_{\max}+k_{\min}}$, and we will thus integrate by parts in $V_{\xi-\eta}$, obtaining that
\begin{equation}
\begin{aligned}
 \Q^{V_{\xi-\eta}}_{\m}(\Gamma_\Ups \Ups S^a f_1,S^b f_2)&=s^{-1}\cdot \Big[ \Q^{V_{\xi-\eta}}_{V_{\xi-\eta}\left(\frac{\m}{V_{\xi-\eta}\Phi}\right)}(\Gamma_\Ups \Ups S^a f_1,S^b f_2)+ \Q^{V_{\xi-\eta}}_{\frac{\m}{V_{\xi-\eta}\Phi}}(V(\Gamma_\Ups \Ups S^a f_1),S^b f_2)\\
 &\qquad\qquad +\Q^{V_{\xi-\eta}}_{\frac{\m}{V_{\xi-\eta}\Phi}}(\Gamma_\Ups \Ups S^a f_1,V_{\xi-\eta} S^b f_2)\Big]
\end{aligned} 
\end{equation}
To estimate these terms we compute that $\abs{V_{\xi-\eta}\Phi}\sim 2^{-2k_2+p_2}\abs{\bar\sigma}\gtrsim 2^{p_2+q_{\max}}2^{-2k_2+k_{\max}+k_{\min}}$, and thus, with notations from \eqref{eq:simpmultbd},
\begin{equation}\label{eq:mfracvphi}
\frac{\vert\m\vert}{ \abs{V_{\xi-\eta}\Phi}}\le\frac{C_\m}{ \abs{V_{\xi-\eta}\Phi}}\lesssim 2^{p_{\max}-p_2}2^{k_2}.
\end{equation}
Similarly, by \eqref{eq:V'onmult} we have $\abs{V_{\xi-\eta}\m}\lesssim (1+2^{k_1-k_2}2^{p_1-p_2})C_\m$, and so with \eqref{eq:vfquot'} it follows from \eqref{eq:mfracvphi} that
\begin{equation}
\begin{aligned}
 \abs{V_{\xi-\eta}\left(\frac{\m}{V_{\xi-\eta}\Phi}\right)}&\lesssim \abs{\frac{V_{\xi-\eta}\m}{V_{\xi-\eta}\Phi}}+\abs{\frac{V_{\xi-\eta}^2\Phi}{V_{\xi-\eta}\Phi}}\abs{\frac{\m}{V_{\xi-\eta}\Phi}}\lesssim (1+2^{k_1-k_2}2^{p_1-p_2})\frac{C_\m}{\abs{V_{\xi-\eta}\Phi}}\\
 &\lesssim 2^{p_{\max}-p_2}2^{k_2}+2^{p_{\max}+p_1-2p_2}2^{k_1}.
\end{aligned} 
\end{equation}
We can thus bound (via $\Sz\lesssim 2^{\frac{k_{\max}}{2}+k_2+p_2}$)
\begin{equation}
\begin{aligned}
 \norm{\Q^{V_{\xi-\eta}}_{V_{\xi-\eta}\left(\frac{\m}{V_{\xi-\eta}\Phi}\right)}(\Gamma_\Ups \Ups S^a f_1,S^b f_2)}_{L^2}&\lesssim \Sz\cdot\abs{V_{\xi-\eta}\left(\frac{\m}{V_{\xi-\eta}\Phi}\right)}\norm{\Ups S^a f_1}_{L^2}\norm{S^b f_2}_{L^2}\\
 &\lesssim 2^{\frac{5}{2}k_{\max}}\norm{\Ups S^a f_1}_{L^2}\norm{S^b f_2}_B\lesssim \ip{s}^{\frac{5}{N_0}}\eps_1^2.
\end{aligned} 
\end{equation}
Furthermore, we compute with that \eqref{eq:Ups_cross_axisymm} that
\begin{equation}
 \abs{V_{\xi-\eta}\Gamma_\Ups}\lesssim \abs{\Gamma_\Ups}+\abs{V_{\xi-\eta}(\omega_c)\cdot\Lambda(\xi)\Lambda(\xi-\eta)}\lesssim 1+2^{p_1+k_1}2^{-p_2-k_2},
\end{equation}
where we have used that by (the symmetric version of) \eqref{eq:Vonangle2} $\abs{V_{\xi-\eta}\omega_c}\lesssim 2^{p_1+k_1}2^{-p_2-k_2}$. From this and \eqref{eq:mfracvphi}, it follows with $\Sz\lesssim 2^{\frac{k_{\max}}{2}+k_2+p_2}$ that
\begin{equation}
\begin{aligned}
 \norm{\Q^{V_{\xi-\eta}}_{\frac{\m}{V_{\xi-\eta}\Phi}}(V(\Gamma_\Ups \Ups S^a f_1),S^b f_2)}_{L^2}&\lesssim \Sz 2^{p_{\max}-p_2}2^{k_2} \Big[\norm{\Ups S^{a+1}f_1}_{L^2}+2^{p_1+k_1}2^{-p_2-k_2}\norm{\Ups S^a f_1}_{L^2}\Big] \norm{S^b f_2}_{L^2}\\
 &\lesssim 2^{\frac{k_{\max}}{2}+2k_2}\Big[\norm{\Ups S^{a+1}f_1}_{L^2}+2^{p_1+k_1}2^{-p_2-k_2}\norm{\Ups S^a f_1}_{L^2}\Big] \norm{S^b f_2}_{L^2}\\
 &\lesssim 2^{\frac{5}{2}k_{\max}}\Big[\norm{\Ups S^{a+1}f_1}_{L^2}+\norm{\Ups S^a f_1}_{L^2}\Big] \norm{S^b f_2}_{B}\lesssim\ip{s}^{\frac{5}{N_0}}\eps_1^2.
\end{aligned} 
\end{equation}
Finally, with (the symmetric version of) \eqref{eq:cross_recall} we have that 
\begin{equation}
 \abs{V_{\xi-\eta}S^b f_2(\eta)}\lesssim 2^{k_1-k_2}\left[\abs{S^{b+1} f_2(\eta)}+(2^{p_1}+2^{p_2})\abs{\Ups S^b f_2(\eta)}\right],
\end{equation}
and thus using \eqref{eq:mfracvphi} and the fact that $\Sz\lesssim 2^{\frac{k_{\max}}{2}+k_2+p_2}$
\begin{equation}
\begin{aligned}
 \norm{\Q^{V_{\xi-\eta}}_{\frac{\m}{V_{\xi-\eta}\Phi}}(\Gamma_\Ups \Ups S^a f_1,V_{\xi-\eta} S^b f_2)}_{L^2}&\lesssim 2^{p_{\max}-p_2}2^{k_2}\cdot 2^{k_1-k_2}\cdot \Sz\cdot \norm{\Ups S^a f_1}_{L^2}\\
 &\qquad\qquad \cdot[\norm{S^{b+1} f_2}_{L^2}+(2^{p_1}+2^{p_2})\norm{\Ups S^b f_2}_{L^2}]\\
 &\lesssim 2^{\frac{5}{2}k_{\max}}\norm{\Ups S^a f_1}_{L^2}\cdot [\norm{S^{b+1} f_2}_{L^2}+\norm{\Ups S^b f_2}_{L^2}]\lesssim\ip{s}^{\frac{5}{N_0}}\eps_1^2.
\end{aligned} 
\end{equation}

\subsection{Proof of Lemma \ref{lem:UPS_red_split} -- $\Ups$ on $\Phi$}\label{ssec:UpsPhase}
We show here that under the assumptions of Proposition \ref{prop:Ups_prop}, we have that
\begin{equation}
 \sup_{\substack{k_1,k_2\in (-2(1+\delta)m,2m/N_0)\\p_1,p_2\in (-2m,0]\\q_1,q_2\in (-4m,0]}}\norm{\B_{s\m\Ups_\xi\Phi}(S^af_1,S^bf_2)}_{L^2}\lesssim 2^{\frac{10}{N_0}\cdot m}\eps_1^2,\qquad a+b\le N+3,\, k_2\leq k_1.
\end{equation} 

As in the previous section, we distinguish between a normal form and two integrations by parts. The details are as follows:
\paragraph{Case $\abs{\Phi}\gtrsim 2^{q_{\max}}$.}
Here we make use of a normal form transformation. We note that by Lemma \ref{lem:phasesymb_bd} and \eqref{eq:simp_symb} there holds
\begin{equation}\label{eq:mfracphiUps_symb}
 \norm{\frac{\m}{\Phi}\bar\chi(2^{-q_{\max}}\Phi)}_{\W}\norm{\Ups_\xi\Phi}_{\W}\lesssim 2^k\cdot[2^p+2^{k-k_1}2^{p_1}]\lesssim 2^k.
\end{equation}
Then we have (with the same notation as in the previous section) that
\begin{equation}
\begin{aligned}
 \B^{\Phi}_{s\m \Ups_\xi\Phi}(S^af_1,S^bf_2)&=-i\Q^{\Phi}_{s\frac{\m}{\Phi} \Ups_\xi\Phi}(S^af_1,S^bf_2)(t)+i\Q^{\Phi}_{s\frac{\m}{\Phi} \Ups_\xi\Phi}(S^af_1,S^bf_2)(0)\\
 &\qquad+i\B^{\Phi}_{s\frac{\m}{\Phi} \Ups_\xi\Phi}(\partial_t S^af_1,S^bf_2)+i\B^{\Phi}_{s\frac{\m}{\Phi} \Ups_\xi\Phi}(S^af_1,\partial_t S^bf_2).
\end{aligned} 
\end{equation}
The boundary terms are easily controlled: By \eqref{eq:mfracphiUps_symb} we have
\begin{equation}
\begin{aligned}
 \norm{\Q^{\Phi}_{s\frac{\m}{\Phi} \Ups_\xi\Phi}(S^af_1,S^bf_2)(t)}_{L^2}+\norm{\Q^{\Phi}_{s\frac{\m}{\Phi} \Ups_\xi\Phi}(S^af_1,S^bf_2)(0)}_{L^2}&\lesssim s\cdot 2^{k}\norm{e^{it\Lambda}S^{\min\{a,b\}} f_i}_{L^\infty}\norm{S^{\max\{a,b\}} f_j}_{L^2}\\
 &\lesssim 2^{\frac{5}{2}k_{\max}}\norm{S^{\min\{a,b\}} f_1}_{D}\norm{S^{\max\{a,b\}} f_2}_{L^2}\\
 &\lesssim\ip{s}^{\frac{5}{N_0}}\eps_1^2.
\end{aligned}
\end{equation}
Next, using in addition \eqref{eq:VNFLinfty} in Lemma \ref{lem:VNFs} we have that if $a\leq b$
\begin{equation}
\begin{aligned}
 \norm{\B^{\Phi}_{s\frac{\m}{\Phi} \Ups_\xi\Phi}(\partial_t S^af_1,S^bf_2)}_{L^2}&\lesssim \int_0^t s2^{k_{\max}}\norm{e^{is\Lambda}\partial_t S^af_1}_{L^\infty}\norm{S^bf_2}_{L^2}ds\lesssim\ip{s}^{\frac{7}{N_0}}\eps_1^3,
\end{aligned} 
\end{equation}
whereas if $a> b$, by \eqref{eq:VNFL2} in Lemma \ref{lem:VNFs} there holds that
\begin{equation}
\begin{aligned}
 \norm{\B^{\Phi}_{s\frac{\m}{\Phi} \Ups_\xi\Phi}(\partial_t S^af_1,S^bf_2)}_{L^2}&\lesssim \int_0^t s2^{k_{\max}}\norm{\partial_t S^af_1}_{L^2}\norm{e^{is\Lambda}S^bf_2}_{L^\infty}ds\lesssim\ip{s}^{\frac{7}{N_0}}\eps_1^3,
\end{aligned} 
\end{equation}
which both are acceptable contributions. The estimate for $\B_{s\frac{\m}{\Phi} \Ups_\xi\Phi}(S^af_1,\partial_t S^bf_2)$ is analogous.

\paragraph{Case $\abs{\Phi}\ll 2^{q_{\max}}$.}
Here by Proposition \ref{prop:phasevssigma} we have that $\abs{\bar\sigma}\gtrsim 2^{q_{\max}}2^{k_{\max}+k_{\min}}$, and we are thus in a setting where we can directly apply Lemma \ref{lem:bdry}. If $a>b$, then we obtain by \eqref{eq:allV1} that
\begin{equation}
\begin{aligned}
 \norm{\Q^{{V_\eta}}_{s\m\Ups\Phi}[S^a f_1,S^b f_2]}_{L^2}&\lesssim s^{-1}2^{2k_{\max}^+}\Big[\norm{(1,S)S^af_1}_B+\norm{\Ups S^a f_1}_{L^2}\Big]\cdot [\norm{(1,S)S^b f_2}_D+\norm{S^b f_2}_B]\\
  &\quad+s^{-1}2^{2k_{\max}^+}\norm{(1,S)S^af_1}_B \Big[\norm{(1,S,S^2)S^bf_2}_B+\norm{(1,S)\Ups S^b f_2}_{L^2}\Big]\\
  &\lesssim \ip{s}^{-1+\frac{4}{N_0}}\eps_1^2,
\end{aligned}  
\end{equation}
which gives an acceptable contribution. If on the other hand $a\leq b$, then by \eqref{eq:allV2} there holds that
\begin{equation}
\begin{aligned}
 \norm{\Q^{{V_\eta}}_{s\m\Ups\Phi}[S^a f_1,S^b f_2]}_{L^2}&\lesssim s^{-1}2^{3k_{\max}^+}\Big[\norm{(1,S)S^af_1}_B+\norm{\Ups V' S^af_1}_{L^2}\Big]\cdot \Big[\norm{(1,S)S^b f_2}_B+\norm{\Ups S^b f_2}_{L^2}\Big]\\
  &\quad +s^{-1}2^{-3k_{\max}^+}\norm{(1,S)S^a f_1}_D\Big[\norm{S^{b+1}f_2}_B+\norm{\Ups S^b f_2}_{L^2}+\norm{S^{b+1}f_2}_{H^{-1}}\Big]\\
  &\lesssim\ip{s}^{-1+\frac{6}{N_0}}\eps_1^2,
\end{aligned}
\end{equation}
which again gives a controlled contribution.

\subsection*{Acknowledgments}
The authors would like to thank F.\ Pusateri for interesting discussions.

Y.\ Guo's research is supported in part by NSF grant DMS-1810868. C.\ Huang is supported by the National Natural Science Foundation of China (No.\ 11971503). B.\ Pausader is supported in part by NSF grant DMS-1700282.

\newpage
\appendix
\section{Vector Field Computations}
For the rotation and scaling vector fields $S=\sum_{i=1}^3x^i\partial_i$ and $\Omega=x_1\partial_2-x_2\partial_1$ we note the following properties:
\begin{lemma}\label{lem:vfbasics}
 Consider the rotation and scaling vector fields $\Omega$ and $S$, as well as the horizontal part $S_\h$ of the latter, given by
 \begin{equation*}
  S=\sum_{i=1}^3x_i\partial_i,\quad \Omega=x_1\partial_2-x_2\partial_1,\quad S_\h:=x_1\partial_1+x_2\partial_2.
 \end{equation*}
 They satisfy
 \begin{enumerate}
  \item $[A,B]=0$ for any choice of $A,B\in\{S,\Omega,\Ups\}$, i.e.\ all these vector fields commute.
  \item $\widehat{Sf}(\xi)=-(3+S)\hat{f}(\xi)$, $\widehat{\Omega f}(\xi)=\Omega\hat{f}(\xi)$,
  \item $\partial_1=\frac{x_\h^\perp}{\abs{x_\h}^2}\cdot (\Omega,S_\h)^\intercal$ and $\partial_2=\frac{x_\h}{\abs{x_\h}^2}\cdot (\Omega,S_\h)^\intercal$,
  \item Writing $y_\h=(y_1,y_2)^\intercal$ we have
    \begin{equation}\label{eq:vfbasics}
     \begin{aligned}
      &y_1\partial_{x_2}-y_2\partial_{x_1}=\frac{(y_\h\cdot x_\h)}{\abs{x_\h}^2} \Omega_x+\frac{(y_\h^\perp\cdot x_\h)}{\abs{x_\h}^2} S_{x,\h},\\
      &y_1\partial_{x_1}+y_2\partial_{x_2}=\frac{(y_\h\cdot x_\h^\perp)}{\abs{x_\h}^2} \Omega_x+\frac{(y_\h\cdot x_\h)}{\abs{x_\h}^2} S_{x,\h}.
     \end{aligned}
    \end{equation}
 \end{enumerate}
\end{lemma}
\begin{proof}
 These are direct computations.
\end{proof}

\subsection{Vector Fields and the Phase}\label{apdx:phase-comp}
We collect now a few computations regarding the interaction of the vector fields and the phase functions. We also recall that $S\Lambda=\Omega\Lambda=0$.

We will use the following notation:
 \begin{equation}
  \Omega_\eta=\eta_1\partial_{\eta_2}-\eta_2\partial_{\eta_1}=\eta_\h^\perp\cdot\nabla_{\eta_\h},\quad S_\eta=\eta_1\partial_{\eta_1}+\eta_2\partial_{\eta_2}+\eta_3\partial_{\eta_3}=\eta\cdot\nabla_{\eta},
 \end{equation}
whereas
 \begin{equation}
  \Omega_{\xi-\eta}=(\xi_\h-\eta_\h)^\perp\cdot\nabla_{\eta_\h},\quad S_{\xi-\eta}=(\xi-\eta)\cdot\nabla_{\eta}.
 \end{equation}

\begin{lemma}\label{lem:vfsizes}
Let
 \begin{equation}
  \Phi:=\pm\Lambda(\xi)+\Lambda(\xi-\eta)\pm\Lambda(\eta).
 \end{equation} 
 Then we have 
 \begin{equation}
  V_\eta\Phi(\xi,\eta)=V_\eta\Lambda(\xi-\eta),\qquad V\in\{S,\Omega\}.
 \end{equation}
 With
 \begin{equation}
  \bar\sigma(\xi,\eta):=\xi_3\eta_\h-\eta_3\xi_\h=(\xi\wedge\eta)_\h
 \end{equation}
there holds that
 \begin{equation}\label{eq:vf_sigma}
  S_\eta\Phi=\bar\sigma(\xi,\eta)\cdot\frac{\xi_\h-\eta_\h}{\abs{\xi-\eta}^3},\quad \Omega_\eta\Phi=-\bar\sigma(\xi,\eta)\cdot\frac{(\xi_\h-\eta_\h)^\perp}{\abs{\xi-\eta}^3},
 \end{equation}
 and hence
 \begin{equation}\label{eq:vflobound}
  \abs{S_\eta\Phi}+\abs{\Omega_\eta\Phi}\sim 2^{-2k_1}2^{p_1}\abs{\bar\sigma(\xi,\eta)}.
 \end{equation}

 Moreover,
 \begin{equation}\label{eq:2ndvfPhi}
 \begin{aligned}
  &S_\eta^2\Phi=S_\eta\Phi\left[3\frac{\eta\cdot(\xi-\eta)}{\abs{\xi-\eta}^2}+2\right]-\frac{\bar\sigma\cdot\xi_\h}{\abs{\xi-\eta}^3},\\
  &\Omega_\eta^2\Phi=3\Omega_\eta\Phi\frac{\in^{ab}\eta_a\xi_b}{\abs{\xi-\eta}^2}-\Lambda(\xi-\eta)\frac{\eta_\h\cdot\xi_\h}{\abs{\xi-\eta}^2},\\
  &S_\eta\Omega_\eta\Phi=\Omega_\eta S_\eta\Phi=S_\eta\Phi\frac{\eta_\h^\perp\cdot\xi_\h}{\abs{\xi-\eta}^2}+\Omega_\eta\Phi[1+2\frac{\eta\cdot(\xi-\eta)}{\abs{\xi-\eta}^2}],
 \end{aligned}
 \end{equation}
and
\begin{align}
 \Omega_{\xi-\eta}\Omega_\eta\Phi&=-\Lambda(\xi-\eta)\frac{\xi_\h\cdot(\xi_\h-\eta_\h)}{\abs{\xi-\eta}^2}=\frac{\xi_3}{\abs{\xi-\eta}}\frac{\abs{\xi_\h-\eta_\h}^2}{\abs{\xi-\eta}^2}-S_\eta\Phi,\label{eq:1VFmix-1}\\
 \Omega_{\xi-\eta}S_\eta\Phi&=\Omega_\eta\Phi,\label{eq:1VFmix-2}\\
 S_{\xi-\eta}S_\eta\Phi&=S_\eta\Phi,\label{eq:1VFmix-3}\\
 S_{\xi-\eta}\Omega_\eta\Phi&=\Omega_\eta\Phi.\label{eq:1VFmix-4}
\end{align}

Furthermore,
\begin{align}
 S_{\xi-\eta}S_{\eta}^2\Phi&=S_\eta\Phi\left[5+6\frac{(\xi-\eta)\cdot\eta}{\abs{\xi-\eta}^2}\right]-2\frac{\bar\sigma\cdot\xi_\h}{\abs{\xi-\eta}^3},\label{eq:2VFmix-1}\\
 S_{\xi-\eta}\Omega_\eta^2\Phi&=6\Omega_\eta\Phi\,\frac{\eta_\h^\perp\cdot\xi_\h}{\abs{\xi-\eta}^2}-2\Lambda(\xi-\eta) \frac{(\xi_\h+\eta_\h)\xi_\h}{\abs{\xi-\eta}^2},\label{eq:2VFmix-2}\\
 \Omega_{\xi-\eta}S_\eta^2\Phi&=3\Omega_\eta\Phi\frac{\xi\cdot(\xi-\eta)}{\abs{\xi-\eta}^2}-3S_\eta\Phi\frac{\eta_\h^\perp\cdot\xi_\h}{\abs{\xi-\eta}^2}+\eta_3\frac{\eta_\h^\perp\cdot\xi_\h}{\abs{\xi-\eta}^3},\label{eq:2VFmix-3}\\
 \Omega_{\xi-\eta}\Omega_\eta^2\Phi&=\Omega_\eta\Phi\left[1-6\frac{\xi_\h\cdot(\xi_\h-\eta_\h)}{\abs{\xi-\eta}^2}\right]\label{eq:2VFmix-4}.
\end{align}

\end{lemma}

\begin{remark}\label{rem:sigma}
We list here some useful, immediate consequences of the computations in Lemma \ref{lem:vfsizes}.
\begin{enumerate}
 \item Note the versatility of $\bar\sigma(\xi,\eta)$ in that
 \begin{equation}
  \bar\sigma(\xi,\eta)=\bar\sigma(\xi-\eta,\eta)=-\bar\sigma(\xi,\xi-\eta).
 \end{equation}
 However, to unburden the notation we will write $\bar\sigma\equiv\bar\sigma(\xi,\eta)$.
 In particular, it is useful to note that $\Omega_\eta\Phi$ can also be expressed as
 \begin{equation}\label{eq:alt_OmegaPhi}
  \Omega_\eta\Phi=\Lambda(\xi-\eta)\frac{\xi_\h\cdot\eta_\h^\perp}{\abs{\xi-\eta}^2}=\Lambda(\xi-\eta)\frac{\in^{ab}\eta_a\xi_b}{\abs{\xi-\eta}^2}
 \end{equation}
 
 \item\label{rem:vfquot} Assuming that $\bar\sigma$ has a lower bound $\abs{\bar\sigma}\gtrsim2^{q_{\max}}2^{k_{\max}+k_{\min}}$, then whenever $\abs{V_\eta\Phi}\sim 2^{-2k_1+p_1}\abs{\bar\sigma}$, $V\in\{S,\Omega\}$, we have
 \begin{equation}\label{eq:vfquotient_apdx}
  \frac{\abs{V_\eta^2\Phi}}{\abs{V_\eta\Phi}}\lesssim 1+2^{k_2-k_1}+2^{-p_1-k_1}\min\{2^{p+k},2^{p_2+k_2}\}+2^{p+p_2-p_1}2^{k+k_2-k_{\max}-k_{\min}}.
 \end{equation}
 \begin{proof}
  This follows directly from the formulas for the second order vector fields on the phase in \eqref{eq:2ndvfPhi}: In case $V=S$ we have
 \begin{equation}
 \begin{aligned}
  \frac{\abs{S_\eta^2\Phi}}{\abs{S_\eta\Phi}}&\lesssim 1+2^{k_2-k_1}+2^{2k_1-p_1}\abs{\bar\sigma}^{-1}\cdot\abs{\bar\sigma}2^{-3k_1}\min\{2^{p+k},2^{p_2+k_2}\}\\
  &\lesssim 1+2^{k_2-k_1}+2^{-p_1-k_1}\min\{2^{p+k},2^{p_2+k_2}\},
 \end{aligned} 
 \end{equation}
 and when $V=\Omega$ we obtain
 \begin{equation}
 \begin{aligned}
  \frac{\abs{\Omega_\eta^2\Phi}}{\abs{\Omega_\eta\Phi}}&\lesssim \abs{\frac{\xi_\h\cdot\eta_\h^\perp}{\abs{\xi-\eta}^2}}+\abs{\Lambda(\xi-\eta)\frac{\xi_\h\cdot\eta_\h}{\abs{\xi-\eta}^2}}\abs{\Omega_\eta\Phi}^{-1}\\
  &\lesssim 2^{p_1-k_1}\min\{2^{p+k},2^{p_2+k_2}\}+2^{q_1}2^{p_2+k_2}2^{p+k}2^{-2k_1}2^{2k_1-p_1}\abs{\bar\sigma}^{-1}\\
  &=2^{p_1-k_1}\min\{2^{p+k},2^{p_2+k_2}\}+2^{q_1+p_2+p-p_1}2^{k+k_2}\abs{\bar\sigma}^{-1},
 \end{aligned}
 \end{equation}
 and we can now use that $\abs{\bar\sigma}\gtrsim 2^{q_{\max}}2^{k_{\max}+k_{\min}}$ to deduce the claim.
 \end{proof}
 
\end{enumerate} 

\end{remark}

\begin{proof}[Proof of Lemma \ref{lem:vfsizes}]
We start by computing
\begin{equation}\label{eq:Lamgrad}
 \nabla\Lambda(\xi)=-\frac{\xi_3}{\abs{\xi}^3}\xi_\h+\frac{\abs{\xi_\h}^2}{\abs{\xi}^3}\vec{e}_3.
\end{equation}
From this, \eqref{eq:vf_sigma} follows by direct computation: we have
\begin{equation}
\begin{aligned}
 S_\eta\Phi&=\eta\cdot\nabla_\eta\Lambda(\xi-\eta)=\frac{1}{\abs{\xi-\eta}^3}\left[(\xi_3-\eta_3)\eta_\h\cdot(\xi_\h-\eta_\h)-\eta_3\abs{\xi_\h-\eta_\h}^2\right]\\
 &=\frac{\xi_\h-\eta_\h}{\abs{\xi-\eta}^3}\cdot[(\xi_3-\eta_3)\eta_\h-\eta_3(\xi_\h-\eta_\h)]\\
 &=\frac{\xi_\h-\eta_\h}{\abs{\xi-\eta}^3}\cdot\bar{\sigma}(\xi,\eta),
\end{aligned} 
\end{equation}
and similarly
\begin{equation}
\begin{aligned}
 \Omega_\eta\Phi&=\eta_\h^\perp\cdot\nabla_\eta\Lambda(\xi-\eta)=\frac{1}{\abs{\xi-\eta}^3}\left[(\xi_3-\eta_3)\eta_\h^\perp\cdot(\xi_\h-\eta_\h)\right]=-\frac{(\xi_\h-\eta_\h)^\perp}{\abs{\xi-\eta}^3}\cdot[(\xi_3-\eta_3)\eta_\h]\\
 &=-\frac{(\xi_\h-\eta_\h)^\perp}{\abs{\xi-\eta}^3}\cdot\bar\sigma(\xi,\eta).
\end{aligned}
\end{equation}
Then from \eqref{eq:vf_sigma} we deduce the bounds \eqref{eq:vflobound}, since 
\begin{equation}
 \max\{\abs{\sigma(\xi,\eta)\cdot(\xi_\h-\eta_\h)},\abs{\sigma(\xi,\eta)\cdot(\xi_\h-\eta_\h)^\perp}\}\geq \frac{1}{\sqrt{2}}\abs{\sigma(\xi,\eta)}\abs{\xi_\h-\eta_\h}.
\end{equation}

\textbf{``Pure'' Second Order.}
We start by computing that 
\begin{equation}
  S_\eta(\abs{\eta})=\abs{\eta},\quad S_\eta\left(\frac{\eta_\mu}{\abs{\eta}}\right)=0\; (1\leq \mu\leq 3),\quad S_\eta(\abs{\xi-\eta})=-\frac{\eta\cdot(\xi-\eta)}{\abs{\xi-\eta}},
\end{equation}
and
\begin{equation}
 \Omega_\eta\frac{1}{\vert\xi-\eta\vert}=\frac{\in^{ab}\eta_a\xi_b}{\vert\xi-\eta\vert^3},\quad\Omega(\in^{ab}\eta_a\xi_b)=-\xi_j\eta_j,\quad\Omega_\eta(\xi_j\eta_j)=\in^{ab}\eta_a\xi_b.
\end{equation}
Next we have that
\begin{equation}\label{eq:grad_sigperp}
 \nabla_\eta(\bar\sigma\cdot(\xi_\h-\eta_\h))=\begin{pmatrix}\xi_3(\xi_\h-\eta_\h)-\bar\sigma\\-\xi_\h\cdot(\xi_\h-\eta_\h)\end{pmatrix},
\end{equation}
so
\begin{equation}
 S_\eta[\bar\sigma\cdot(\xi_\h-\eta_\h)]=\bar\sigma\cdot(\xi_\h-2\eta_\h).
\end{equation}
It follows that
\begin{equation}
\begin{aligned}
 S_\eta^2\Phi&=S_\eta[\abs{\xi-\eta}^{-3}\bar\sigma\cdot(\xi_\h-\eta_\h)]=3\abs{\xi-\eta}^{-4}\frac{\eta\cdot(\xi-\eta)}{\abs{\xi-\eta}}\bar\sigma\cdot(\xi_\h-\eta_\h)+\abs{\xi-\eta}^{-3}\bar\sigma\cdot(\xi_\h-2\eta_\h)\\
 &=S_\eta\Phi\left[3\frac{\eta\cdot(\xi-\eta)}{\abs{\xi-\eta}^2}+2\right]-\frac{\bar\sigma\cdot\xi_\h}{\abs{\xi-\eta}^3}.
\end{aligned} 
\end{equation}
On the other hand, from \eqref{eq:alt_OmegaPhi} we have
\begin{equation}
\begin{aligned}
 \Omega_\eta^2\Phi&=\Omega_\eta\left(\Lambda(\xi-\eta)\frac{\in^{ab}\eta_a\xi_b}{\abs{\xi-\eta}^2}\right)=\Omega_\eta\Phi\frac{\in^{ab}\eta_a\xi_b}{\abs{\xi-\eta}^2}+\Lambda(\xi-\eta)\left[-\frac{\eta_\h\cdot\xi_\h}{\abs{\xi-\eta}^2}+2\frac{(\in^{ab}\eta_a\xi_b)^2}{\abs{\xi-\eta}^2}\right]\\
 &=3\Omega_\eta\Phi\frac{\in^{ab}\eta_a\xi_b}{\abs{\xi-\eta}^2}-\Lambda(\xi-\eta)\frac{\eta_\h\cdot\xi_\h}{\abs{\xi-\eta}^2}.
\end{aligned} 
\end{equation}

Similarly one computes
\begin{equation}
 S_\eta\Omega_\eta\Phi=S_\eta\Phi\frac{\eta_\h^\perp\cdot\xi_\h}{\abs{\xi-\eta}^2}+\Omega_\eta\Phi[1+2\frac{\eta\cdot(\xi-\eta)}{\abs{\xi-\eta}^2}].
\end{equation}

\textbf{``Mixed'' Second Order.}
We have
\begin{equation}
\begin{aligned}
 \Omega_{\xi-\eta}\Omega_\eta\Phi&=(\xi_\h-\eta_\h)^\perp\cdot\nabla_{\eta_\h}\left(\Lambda(\xi-\eta)\frac{\in^{ab}\eta_a\xi_b}{\abs{\xi-\eta}^2}\right)=\frac{\Lambda(\xi-\eta)}{\abs{\xi-\eta}^2}(\xi_\h-\eta_\h)^\perp\cdot\nabla_{\eta_\h}(\in^{ab}\eta_a\xi_b)\\
 &=-\Lambda(\xi-\eta)\frac{\xi_\h\cdot(\xi_\h-\eta_\h)}{\abs{\xi-\eta}^2},
\end{aligned}
\end{equation}
and by \eqref{eq:grad_sigperp}
\begin{equation}
\begin{aligned}
 \Omega_{\xi-\eta}S_\eta\Phi&=(\xi_\h-\eta_\h)^\perp\cdot\nabla_{\eta_\h}[\frac{\xi_\h-\eta_\h}{\abs{\xi-\eta}^3}\cdot\bar\sigma]=\frac{1}{\abs{\xi-\eta}^3}(\xi_\h-\eta_\h)^\perp\cdot\nabla_{\eta_\h}[(\xi_\h-\eta_\h)\cdot\bar\sigma]\\
 &=\Omega_\eta\Phi.
\end{aligned}
\end{equation}
Similarly we compute that $(\xi-\eta)\cdot\nabla_\eta[(\xi_\h-\eta_\h)\cdot\bar\sigma]=-2(\xi_\h-\eta_\h)\cdot\bar\sigma$, so that
\begin{equation}
\begin{aligned}
 S_{\xi-\eta}S_\eta\Phi&=(\xi-\eta)\cdot\nabla_\eta[\frac{\xi_\h-\eta_\h}{\abs{\xi-\eta}^3}\cdot\bar\sigma]=\frac{3}{\abs{\xi-\eta}^3}[(\xi_\h-\eta_\h)\cdot\bar\sigma]-\frac{2}{\abs{\xi-\eta}^3}[(\xi_\h-\eta_\h)\cdot\bar\sigma]\\
 &=S_\eta\Phi.
\end{aligned} 
\end{equation}
Since $(\xi-\eta)\cdot\nabla_\eta[\abs{\xi-\eta}^{-2}]=2\abs{\xi-\eta}^{-2}$ and $(\xi-\eta)\nabla_\eta(\in^{ab}\eta_a\xi_b)=\xi_\h\cdot(\xi_\h-\eta_\h)^\perp$ we obtain
\begin{equation}
\begin{aligned}
 S_{\xi-\eta}\Omega_\eta\Phi&=(\xi-\eta)\cdot\nabla_\eta[\Lambda(\xi-\eta)\frac{\in^{ab}\eta_a\xi_b}{\abs{\xi-\eta}^2}]=\Lambda(\xi-\eta)(\xi-\eta)\cdot\nabla_\eta[\frac{\in^{ab}\eta_a\xi_b}{\abs{\xi-\eta}^2}]\\
 &=\Lambda(\xi-\eta)\left[\frac{\xi_\h\cdot(\xi_\h-\eta_\h)^\perp}{\abs{\xi-\eta}^2}+2\frac{\eta_\h^\perp\cdot\xi_\h}{\abs{\xi-\eta}^2}\right]\\
 &=\Omega_\eta\Phi.
\end{aligned}
\end{equation}

\textbf{``Mixed'' Third Order.}
To compute 
\begin{equation}
 S_{\xi-\eta}S_{\eta}^2\Phi=S_{\xi-\eta}\left\{S_\eta\Phi[3\frac{\eta\cdot(\xi-\eta)}{\abs{\xi-\eta}^2}+2]-\frac{\bar\sigma\cdot\xi_\h}{\abs{\xi-\eta}^3}\right\}
\end{equation}
we note that $S_{\xi-\eta}(\eta\cdot(\xi-\eta))=(\xi-\eta)\cdot(\xi-2\eta)$, and $S_{\xi-\eta}(\bar\sigma\cdot\xi_\h)=-\bar\sigma\cdot\xi_\h$, so that
\begin{equation}
\begin{aligned}
 S_{\xi-\eta}S_{\eta}^2\Phi&=S_{\xi-\eta}S_\eta\Phi[3\frac{\eta\cdot(\xi-\eta)}{\abs{\xi-\eta}^2}+2]+3S_\eta\Phi\left\{\frac{(\xi-\eta)\cdot(\xi-2\eta)}{\abs{\xi-\eta}^2}+2\frac{\eta\cdot(\xi-\eta)}{\abs{\xi-\eta}^2}\right\}+\frac{\bar\sigma\cdot\xi_\h}{\abs{\xi-\eta}^3}(-3+1)\\
 &=S_\eta\Phi\left\{2+3\frac{(\xi-\eta)\cdot(\xi+\eta)}{\abs{\xi-\eta}^2}\right\}-2\frac{\bar\sigma\cdot\xi_\h}{\abs{\xi-\eta}^3}\\
 &=S_\eta\Phi\left[5+6\frac{(\xi-\eta)\cdot\eta}{\abs{\xi-\eta}^2}\right]-2\frac{\bar\sigma\cdot\xi_\h}{\abs{\xi-\eta}^3}.
\end{aligned} 
\end{equation}
Similarly, we recall that $S_{\xi-\eta}(\in^{ab}\eta_a\xi_b)=\xi_\h\cdot(\xi_\h-\eta_\h)^\perp$ to obtain
\begin{equation}
\begin{aligned}
 S_{\xi-\eta}\Omega_\eta^2\Phi&=S_{\xi-\eta}\left[3\Omega_\eta\Phi\frac{\in^{ab}\eta_a\xi_b}{\abs{\xi-\eta}^2}-\Lambda(\xi-\eta)\frac{\eta_\h\cdot\xi_\h}{\abs{\xi-\eta}^2}\right]=3S_{\xi-\eta}\Omega_\eta\Phi\frac{\in^{ab}\eta_a\xi_b}{\abs{\xi-\eta}^2}\\
 &\qquad\qquad +3\Omega_\eta\Phi\frac{1}{\abs{\xi-\eta}^2}[\xi_\h\cdot(\xi_\h-\eta_\h)^\perp+2\eta_\h^\perp\cdot\xi_\h]-\frac{\Lambda(\xi-\eta)}{\abs{\xi-\eta}^2}[(\xi_\h-\eta_\h)\cdot\xi_\h+2\eta_\h\cdot\xi_\h]\\
 &=6\cdot\Omega_\eta\Phi\,\frac{\eta_\h^\perp\cdot\xi_\h}{\abs{\xi-\eta}^2}-2\Lambda(\xi-\eta) \frac{(\xi_\h+\eta_\h)\xi_\h}{\abs{\xi-\eta}^2}.
\end{aligned}
\end{equation}
Next we compute $\Omega_{\xi-\eta}(\eta\cdot(\xi-\eta))=\xi_\h\cdot(\xi_\h-\eta_\h)^\perp$, as well as $\Omega_{\xi-\eta}(\bar\sigma\cdot\xi_\h)=\xi_3\xi_\h\cdot(\xi_\h-\eta_\h)^\perp$ to conclude that
\begin{equation}
\begin{aligned}
 \Omega_{\xi-\eta}S_\eta^2\Phi&=\Omega_{\xi-\eta}\left\{S_\eta\Phi[3\frac{\eta\cdot(\xi-\eta)}{\abs{\xi-\eta}^2}+2]-\frac{\bar\sigma\cdot\xi_\h}{\abs{\xi-\eta}^3}\right\}\\
 &=\Omega_{\xi-\eta}S_\eta\Phi[3\frac{\eta\cdot(\xi-\eta)}{\abs{\xi-\eta}^2}+2]+3S_\eta\Phi\frac{\xi_\h\cdot(\xi_\h-\eta_\h)^\perp}{\abs{\xi-\eta}^2}-\xi_3\frac{\xi_\h\cdot(\xi_\h-\eta_\h)^\perp}{\abs{\xi-\eta}^3}\\
 &=\Omega_\eta\Phi[3\frac{\eta\cdot(\xi-\eta)}{\abs{\xi-\eta}^2}+2]-3S_\eta\Phi\frac{\eta_\h^\perp\cdot\xi_\h}{\abs{\xi-\eta}^2}+\xi_3\frac{\eta_\h^\perp\cdot\xi_\h}{\abs{\xi-\eta}^3}\\
 &=3\Omega_\eta\Phi\frac{\xi\cdot(\xi-\eta)}{\abs{\xi-\eta}^2}-3S_\eta\Phi\frac{\eta_\h^\perp\cdot\xi_\h}{\abs{\xi-\eta}^2}+\eta_3\frac{\eta_\h^\perp\cdot\xi_\h}{\abs{\xi-\eta}^3}.
\end{aligned}
\end{equation}
Finally, we recall that $\Omega_{\xi-\eta}(\in^{ab}\eta_a\xi_b)=-\xi_\h\cdot(\xi_\h-\eta_\h)$, which yields
\begin{equation}
\begin{aligned}
 \Omega_{\xi-\eta}\Omega_\eta^2\Phi&=\Omega_{\xi-\eta}\left\{3\Omega_\eta\Phi\frac{\in^{ab}\eta_a\xi_b}{\abs{\xi-\eta}^2}-\Lambda(\xi-\eta)\frac{\eta_\h\cdot\xi_\h}{\abs{\xi-\eta}^2}\right\}\\
 &=3\Omega_{\xi-\eta}\Omega_\eta\Phi\frac{\in^{ab}\eta_a\xi_b}{\abs{\xi-\eta}^2}-3\Omega_\eta\Phi\frac{\xi_\h\cdot(\xi_\h-\eta_\h)}{\abs{\xi-\eta}^2}-\Lambda(\xi-\eta)\frac{(\xi_\h-\eta_\h)^\perp\cdot\xi_\h}{\abs{\xi-\eta}^2}\\
 &=-3\Lambda(\xi-\eta)\frac{\xi_\h\cdot(\xi_\h-\eta_\h)}{\abs{\xi-\eta}^2}\frac{\in^{ab}\eta_a\xi_b}{\abs{\xi-\eta}^2}-3\Omega_\eta\Phi\frac{\xi_\h\cdot(\xi_\h-\eta_\h)}{\abs{\xi-\eta}^2}+\Omega_\eta\Phi\\
 &=\Omega_\eta\Phi\left[1-6\frac{\xi_\h\cdot(\xi_\h-\eta_\h)}{\abs{\xi-\eta}^2}\right].
\end{aligned}
\end{equation}

\end{proof}

\begin{lemma}[$\Ups$ and $V$ on the Phase]\label{lem:vfonUpsPhi}
 We have the bounds
 \begin{equation}\label{eq:vfonUpsPhi}
\begin{aligned}
 \abs{V_\eta\Ups_\xi\Phi}&\lesssim 2^{k-k_1}[2^{-p_1}\abs{V_\eta\Phi}+(2^{p_1}+2^{p_2})2^{k_2-k_1}],\qquad \abs{V_{\xi-\eta}\Ups_\xi\Phi}&\lesssim 2^{k-k_1+p_1}.
\end{aligned} 
\end{equation}
Moreover, there holds that
\begin{equation}\label{eq:2vfonUpsPhi}
 \abs{V'_{\xi-\eta}V_\eta\Ups_\xi\Phi}\lesssim 2^{k-k_1}2^{-p_1}\left[\abs{V_\eta\Phi}+\abs{V'_{\xi-\eta}V_\eta\Phi}\right]+(1+2^{p_1+k_1-p_2-k_2})\cdot(2^{p_1}+2^{p_2})2^{k_2-k_1}.
\end{equation}

\end{lemma}

\begin{proof}
Via direct computation in \eqref{eq:UpsPhi} we conclude that for $\zeta\in\{\eta,\xi-\eta\}$
\begin{equation}\label{eq:vfUpsabs}
 \abs{V_\zeta\Ups_\xi\Phi}= \abs{\xi}\abs{V_\zeta\left(\frac{\sqrt{1-\Lambda^2}(\xi-\eta)}{\vert\xi-\eta\vert}\left[\Lambda(\xi)\Lambda(\xi-\eta)\frac{\xi_\h\cdot(\xi-\eta)_\h}{\vert\xi_\h\vert\vert(\xi-\eta)_\h\vert}+\sqrt{1-\Lambda^2}(\xi)\sqrt{1-\Lambda^2}(\xi-\eta)\right]\right)}, 
\end{equation}
so that with $V_\eta\Lambda(\xi-\eta)=V_\eta\Phi$ and $V_\eta\sqrt{1-\Lambda^2}(\xi-\eta)=\frac{\Lambda(\xi-\eta)}{\sqrt{1-\Lambda^2}(\xi-\eta)}V_\eta\Lambda(\xi-\eta)$ and the below \eqref{eq:Vonangle1} we obtain the first estimate of \eqref{eq:vfonUpsPhi}. For the second one, we use that $V_{\xi-\eta}$ vanishes on all terms of $\Ups_\xi\Phi$ except for $\abs{\xi-\eta}$ (in case $V=S$, when $V_{\xi-\eta}\abs{\xi-\eta}=\abs{\xi-\eta}$) and the angle, where by (the symmetric version of) \eqref{eq:Vonangle0} we have
\begin{equation*}
 V_{\xi-\eta}\left(\frac{\xi_\h\cdot(\xi-\eta)_\h}{\vert\xi_\h\vert\vert(\xi-\eta)_\h\vert}\right)=
 \begin{cases}
 \frac{\xi_\h\cdot(\xi-\eta)_\h^\perp}{\vert\xi_\h\vert\vert(\xi-\eta)_\h\vert},&V=\Omega,\\
 0,&V=S.
 \end{cases}
\end{equation*}

Furthermore, distributing the vector field $V_\eta$ in \eqref{eq:vfUpsabs}, we see that a further application of $V'_{\xi-\eta}$ vanishes on all terms in $V_\eta\Ups_\xi\Phi$, except for $\abs{\xi-\eta}^{-1}$, $V_\eta(\abs{\xi-\eta}^{-1})$, $V_\eta\Lambda(\xi-\eta)$, $\frac{\xi_\h\cdot(\xi-\eta)_\h}{\vert\xi_\h\vert\vert(\xi-\eta)_\h\vert}$ and $V_\eta\left(\frac{\xi_\h\cdot(\xi-\eta)_\h}{\vert\xi_\h\vert\vert(\xi-\eta)_\h\vert}\right)$. For these we have 
\begin{equation}
\begin{aligned}
 &\abs{V'_{\xi-\eta}\abs{\xi-\eta}^{-1}}\lesssim 2^{-k_1},\qquad \abs{V'_{\xi-\eta}V_\eta(\abs{\xi-\eta}^{-1})}\lesssim1+2^{k_2-k_1},\\
 &\abs{V'_{\xi-\eta}\left(\frac{\xi_\h\cdot(\xi-\eta)_\h}{\vert\xi_\h\vert\vert(\xi-\eta)_\h\vert}\right)}\lesssim 1,\qquad V'_{\xi-\eta}V_\eta\left(\frac{\xi_\h\cdot(\xi-\eta)_\h}{\vert\xi_\h\vert\vert(\xi-\eta)_\h\vert}\right)\lesssim (1+2^{p_1+k_1-p_2-k_2})\cdot 2^{p_2+k_2-p_1-k_1},
\end{aligned} 
\end{equation}
and $V'_{\xi-\eta}V_\eta\Lambda(\xi-\eta)=V'_{\xi-\eta}V_\eta\Phi$. Collecting terms gives \eqref{eq:2vfonUpsPhi}.

\end{proof}

\subsection{Vector Fields and Multipliers}\label{sec:morevf}

We begin with an estimate for vector fields on ``angles'', as they appear in the multipliers of Lemma \ref{lem:IERmult}.

\begin{lemma}[Vector Fields $V$ and ``angles'']
 We have that
 \begin{equation}\label{eq:Vonangle2}
 \begin{aligned}
  V_\eta\left(\frac{\eta_\h\cdot(\xi_\h-\eta_\h)}{\abs{\eta_\h}\abs{\xi_\h-\eta_\h}}\right)
  &=
  \begin{cases}
   \frac{\abs{\eta_\h}}{\abs{\xi_\h-\eta_\h}}\left[\left(\frac{\eta_\h\cdot(\xi_\h-\eta_\h)}{\abs{\eta_\h}\abs{\xi_\h-\eta_\h}}\right)^2-1\right], &V=S,\\
   \frac{\eta_\h^\perp\cdot(\xi_\h-\eta_\h)}{\abs{\eta_\h}\abs{\xi_\h-\eta_\h}}\left[\frac{\abs{\eta_\h}}{\abs{\xi_\h-\eta_\h}}\frac{\eta_\h\cdot(\xi_\h-\eta_\h)}{\abs{\eta_\h}\abs{\xi_\h-\eta_\h}}+1\right],&V=\Omega,
  \end{cases}
 \end{aligned}
 \end{equation}
and 
\begin{equation}\label{eq:Vonangle1}
\begin{aligned}
 V_\eta\left(\frac{\xi_\h\cdot(\xi_\h-\eta_\h)}{\abs{\xi_\h}\abs{\xi_\h-\eta_\h}}\right)
 &=\frac{\abs{\eta_\h}}{\abs{\xi_\h-\eta_\h}}
 \begin{cases}
  \frac{\xi_\h\cdot(\xi_\h-\eta_\h)}{\abs{\xi_\h}\abs{\xi_\h-\eta_\h}}\frac{\eta_\h\cdot(\xi_\h-\eta_\h)}{\abs{\eta_\h}\abs{\xi_\h-\eta_\h}}-\frac{\xi_\h\cdot\eta_\h}{\abs{\xi_\h}\abs{\eta_\h}}, &V=S,\\
  \frac{\xi_\h\cdot(\xi_\h-\eta_\h)}{\abs{\xi_\h}\abs{\xi_\h-\eta_\h}}\frac{\eta_\h^\perp\cdot(\xi_\h-\eta_\h)}{\abs{\eta_\h}\abs{\xi_\h-\eta_\h}}-\frac{\xi_\h\cdot\eta_\h^\perp}{\abs{\xi_\h}\abs{\eta_\h}}, &V=\Omega.                                                                                                                                                                                                                                                                                                     \end{cases}
\end{aligned}  
\end{equation}
\end{lemma}
\begin{proof}
We compute directly that
\begin{equation}\label{eq:Vonangle0}
 V_\eta\left(\frac{\xi_\h\cdot\eta_\h}{\abs{\xi_\h}\abs{\eta_\h}}\right)=
 \begin{cases}0,&V=S,\\ 
 \frac{\xi_\h\cdot\eta_\h^\perp}{\abs{\xi_\h}\abs{\eta_\h}},&V=\Omega.
 \end{cases}
\end{equation}
Next we have that for $\alpha\in\R$
\begin{equation}\label{eq:Vonabs}
 V_\eta\abs{\xi_\h-\eta_\h}^\alpha=-\alpha\frac{\abs{\eta_\h}}{\abs{\xi_\h-\eta_\h}}\abs{\xi_\h-\eta_\h}^\alpha\frac{\xi_\h-\eta_\h}{\abs{\xi_\h-\eta_\h}}\cdot\begin{cases}
  \frac{\eta_\h}{\abs{\eta_\h}}, &V=S,\\
  \frac{\eta_\h^\perp}{\abs{\eta_\h}}, &V=\Omega.                                                                                                                                                     \end{cases}
\end{equation}
Hence
\begin{equation}
\begin{aligned}
 V_\eta\left(\frac{\xi_\h\cdot(\xi_\h-\eta_\h)}{\abs{\xi_\h}\abs{\xi_\h-\eta_\h}}\right)&=V_\eta\left(\frac{1}{\abs{\xi_\h-\eta_\h}}\left[\abs{\xi_\h}-\frac{\xi_\h\cdot\eta_\h}{\abs{\xi_\h}\abs{\eta_\h}}\abs{\eta_\h}\right]\right)\\
 &=\frac{\abs{\eta_\h}}{\abs{\xi_\h-\eta_\h}}
 \begin{cases}
  \frac{\xi_\h\cdot(\xi_\h-\eta_\h)}{\abs{\xi_\h}\abs{\xi_\h-\eta_\h}}\frac{\eta_\h\cdot(\xi_\h-\eta_\h)}{\abs{\eta_\h}\abs{\xi_\h-\eta_\h}}-\frac{\xi_\h\cdot\eta_\h}{\abs{\xi_\h}\abs{\eta_\h}}, &V=S,\\
  \frac{\xi_\h\cdot(\xi_\h-\eta_\h)}{\abs{\xi_\h}\abs{\xi_\h-\eta_\h}}\frac{\eta_\h^\perp\cdot(\xi_\h-\eta_\h)}{\abs{\eta_\h}\abs{\xi_\h-\eta_\h}}-\frac{\xi_\h\cdot\eta_\h^\perp}{\abs{\xi_\h}\abs{\eta_\h}}, &V=\Omega.                                                                                                                                                                                                                                                                                                     \end{cases}
\end{aligned}  
\end{equation}
Similarly we have
\begin{equation}
\begin{aligned}
 V_\eta\left(\frac{\eta_\h\cdot(\xi_\h-\eta_\h)}{\abs{\eta_\h}\abs{\xi_\h-\eta_\h}}\right)&=V_\eta\left(\frac{1}{\abs{\xi_\h-\eta_\h}}\left[\frac{\xi_\h\cdot\eta_\h}{\abs{\xi_\h}\abs{\eta_\h}}\abs{\xi_\h}-\abs{\eta_\h}\right]\right)\\
 &=
 \begin{cases}
  \frac{\abs{\eta_\h}}{\abs{\xi_\h-\eta_\h}}\left[\left(\frac{\eta_\h\cdot(\xi_\h-\eta_\h)}{\abs{\eta_\h}\abs{\xi_\h-\eta_\h}}\right)^2-1\right], &V=S,\\
  \frac{\eta_\h^\perp\cdot(\xi_\h-\eta_\h)}{\abs{\eta_\h}\abs{\xi_\h-\eta_\h}}\left[\frac{\abs{\eta_\h}}{\abs{\xi_\h-\eta_\h}}\frac{\eta_\h\cdot(\xi_\h-\eta_\h)}{\abs{\eta_\h}\abs{\xi_\h-\eta_\h}}+1\right],&V=\Omega.
 \end{cases}
\end{aligned}
\end{equation}
Here we computed that by \eqref{eq:Vonabs} and \eqref{eq:Vonangle0} there holds
\begin{equation*}
\begin{aligned}
 &\Omega_\eta\left(\frac{1}{\abs{\xi_\h-\eta_\h}}\left[\frac{\xi_\h\cdot\eta_\h}{\abs{\xi_\h}\abs{\eta_\h}}\abs{\xi_\h}-\abs{\eta_\h}\right]\right)=\Omega_\eta(\abs{\xi_\h-\eta_\h}^{-1})\frac{\eta_\h\cdot (\xi_\h-\eta_\h)}{\abs{\eta_\h}}+\frac{1}{\abs{\xi_\h-\eta_\h}}\Omega_\eta\left(\frac{\xi_\h\cdot\eta_\h}{\abs{\xi_\h}\abs{\eta_\h}}\right)\abs{\xi_\h}\\
 &\qquad=\frac{\abs{\eta_\h}}{\abs{\xi_\h-\eta_\h}}\abs{\xi_\h-\eta_\h}^{-1}\frac{(\xi_\h-\eta_\h)\cdot\eta_\h^\perp}{\abs{\xi_\h-\eta_\h}\abs{\eta_\h}}\frac{\eta_\h\cdot (\xi_\h-\eta_\h)}{\abs{\eta_\h}}+\frac{1}{\abs{\xi_\h-\eta_\h}}\frac{\xi_\h\cdot\eta_\h^\perp}{\abs{\xi_\h}\abs{\eta_\h}}\abs{\xi_\h}\\
 &\qquad=\frac{\abs{\eta_\h}}{\abs{\xi_\h-\eta_\h}}\frac{(\xi_\h-\eta_\h)\cdot\eta_\h^\perp}{\abs{\xi_\h-\eta_\h}\abs{\eta_\h}}\frac{\eta_\h\cdot (\xi_\h-\eta_\h)}{\abs{\eta_\h}\abs{\xi_\h-\eta_\h}}+\frac{(\xi_\h-\eta_\h)\cdot\eta_\h^\perp}{\abs{\xi_\h-\eta_\h}\abs{\eta_\h}}.
\end{aligned} 
\end{equation*}

\end{proof}

For future reference we note also that
\begin{equation}\label{eq:Vonfrac}
 V'_{\xi-\eta}\left(\frac{\abs{\eta_\h}}{\abs{\xi_\h-\eta_\h}}\right)\lesssim 1+\frac{\abs{\eta_\h}}{\abs{\xi_\h-\eta_\h}},
\end{equation}
and the analogous computations as \eqref{eq:Vonangle0}--\eqref{eq:Vonangle2} hold for $V'_{\xi-\eta}$ with the roles of $\eta$ and $\xi-\eta$ exchanged.

Next we give bounds for the vector fields on the multipliers that arise in our formulation \eqref{eq:IER_disp} of the rotating Euler equations \eqref{eq:IER}. 
\begin{lemma}[Vector Fields $V$ and Multipliers]\label{lem:vfmult}
The multipliers $\m=m_i^{\mu\nu}$, $i\in\{1,2\},\mu,\nu\in\{-,+\}$, of \eqref{eq:IER_disp} as in Lemma \ref{lem:IERmult} satisfy
\begin{equation}\label{eq:vfsumonmult}
 (S_\xi+S_\eta)\m(\xi,\eta)=\m(\xi,\eta),\quad (\Omega_\xi+\Omega_\eta)\m(\xi,\eta)=0.
\end{equation}
Moreover, they are bounded by
\begin{equation}\label{eq:simpmultbd}
 \abs{\m(\xi,\eta)}\lesssim 2^k 2^{\max\{p,p_1,p_2\}+\max\{q,q_1,q_2\}}=:C_\m,
\end{equation}
and we have for $V\in\{S,\Omega\}$
\begin{equation}\label{eq:vfonmult}
\begin{aligned}
 \abs{V_\eta [\m(\xi,\eta)]}\lesssim &\left\{1+\min\{2^{p_2+k_2}2^{-p_1-k_1},2^{p+k}2^{-p_1-k_1}\}\right\}C_\m +2^{k-p_1}\abs{V_\eta\Lambda(\xi-\eta)},
\end{aligned}
\end{equation}

\end{lemma}

\begin{remark}
 Since our equations commute with the vectorfields $S,\Omega$, properties like \eqref{eq:vfsumonmult} are expected to hold for the full nonlinear symbols. Here the interesting point is that they also hold just for the ``constituent pieces'' $\m=m_i^{\mu\nu}$.
\end{remark}

\begin{proof}
From the expressions \eqref{eq:mult1} and \eqref{eq:mult2}, the bound \eqref{eq:simpmultbd} directly follows. Towards \eqref{eq:vfsumonmult} we compute that
\begin{equation}
 (V_\xi+V_\eta)\Lambda(\zeta)=0, \quad\zeta\in\{\xi,\xi-\eta,\eta\},
\end{equation}
and 
\begin{equation}
 \Omega_\xi \sqrt{1-\Lambda^2}(\xi)=S_\xi \sqrt{1-\Lambda^2}(\xi)=0,\quad \Omega_\xi\abs{\xi}=0,\quad S_\xi\abs{\xi}=\abs{\xi},
\end{equation}
as well as 
\begin{equation}
 S_\xi\left(\frac{\xi_\h}{\abs{\xi}}\right)=S_\xi\left(\frac{\xi_\h}{\abs{\xi_\h}}\right)=0,\quad \Omega_\xi\left(\frac{\xi_\h}{\abs{\xi}}\right)=\frac{\xi_\h^\perp}{\abs{\xi}},\quad \Omega_\xi\left(\frac{\xi_\h}{\abs{\xi_\h}}\right)=\frac{\xi_\h^\perp}{\abs{\xi_\h}}.
\end{equation} 
It thus follows that for $V\in\{S,\Omega\}$ we have
\begin{equation}\label{eq:vfsumangle}
 (V_\xi+V_\eta)\left(\frac{\xi_\h}{\abs{\xi_\h}}\cdot\frac{\zeta}{\abs{\zeta}}\right)=0,\quad \zeta\in\{\eta_\h,\eta_\h^\perp,\xi_\h-\eta_\h,\xi_\h^\perp-\eta_\h^\perp\}.
\end{equation}
Since the Riesz transforms in \eqref{eq:mult1}, \eqref{eq:mult2} appear only as ``angles'' between an input and the output, this suffices to conclude the proof of \eqref{eq:vfsumonmult}.

Furthermore, to establish \eqref{eq:vfonmult}, we note that 
\begin{equation}
 V_\eta\sqrt{1-\Lambda^2}(\xi-\eta)=\frac{\Lambda(\xi-\eta)}{\sqrt{1-\Lambda^2}(\xi-\eta)}V_\eta\Lambda(\xi-\eta).
\end{equation}
Finally we compute that
\begin{equation}
\begin{aligned}
 S_\eta\left(\frac{\xi_\h-\eta_\h}{\abs{\xi_\h-\eta_\h}}\right)&=\frac{\abs{\eta_\h}}{\abs{\xi_\h-\eta_\h}}\left(-\frac{\eta_\h}{\abs{\eta_\h}}+\frac{\xi_\h-\eta_\h}{\abs{\xi_\h-\eta_\h}}\left(\frac{\xi_\h-\eta_\h}{\abs{\xi_\h-\eta_\h}}\cdot\frac{\eta_\h}{\abs{\eta_\h}}\right)\right),\\
 \Omega_\eta\left(\frac{\xi_\h-\eta_\h}{\abs{\xi_\h-\eta_\h}}\right)&=\frac{\abs{\eta_\h}}{\abs{\xi_\h-\eta_\h}}\left(-\frac{\eta_\h^\perp}{\abs{\eta_\h}}+\frac{\xi_\h-\eta_\h}{\abs{\xi_\h-\eta_\h}}\left(\frac{\xi_\h-\eta_\h}{\abs{\xi_\h-\eta_\h}}\cdot\frac{\eta_\h^\perp}{\abs{\eta_\h}}\right)\right),
\end{aligned} 
\end{equation}
which already gives the trivial bounds $V_\eta\left(\frac{\xi_\h-\eta_\h}{\abs{\xi_\h-\eta_\h}}\right)\lesssim 2^{p_2+k_2}2^{-p_1-k_1}$. For our multipliers these can be improved as follows: note that $(S_\xi+S_\eta)\left(\frac{\xi_\h-\eta_\h}{\abs{\xi_\h-\eta_\h}}\right)=0$, so
\begin{equation}
 \abs{S_\eta\left(\frac{\xi_\h-\eta_\h}{\abs{\xi_\h-\eta_\h}}\right)}\lesssim 2^{-p_1-k_1}\min\{2^{p+k},2^{p_2+k_2}\}.
\end{equation}
For $\Omega$ we use \eqref{eq:vfsumangle} to deduce that 
\begin{equation}
\begin{aligned}
 \abs{\Omega_\eta\left(\frac{\xi_\h}{\abs{\xi_\h}}\cdot\frac{\xi_\h-\eta_\h}{\abs{\xi_\h-\eta_\h}}\right)}&=\abs{-\frac{\xi_\h^\perp}{\abs{\xi_\h}}\cdot\frac{\xi_\h-\eta_\h}{\abs{\xi_\h-\eta_\h}}-\frac{\xi_\h}{\abs{\xi_\h}}\cdot\Omega_\xi\left(\frac{\xi_\h-\eta_\h}{\abs{\xi_\h-\eta_\h}}\right)}\\
 &\lesssim \min\{2^{p_2+k_2}2^{-p_1-k_1},1+2^{p+k}2^{-p_1-k_1}\},
\end{aligned} 
\end{equation}
and similarly for the other ``angles'' as in \eqref{eq:vfsumangle}. This gives the following, more precise version of \eqref{eq:vfonmult}
\begin{equation}
\begin{aligned}
 \abs{S_\eta [\m(\xi,\eta)]}&\lesssim \left\{1+\min\{2^{p+k},2^{p_2+k_2}\}2^{-p_1-k_1}\right\}\cdot C_\m+2^k[2^{\max\{p,p_1,p_2\}}+2^{\max\{q,q_1,q_2\}}2^{q_1-p_1}]\abs{S_\eta\Lambda(\xi-\eta)},\\
 \abs{\Omega_\eta [\m(\xi,\eta)]}&\lesssim \left\{1+\min\{2^{p_2+k_2}2^{-p_1-k_1},1+2^{p+k}2^{-p_1-k_1}\}\right\}\cdot C_\m\\
 &\qquad +2^k[2^{\max\{p,p_1,p_2\}}+2^{\max\{q,q_1,q_2\}}2^{q_1-p_1}]\abs{\Omega_\eta\Lambda(\xi-\eta)}.
\end{aligned} 
\end{equation}
\end{proof}

Next, when $\Ups$ hits one of the multipliers, we have that

\begin{lemma}\label{lem:Ups_mult}
 We have that for any multiplier $\m$ as in Lemma \ref{lem:IERmult}, there holds that
 \begin{equation}\label{eq:Ups_mult}
  \abs{\Upsilon_\xi\m}\lesssim 2^{k}[1+2^{k-k_1}+2^{k-k_1}2^{q+p_2-p_1}]=:C_{\Ups\m}.
 \end{equation}
\end{lemma}

\begin{proof}
 It suffices to compute $\Upsilon$ on the various components of the multiplier $\m$. To this end, we recall that 
 \begin{equation}
 \nabla\Lambda(\xi)=-\frac{\xi_3}{\abs{\xi}^3}\xi_\h+\frac{\abs{\xi_\h}^2}{\abs{\xi}^3}\vec{e}_3.
\end{equation}
 We then have, using spherical coordinates, 
\begin{equation}
\begin{split}
 \Upsilon_\xi\abs{\xi}&=\partial_\phi\rho=0,\qquad\Upsilon_\xi\vert\xi_\h\vert=\partial_\phi(\rho\sin\phi)=\rho\Lambda\\
 \Upsilon_\xi\Lambda(\xi)&=-\sqrt{1-\Lambda^2}\cdot\partial_\Lambda\Lambda=-\sqrt{1-\Lambda^2}(\xi),\qquad 
 \Upsilon_\xi\sqrt{1-\Lambda^2}(\xi)=\partial_\phi\sin\phi=\Lambda(\xi).
 \end{split}
\end{equation}

Similarly we have that
\begin{equation}
\begin{aligned}
\Upsilon_\xi\vert\xi-\eta\vert&=\frac{\xi-\eta}{\vert\xi-\eta\vert}\cdot\left\{\frac{\Lambda(\xi)}{\sqrt{1-\Lambda^2}(\xi)}\xi_\h-\sqrt{1-\Lambda^2}(\xi)\vert\xi\vert\vec{e}_3\right\}\\
&=\vert\xi\vert\left[\Lambda(\xi)\sqrt{1-\Lambda^2}(\xi-\eta)\frac{\xi_\h\cdot(\xi-\eta)_\h}{\vert\xi_\h\vert\vert(\xi-\eta)_\h\vert}-\sqrt{1-\Lambda^2}(\xi)\Lambda(\xi-\eta)\right],\\
\Upsilon_\xi\vert(\xi-\eta)_\h\vert&=\frac{(\xi-\eta)_\h}{\vert(\xi-\eta)_\h\vert}\cdot\left\{\frac{\Lambda(\xi)}{\sqrt{1-\Lambda^2}(\xi)}\xi_\h-\sqrt{1-\Lambda^2}(\xi)\vert\xi\vert\vec{e}_3\right\}=\vert\xi\vert\Lambda(\xi)\frac{\xi_\h\cdot(\xi-\eta)_\h}{\vert\xi_\h\vert\vert(\xi-\eta)_\h\vert},\\
\end{aligned}
\end{equation}
and
\begin{equation}
\begin{aligned}
 \Upsilon_\xi\Lambda(\xi-\eta)&=\frac{\Lambda(\xi)}{\sqrt{1-\Lambda^2}(\xi)}\xi_\h\cdot(-(\xi_3-\eta_3)\frac{\xi_\h-\eta_\h}{\abs{\xi-\eta}^3})-\abs{\xi}\sqrt{1-\Lambda^2(\xi)}\frac{\abs{\xi_\h-\eta_\h}^2}{\abs{\xi-\eta}^3}\\
 &=-\frac{\abs{\xi}}{\abs{\xi-\eta}}\sqrt{1-\Lambda^2}(\xi-\eta)\left[\Lambda(\xi)\Lambda(\xi-\eta)\frac{\xi_\h\cdot(\xi_\h-\eta_\h)}{\abs{\xi_\h}\abs{\xi_\h-\eta_\h}}+\sqrt{1-\Lambda^2}(\xi)\sqrt{1-\Lambda^2}(\xi-\eta)\right],\\
 \Upsilon_\xi\sqrt{1-\Lambda^2}(\xi-\eta)&=-\frac{\Lambda(\xi-\eta)}{\sqrt{1-\Lambda^2}(\xi-\eta)}\Upsilon_\xi\Lambda(\xi-\eta)\\
 &=-\frac{\abs{\xi}}{\abs{\xi-\eta}}\Lambda(\xi-\eta)\left[\Lambda(\xi)\Lambda(\xi-\eta)\frac{\xi_\h\cdot(\xi_\h-\eta_\h)}{\abs{\xi_\h}\abs{\xi_\h-\eta_\h}}-\sqrt{1-\Lambda^2}(\xi)\sqrt{1-\Lambda^2}(\xi-\eta)\right].
\end{aligned} 
\end{equation}

Finally, we compute the critical ``horizontal angle'' terms in the multipliers: We have
\begin{equation}
\begin{split}
\Upsilon\left(\frac{\xi_\h\cdot\eta_\h}{\abs{\xi_\h}\abs{\eta_\h}}\right)&=\partial_\phi\left\{A_\eta\cos(\theta)+B_\eta\sin(\theta)\right\}=0\\
\Upsilon_\xi\left(\frac{\xi_\h\cdot(\xi-\eta)_\h}{\abs{\xi_\h}\abs{(\xi-\eta)_\h}}\right)&=\Upsilon_\xi\left\{\frac{1}{\vert(\xi-\eta)_\h\vert}\left\{\vert\xi_\h\vert-\frac{\xi_\h\cdot\eta_\h}{\vert\xi_\h\vert\vert\eta_\h\vert}\vert\eta_\h\vert\right\}\right\}\\
&=\frac{\vert\xi\vert\Lambda(\xi)}{\vert(\xi-\eta)_\h\vert}\left\{1-\left(\frac{\xi_\h\cdot(\xi-\eta)_\h}{\vert\xi_\h\vert\vert(\xi-\eta)_\h\vert}\right)^2\right\}
\end{split}
\end{equation}
Analogously we compute that 
\begin{equation}
 \nabla_{\xi_\h}\left(\frac{\xi_\h\cdot(\xi_\h-\eta_\h)}{\abs{\xi_\h}\abs{\xi_\h-\eta_\h}}\right)=\frac{\xi_\h+(\xi_\h-\eta_\h)}{\abs{\xi_\h}\abs{\xi_\h-\eta_\h}}-\frac{\xi_\h\cdot(\xi_\h-\eta_\h)}{\abs{\xi_\h}\abs{\xi_\h-\eta_\h}}\left(\frac{\xi_\h}{\abs{\xi_\h}^2}+\frac{\xi_\h-\eta_\h}{\abs{\xi_\h-\eta_\h}^2}\right),
\end{equation}
which implies that
\begin{equation}
\begin{aligned}
 \Upsilon_\xi\left(\frac{\xi_\h\cdot(\xi_\h-\eta_\h)}{\abs{\xi_\h}\abs{\xi_\h-\eta_\h}}\right)&=\frac{\Lambda(\xi)}{\sqrt{1-\Lambda^2}(\xi)}\,\xi_\h\cdot\nabla_{\xi_\h}\left(\frac{\xi_\h\cdot(\xi_\h-\eta_\h)}{\abs{\xi_\h}\abs{\xi_\h-\eta_\h}}\right)\\
 &=\frac{\Lambda(\xi)}{\sqrt{1-\Lambda^2}(\xi)}\frac{\abs{\xi_\h}}{\abs{\xi_\h-\eta_\h}}\left[1-\left(\frac{\xi_\h\cdot(\xi_\h-\eta_\h)}{\abs{\xi_\h}\abs{\xi_\h-\eta_\h}}\right)^2\right]\\
 &=\frac{\Lambda(\xi)\abs{\xi}}{\abs{\xi_\h-\eta_\h}}\left(\frac{\xi_\h\cdot(\xi_\h-\eta_\h)^\perp}{\abs{\xi_\h}\abs{\xi_\h-\eta_\h}}\right)^2.
\end{aligned}
\end{equation}

\end{proof}

Similarly one checks the following lemma by direct computation.
\begin{lemma}\label{lem:VUpsmult}
 With $C_{\Ups\m}$ defined in \eqref{eq:Ups_mult}, we have
 \begin{equation}\label{eq:VUpsmult}
  \abs{V_{\xi-\eta}\Ups\m}\lesssim C_{\Ups\m}\cdot 2^{k_1-k_2}[1+2^{p_1-p_2}]
 \end{equation}
\end{lemma}
\begin{proof}
 It suffices to compute $V_{\xi-\eta}$ on the ``angles'' and on $\Lambda(\eta)$ and $\sqrt{1-\Lambda^2}(\eta)$, where we can bound by a factor of $2^{k_1-k_2}[1+2^{p_1-p_2}]$. This gives the claim.
\end{proof}

\subsection{Bounds for two integrations by parts}\label{ssec:prelim_bds}
Here we give some bounds for the multipliers that arise when integrating by parts twice. For this, a basic fact that we shall use repeatedly is the triangle inequality, in the form that
\begin{equation}\label{eq:freq_triang}
 \begin{aligned}
  k=k_{\min}:\quad &2^{p_1-p_2}\lesssim 1+2^{k-k_2}2^{p-p_2},\quad 2^{p_2-p_1}\lesssim 1+2^{k-k_1}2^{p-p_1},\\
  k_2=k_{\min}:\quad &2^{p_1-p}\lesssim 1+2^{k_2-k}2^{p_2-p},\quad 2^{p-p_1}\lesssim 1+2^{k_2-k_1}2^{p_2-p_1}.
 \end{aligned}
 \end{equation}
In particular, we note that we have
\begin{alignat}{2}
 2^{-k}(2^{p_2-p_1}+2^{p_1-p_2})\Sz&\lesssim 2^\frac{q+k}{2}2^p,\qquad&&\textnormal{when }k=k_{\min},\\
 2^{-k_2}(2^{p-p_1}+2^{p_1-p})\Sz&\lesssim 2^\frac{q+k}{2}2^{p_2},\qquad&& \textnormal{when }k_2=k_{\min}.
\end{alignat}

\begin{lemma}\label{lem:ibp_mult_bds}
Assume that $\abs{\bar\sigma}\gtrsim 2^{k_{\max}+k_{\min}+q_{\max}}$ and let $V,V'\in\{S,\Omega\}$ be such that
 \begin{equation}
  \abs{V_\eta\Phi}\sim 2^{-2k_1+p_1}\abs{\bar\sigma},\quad \abs{V'_{\xi-\eta}\Phi}\sim 2^{-2k_2+p_2}\abs{\bar\sigma},
 \end{equation}
 so that we can integrate by parts in $V_\eta$ and in $V'_{\xi-\eta}$. Assume moreover that $k_2\leq k_1$.
 
 Then there holds that
 \begin{equation}\label{eq:bds_1ibp}
 \begin{aligned}
  \abs{\frac{\m\Ups_\xi\Phi}{V_\eta\Phi}}&\lesssim 2^{2k_1-k_2}[2^{p-p_1}+2^{k-k_1}]
  \lesssim 2^{2k_{\max}-k_2}(1+2^{p-p_1}),\\
  \abs{V_\eta\left(\frac{\m\Ups_\xi\Phi}{V_\eta\Phi}\right)}&
  \lesssim 
 (1+2^{p-p_1}) \begin{cases}
   2^{k_{\max}}+2^{k}2^{p-p_1},&k=k_{\min},\\
   2^{2k_{\max}-k_2}+2^{k_{\max}}2^{p_2-p_1},&k_2=k_{\min},
  \end{cases}
 \end{aligned}
 \end{equation}
 and similarly
 \begin{equation}\label{eq:bds_1ibp'}
  \abs{V'_{\xi-\eta}\left(\frac{\m\Ups_\xi\Phi}{V_\eta\Phi}\right)}\lesssim 
  \begin{cases}
   2^{k_{\max}}(2^{p-p_1}+2^{p-p_2})+2^{k}(1+2^{p_1-p_2}),&k=k_{\min},\\
   2^{k_1}2^{p_2-p_1}+2^{2k_1-k_2}[1+2^{k_1-k_2}2^{p_1-p_2}],&k_2=k_{\min}.
  \end{cases}
 \end{equation}
 
 Moreover, we have
 \begin{equation}\label{eq:mixvf_quot}
 \begin{aligned}
  \abs{\frac{V'_{\xi-\eta}V_\eta\Phi}{V_\eta\Phi}}&\lesssim
  \begin{cases}
   2^{k_{\max}-k_2+p_1},&V=V'=\Omega,\\
   1,&\text{else}.
  \end{cases}
  \\
  \abs{\frac{V'_{\xi-\eta}V^2_\eta\Phi}{V_\eta\Phi}}&\lesssim
 \begin{cases}
  1+2^{k-k_1}2^{p-p_1},&k=k_{\min},\\
  1+2^{k_2-k_1}2^{p_2-p_1}+2^{p}2^{k-k_2},&k_2=k_{\min}.
 \end{cases}
 \end{aligned} 
 \end{equation}
 and 
 \begin{equation}\label{eq:bds_2ibp}
  \abs{V'_{\xi-\eta}\left(V_\eta\left(\frac{\m\Ups_\xi\Phi}{V_\eta\Phi}\right)\right)}\lesssim 
  \begin{cases}
   (2^{k_{\max}}+2^{k}2^{p-p_1})[1+2^{p-p_1}+2^{p-p_2}],&k=k_{\min},\\
   2^{2k_{\max}-k_2-p_2}(1+2^{p_2-p_1}+ 2^{k_{\max}-k_2})(2^{p_1}+2^{p})&\\
   \quad +(1+2^{p_2-p_1})(2^{k_{\max}}+2^{k_2}2^{p_2-p_1}),&k_2=k_{\min}.
  \end{cases}
 \end{equation}
\end{lemma}

\begin{remark}
We highlight that in the above estimates the highest losses in $k_{\min}$ never come together with the highest losses in $p_1$ or $p_2$. This is crucial for the two integrations by parts that we carry out in Section \ref{ssec:2ibp}. For this it is essential that the consecutive integrations by parts be carried out in the two different directions: the heuristic -- confirmed by the above estimates -- is that one does not repeatedly lose too many powers of the same localization sizes (compare for example \eqref{eq:bds_1ibp} with \eqref{eq:bds_2ibp}).
\end{remark}

\begin{proof}
 The first estimate in \eqref{eq:bds_1ibp} is direct using \eqref{eq:UpsPhi_bd}: we have that
 \begin{equation}
 \begin{aligned}
  \abs{\frac{\m\Ups_\xi\Phi}{V_\eta\Phi}}&\lesssim 2^{2k_1+k}[2^{p-p_1}+2^{k-k_1}]\cdot 2^{-k_{\max}-k_{\min}}\lesssim 2^{2k_1-k_2}[2^{p-p_1}+2^{k-k_1}] \lesssim 2^{2k_{\max}-k_2}[1+2^{p-p_1}].
 \end{aligned} 
 \end{equation}
 For the second, we compute that
 \begin{equation}
  V_\eta\left(\frac{\m\Ups_\xi\Phi}{V_\eta\Phi}\right)=\frac{1}{V_\eta\Phi}\left[(V_\eta\m)\Ups_\xi\Phi+\m V_\eta\Ups_\xi\Phi-\m\Ups_\xi\Phi\frac{V_\eta^2\Phi}{V_\eta\Phi}\right].
 \end{equation}
 By Lemma \ref{lem:vfmult} we have
 \begin{equation}\label{eq:vfmupsquot}
 \begin{aligned}
  \abs{\frac{(V_\eta\m)\Ups_\xi\Phi}{V_\eta\Phi}}&\lesssim [1+2^{-k_1-p_1}\min\{2^{p+k},2^{p_2+k_2}\}]C_{\m}\abs{\frac{\Ups_\xi\Phi}{V_\eta\Phi}}+2^{k-p_1}\abs{\Ups_\xi\Phi}\\
  &\lesssim [1+2^{-k_1-p_1}\min\{2^{p+k},2^{p_2+k_2}\}]\cdot 2^{2k_1-k_2}[2^{p-p_1}+2^{k-k_1}]+2^{k-p_1}[2^p+2^{k-k_1}2^{p_1}]\\
  &\lesssim
  \begin{cases}
   2^{k_1}[1+2^{k-k_1}2^{p-p_1}][2^{p-p_1}+2^{k-k_1}]+2^{k-p_1}[2^p+2^{k-k_1}2^{p_1}],&k=k_{\min},\\
   2^{2k_1-k_2}[1+2^{k_2-k_1}2^{p_2-p_1}][2^{p-p_1}+1]+2^{k}[2^{p-p_1}+1],&k_2=k_{\min},
  \end{cases}\\
  &\lesssim
  \begin{cases}
   2^{p-p_1}[2^{k_1}+2^{k}(2^{p-p_1}+2^{k-k_1})]+2^{2k-k_1},&k=k_{\min},\\
   (1+2^{p-p_1})[2^k+2^{2k_1-k_2}+2^{k_1}2^{p_2-p_1}],&k_2=k_{\min},
  \end{cases}\\  
  &\lesssim 2^{k_{\max}}(1+2^{p-p_1})\cdot [1+2^{k_1-k_2}+2^{p_2-p_1}].
 \end{aligned} 
 \end{equation}
 Moreover, from \eqref{eq:vfonUpsPhi} it follows that
 \begin{equation}\label{eq:mvfupsquot}
 \begin{aligned}
  \abs{\frac{\m V_\eta\Ups_\xi\Phi}{V_\eta\Phi}}&\lesssim \abs{\m}2^{k-k_1-p_1}+\abs{\m}(1+2^{p_2-p_1})2^{k+k_2}\abs{\bar\sigma}^{-1}\lesssim 2^{p_{\max}-p_1}2^k
 \end{aligned}
 \end{equation}
 and by \eqref{eq:freq_triang}, \eqref{eq:bds_1ibp} and \eqref{eq:vfquotient_apdx}
 \begin{equation}
 \begin{aligned}
  \abs{\frac{\m \Ups_\xi\Phi}{V_\eta\Phi}\frac{V_\eta^2\Phi}{V_\eta\Phi}}&\lesssim 2^{2k_1-k_2}[2^{p-p_1}+2^{k-k_1}][1+2^{k_2-k_1}+2^{-p_1-k_1}\min\{2^{p+k},2^{p_2+k_2}\}+2^{p+p_2-p_1}]\\
  &\lesssim
  \begin{cases}
   2^{k_{\max}}[2^{p-p_1}+2^{k-k_1}][1+2^{k-k_1}2^{p-p_1}+2^{p}(1+2^{k-k_1}2^{p-p_1})],&k=k_{\min},\\
   2^{2k_{\max}-k_2}[2^{p-p_1}+1][1+2^{k_2-k_1}2^{p_2-p_1}+2^{p_2}(1+2^{k_2-k_1}2^{p_2-p_1})],&k_2=k_{\min},
  \end{cases}\\ 
  &\lesssim
  \begin{cases}
   2^{k_{\max}}[2^{p-p_1}+2^{k-k_1}][1+2^{k-k_1}2^{p-p_1}],&k=k_{\min},\\
   2^{2k_{\max}-k_2}[2^{p-p_1}+1][1+2^{k_2-k_1}2^{p_2-p_1}],&k_2=k_{\min},
  \end{cases}\\
  &\lesssim
  \begin{cases}
   2^{k_{\max}}[2^{p-p_1}+2^{k-k_1}][1+2^{k-k_1}2^{p-p_1}],&k=k_{\min},\\
   2^{2k_{\max}-k_2}[1+2^{p-p_1}]+2^{k_{\max}}2^{p+p_2-2p_1},&k_2=k_{\min}.
  \end{cases}
 \end{aligned}
 \end{equation}
These combine to give
 \begin{equation}
  \abs{V_\eta\left(\frac{\m\Ups_\xi\Phi}{V_\eta\Phi}\right)}\lesssim
  \begin{cases}
   2^{k_{\max}}(1+2^{p-p_1})+2^{k}2^{2p-2p_1} ,&k=k_{\min},\\
   2^{2k_{\max}-k_2}[1+2^{p-p_1}]+2^{k_{\max}}2^{p_2-p_1}[1+2^{p-p_1}],&k_2=k_{\min},
  \end{cases}
 \end{equation}
which is the second claim of \eqref{eq:bds_1ibp}.

 The mixed quotient bounds follow from the bounds in Lemma \ref{lem:vfsizes}: The third order bounds in \eqref{eq:mixvf_quot} simply collect terms from Lemma \ref{lem:vfsizes}: In \eqref{eq:2VFmix-1} we observe that splitting $\xi_\h=(\xi_\h-\eta_\h)+\eta_\h$ gives via \eqref{eq:vf_sigma} that $\frac{\bar\sigma\cdot\xi_\h}{\abs{\xi-\eta}^3}=S_\eta\Phi+\frac{\bar\sigma\cdot\eta_\h}{\abs{\xi-\eta}^3}$, and so we have
 \begin{equation}
  \abs{\frac{S_{\xi-\eta}S^2_\eta\Phi}{S_\eta\Phi}}\lesssim 1+2^{k_2-k_1}+2^{-p_1-k_1}\min\{2^{p+k},2^{p_2+k_2}\},
 \end{equation}
 and by \eqref{eq:2VFmix-4} there holds that
 \begin{equation}
  \abs{\frac{\Omega_{\xi-\eta}\Omega^2_\eta\Phi}{\Omega_\eta\Phi}}\lesssim 1+2^{p+k}2^{p_1-k_1}.
 \end{equation}
 Next we have by \eqref{eq:2VFmix-2} that
 \begin{equation}
 \begin{aligned}
  \abs{\frac{S_{\xi-\eta}\Omega^2_\eta\Phi}{\Omega_\eta\Phi}}&\lesssim \abs{\frac{\eta_\h^\perp\cdot\xi_\h}{\abs{\xi-\eta}^2}}+\abs{\Lambda(\xi-\eta)\frac{(\xi_\h+\eta_\h)\cdot\xi_\h}{\abs{\xi-\eta}^2}}\frac{1}{\abs{\Omega_\eta\Phi}}\\
  &\lesssim 2^{p_1-k_1}\min\{2^{p+k},2^{p_2+k_2}\}+2^{q_1-p_1}(2^{p+k}+2^{p_2+k_2})2^{p+k}\abs{\bar\sigma}^{-1}\\
  &\lesssim
  \begin{cases}
   2^{p+p_1}2^{k-k_1}+2^{2p-p_1}2^{k-k_1}+2^{p+p_2-p_1},&k=k_{\min},\\
   2^{p_1+p_2}2^{k_2-k_1}+2^{p-p_1}(2^{p+k-k_2}+2^{p_2}),&k_2=k_{\min},
  \end{cases}
  \\
  &\lesssim
  \begin{cases}
   2^{p}[1+2^{k-k_1}(2^{p_1}+2^{p-p_1})],&k=k_{\min},\\
   2^{p_1+p_2}2^{k_2-k_1}+2^{p}[2^{k-k_2}+2^{p_2-p_1}],&k_2=k_{\min},
  \end{cases}
 \end{aligned} 
 \end{equation}
where we used that if $k=k_{\min}$ then $2^{p_2-p_1}\lesssim 1+2^{k-k_1}2^{p-p_1}$ and if $k_2=k_{\min}$ then $2^{p-p_1}\lesssim 1+2^{k_2-k_1}2^{p_2-p_1}$. Finally, for \eqref{eq:2VFmix-3} we compute that $\eta_3\frac{\eta_\h^\perp\cdot\xi_\h}{\abs{\xi-\eta}^3}=-\Omega_\eta+\xi_3\frac{\eta_\h^\perp\cdot\xi_\h}{\abs{\xi-\eta}^3}$, so that
\begin{equation}
\begin{aligned}
 \abs{\eta_3\frac{\eta_\h^\perp\cdot\xi_\h}{\abs{\xi-\eta}^3}}\abs{S_\eta\Phi}^{-1}&\lesssim 1+\min\{2^{q_2+k_2},2^{q+k}\}\cdot\min\{2^{p+k},2^{p_2+k_2}\}\cdot\abs{\bar\sigma}^{-1},
\end{aligned}
\end{equation}
from which we conclude that
\begin{equation}
\begin{aligned}
 \abs{\frac{\Omega_{\xi-\eta}S^2_\eta\Phi}{S_\eta\Phi}}&\lesssim 2^{k-k_1}+2^{p_1-k_1}\min\{2^{p+k},2^{p_2+k_2}\}+1+\min\{2^{p+k},2^{p_2+k_2}\}\cdot \min\{2^{q_2+k_2},2^{q+k}\}\cdot\abs{\bar\sigma}^{-1}\lesssim 1\\
\end{aligned}
\end{equation}
Collecting terms now shows that
\begin{equation}
 \abs{\frac{V'_{\xi-\eta}V^2_\eta\Phi}{V_\eta\Phi}}\lesssim
 \begin{cases}
  1+2^{k-k_1}2^{p-p_1},&k=k_{\min},\\
  1+2^{k_2-k_1}2^{p_2-p_1}+2^{p}2^{k-k_2},&k_2=k_{\min}.
 \end{cases}
\end{equation}
Similarly, the second order bounds in \eqref{eq:mixvf_quot} follow directly from \eqref{eq:1VFmix-2}-\eqref{eq:1VFmix-4}, except when $V=V'=\Omega$: there we invoke \eqref{eq:1VFmix-1}, and see that since by assumption $\abs{\Omega_\eta\Phi}\sim 2^{-2k_1+p_1}\abs{\bar\sigma}\geq \abs{S_\eta\Phi}$, it suffices to control
 \begin{equation}
  \frac{1}{\abs{\Omega_\eta\Phi}}\frac{\abs{\xi_3}}{\abs{\xi-\eta}}\frac{\abs{\xi_\h-\eta_\h}^2}{\abs{\xi-\eta}^2}\lesssim 2^{q+k+k_1+p_1}\abs{\bar\sigma}^{-1}\lesssim 2^{k_{\max}-k_2+p_1},
 \end{equation}
since by assumption $\abs{\bar\sigma}\gtrsim 2^{k_{\max}+k_{\min}+q_{\max}}$.
 
 Moving on to \eqref{eq:bds_2ibp}, we have that
 \begin{equation}
  \begin{aligned}
   V'_{\xi-\eta}\left(V_\eta\left(\frac{\m\Ups_\xi\Phi}{V_\eta\Phi}\right)\right)&=V_{\xi-\eta}'\left[\frac{V_\eta(\m\Ups_\xi\Phi)}{V_\eta\Phi}-\frac{\m\Ups_\xi\Phi}{V_\eta\Phi}\frac{V_\eta^2\Phi}{V_\eta\Phi}\right]\\
   &=\frac{V_{\xi-\eta}'V_\eta(\m\Ups_\xi\Phi)}{V_\eta\Phi}-\frac{V_\eta(\m\Ups_\xi\Phi)}{V_\eta\Phi}\frac{V'_{\xi-\eta}V_\eta\Phi}{V_\eta\Phi}-\frac{V'_{\xi-\eta}(\m\Ups_\xi\Phi)}{V_\eta\Phi}\frac{V_\eta^2\Phi}{V_\eta\Phi}\\
   &\qquad -\frac{\m\Ups_\xi\Phi}{V_\eta\Phi}\frac{V'_{\xi-\eta}V_\eta^2\Phi}{V_\eta\Phi}+2\frac{\m\Ups_\xi\Phi}{V_\eta\Phi}\frac{V_\eta^2\Phi}{V_\eta\Phi}\frac{V'_{\xi-\eta}V_\eta\Phi}{V_\eta\Phi}.
 \end{aligned}  
 \end{equation}
We treat this term by term, all of which give a contribution that can be bounded as claimed in \eqref{eq:bds_2ibp}:
\begin{itemize}
 \item We have by \eqref{eq:mixvf_quot}, \eqref{eq:bds_1ibp} and \eqref{eq:vfquotient_apdx} that
 \begin{equation}
 \begin{aligned}
  &\abs{\frac{\m\Ups_\xi\Phi}{V_\eta\Phi}\frac{V_\eta^2\Phi}{V_\eta\Phi}\frac{V'_{\xi-\eta}V_\eta\Phi}{V_\eta\Phi}}\lesssim 2^{2k_1-k_2}[2^{p-p_1}+2^{k-k_1}]\cdot [1+2^{k_{\max}-k_2+p_1}]\\
  &\hspace{5cm}\cdot [1+2^{k_2-k_1}+2^{-p_1-k_1}\min\{2^{p+k},2^{p_2+k_2}\}+2^{p+p_2-p_1}\}]\\
  &\qquad\lesssim
  \begin{cases}
  2^{k_{\max}}[1+2^{p-p_1}][1+2^{k-k_1}2^{p-p_1}] ,&k=k_{\min},\\
  2^{3k_{\max}-2k_2+p_1}+2^{2k_{\max}-k_2}[1+2^{p_2-p_1}]+2^{k_{\max}}2^{p_2-p_1}[1+2^{k_2-k_{\max}}2^{p_2-p_1}] ,&k_2=k_{\min},
  \end{cases}
 \end{aligned} 
 \end{equation}

 \item Similarly, we estimate by \eqref{eq:bds_1ibp} and \eqref{eq:mixvf_quot}
 \begin{equation}
 \begin{aligned}
  \abs{\frac{\m\Ups_\xi\Phi}{V_\eta\Phi}\frac{V'_{\xi-\eta}V_\eta^2\Phi}{V_\eta\Phi}}&\lesssim 2^{2k_1-k_2}(2^{k-k_1}+2^{p-p_1})\cdot 
  \begin{cases}
  1+2^{k-k_1}2^{p-p_1},&k=k_{\min},\\
  1+2^{k_2-k_1}2^{p_2-p_1}+2^{p}2^{k-k_2},&k_2=k_{\min}.
 \end{cases}\\
 &\lesssim
 \begin{cases}
  2^{k_{\max}}[2^{p-p_1}+2^{k-k_1}][1+2^{k-k_1}2^{p-p_1}],&k=k_{\min},\\
  [1+2^{p-p_1}][2^{3k_{\max}-2k_2+p}+2^{2k_{\max}-k_2}+2^{k_{\max}}2^{p_2-p_1}],&k_2=k_{\min}.
 \end{cases}
 \end{aligned} 
 \end{equation}
 
 \item With the estimates of Section \ref{sec:morevf} we have that
 \begin{equation}\label{eq:V'onmult}
  \abs{V'_{\xi-\eta}\m}\lesssim (1+2^{p_1+k_1-p_2-k_2})C_\m,
 \end{equation}
 so that with \eqref{eq:vfonUpsPhi} and \eqref{eq:bds_1ibp} it follows that
 \begin{equation}\label{eq:1mixibp}
 \begin{aligned}
  \abs{\frac{V'_{\xi-\eta}(\m\Ups_\xi\Phi)}{V_\eta\Phi}}&\lesssim \abs{\frac{(V'_{\xi-\eta}\m)\Ups_\xi\Phi}{V_\eta\Phi}}+\abs{\frac{\m V'_{\xi-\eta}(\Ups_\xi\Phi)}{V_\eta\Phi}}\\
  &\lesssim (1+2^{p_1+k_1-p_2-k_2})\cdot\abs{\frac{\m\Ups_\xi\Phi}{V_\eta\Phi}} + \abs{\frac{\m}{V_\eta\Phi}}2^{k-k_1+p_1}\\ 
  &\lesssim (1+2^{p_1+k_1-p_2-k_2})2^{2k_1-k_2}[2^{p-p_1}+2^{k-k_1}]+ 2^{k+p_{\max}+q_{\max}}2^{2k_1-p_1}\abs{\bar\sigma}^{-1}\cdot 2^{k-k_1+p_1}\\ 
  &\lesssim
  \begin{cases}
   (1+2^{p_1-p_2})2^{k_{\max}}[2^{p-p_1}+2^{k-k_1}],&k=k_{\min},\\
   (1+2^{k_1-k_2}2^{p_1-p_2})2^{2k_1-k_2}[1+2^{p-p_1}],&k_2=k_{\min},
  \end{cases}\\
  &\lesssim
  \begin{cases}
   2^{k_{\max}}(2^{p-p_1}+2^{p-p_2})+2^{k}(1+2^{p_1-p_2}),&k=k_{\min},\\
   2^{k_1}2^{p_2-p_1}+2^{2k_1-k_2}[1+2^{k_1-k_2}2^{p_1-p_2}],&k_2=k_{\min}.
  \end{cases}
 \end{aligned} 
 \end{equation}
 Together with \eqref{eq:vfquotient_apdx} we get that
 \begin{equation}
 \begin{aligned}
 &\abs{\frac{V'_{\xi-\eta}(\m\Ups_\xi\Phi)}{V_\eta\Phi}\frac{V_\eta^2\Phi}{V_\eta\Phi}}\lesssim \abs{\frac{V'_{\xi-\eta}(\m\Ups_\xi\Phi)}{V_\eta\Phi}}\cdot 
  \begin{cases}
   1+2^{k-k_1}2^{p-p_1},&k=k_{\min},\\
   1+2^{k_2-k_1}2^{p_2-p_1},&k_2=k_{\min},
  \end{cases}\\
 &\qquad \lesssim
  \begin{cases}
   2^{k_{\max}}(2^{p-p_1}+2^{p-p_2})+2^{k}+2^{k}[2^{p-p_1}+2^{k-k_1}](2^{p-p_1}+2^{p-p_2}),&k=k_{\min},\\
    2^{3k_1-2k_2}2^{p_1-p_2}+2^{2k_1-k_2}+(2^{k_1}+2^{k_2}2^{p_2-p_1})2^{p_2-p_1},&k_2=k_{\min}.
  \end{cases}
 \end{aligned}
 \end{equation}

  \item Next, using \eqref{eq:vfmupsquot}, \eqref{eq:mvfupsquot} and \eqref{eq:mixvf_quot}, we find that
  \begin{equation}
  \begin{split}
   \abs{\frac{V_\eta(\m\Ups_\xi\Phi)}{V_\eta\Phi}\frac{V'_{\xi-\eta}V_\eta\Phi}{V_\eta\Phi}}&\lesssim \abs{\frac{V'_{\xi-\eta}V_\eta\Phi}{V_\eta\Phi}}\left\{\abs{\frac{(V_\eta\m)\Ups_\xi\Phi}{V_\eta\Phi}}+\abs{\frac{\m V_\eta(\Ups_\xi\Phi)}{V_\eta\Phi}}\right\}\\
   &\lesssim 2^{k_{\max}}(1+2^{p-p_1})(1+2^{k_{\max}-k_2+p_1})(1+2^{k_1-k_2}+2^{p_2-p_1}).
   \end{split}
  \end{equation}

  \item Finally, to bound $\abs{\frac{V_{\xi-\eta}'V_\eta(\m\Ups_\xi\Phi)}{V_\eta\Phi}}$ we proceed as follows. First, note that with the bounds of Section \ref{sec:morevf} we have that
  $$\abs{V'_{\xi-\eta}V_\eta\m}\lesssim (1+2^{p_1+k_1-p_2-k_2})(1+2^{p_2+k_2-p_1-k_1})C_{\m},$$ and hence, using \eqref{eq:UpsPhi_bd},
  \begin{equation}
  \begin{aligned}
   \abs{\frac{(V_{\xi-\eta}'V_\eta\m)\Ups_\xi\Phi}{V_\eta\Phi}}&\lesssim (1+2^{p_1+k_1-p_2-k_2}+2^{p_2+k_2-p_1-k_1})C_{\m}\cdot (2^p+2^{k-k_1}2^{p_1})\cdot 2^{2k_1-p_1}\abs{\bar\sigma}^{-1}\\
   &\lesssim
  \begin{cases}
   (1+2^{p_1-p_2}+2^{p_2-p_1})(2^{p-p_1}+2^{k-k_1})\cdot 2^{k_1},&k=k_{\min},\\
   (1+2^{k_1-k_2}2^{p_1-p_2}+2^{k_2-k_1}2^{p_2-p_1})(2^{p-p_1}+1)\cdot 2^{2k_1-k_2},&k_2=k_{\min},
  \end{cases}\\
  &\lesssim
  \begin{cases}
   2^{k_{\max}}2^{p-p_1}+2^{k}(1+2^{p-p_2}+2^{p-p_1})(1+2^{p-p_1}),&k=k_{\min},\\
   2^{3k_1-2k_2}2^{p_1-p_2}+2^{2k_1-k_2}+2^{k_1}2^{p_2-p_1}+2^{k_2}2^{2p_2-2p_1},&k_2=k_{\min},
  \end{cases}
  \end{aligned}
  \end{equation}
  having used \eqref{eq:freq_triang}. Secondly, note that with \eqref{eq:mixvf_quot} we can simplify estimate \eqref{eq:2vfonUpsPhi} and write
  \begin{equation}
  \begin{aligned}
   \abs{V'_{\xi-\eta}V_\eta\Ups_\xi\Phi}&\lesssim \abs{V'_{\xi-\eta}V_\eta\Ups_\xi\Phi}\lesssim 2^{k-k_1}2^{-p_1}\abs{V_\eta\Phi}[1+\frac{\vert V'_{\xi-\eta}V_\eta\Phi\vert}{\vert V_\eta\Phi\vert}]+(1+2^{p_1+k_1-p_2-k_2})\cdot(2^{p_1}+2^{p_2})2^{k_2-k_1}\\
   &\lesssim (1+2^{p_1+k_1-p_2-k_2})\cdot2^{k-k_1}[2^{-p_1}\abs{V_\eta\Phi}+(2^{p_1}+2^{p_2})2^{k_2-k_1}],
  \end{aligned} 
  \end{equation}
 so that with \eqref{eq:freq_triang} there holds that
 \begin{equation}
 \begin{aligned}
  \abs{\frac{\m V_{\xi-\eta}'V_\eta\Ups_\xi\Phi}{V_\eta\Phi}}&\lesssim C_\m(1+2^{p_1+k_1-p_2-k_2})2^{k-k_1-p_1}+C_\m(2^{k_2-k_1}+2^{p_1-p_2})2^{k+k_1}(1+2^{p_2-p_1})\abs{\bar\sigma}^{-1}\\
  &\lesssim
  \begin{cases}
   2^k(1+2^{p_1-p_2})2^{p_{\max}-p_1}+2^k(1+2^{p_1-p_2}+2^{p_2-p_1}),&k=k_{\min},\\
   2^k2^{p_{\max}-p_1}(1+2^{k_1-k_2}2^{p_1-p_2})+2^{2k-k_2}(2^{k_2-k_1}+2^{p_1-p_2})(1+2^{p_2-p_1}),&k_2=k_{\min},
  \end{cases}\\
  &\lesssim
  \begin{cases}
   2^k(1+2^{p-p_2}+2^{p-p_1}),&k=k_{\min},\\
   2^{2k_1-k_2}(1+2^{p_1-p_2}+2^{p-p_2})+2^{k_1}(1+2^{p_2-p_1}+2^{p_1-p_2}),&k_2=k_{\min}.
  \end{cases}
 \end{aligned}
 \end{equation}
The terms $\abs{\frac{V_{\xi-\eta}'\m V_\eta\Ups_\xi\Phi}{V_\eta\Phi}}$ and $\abs{\frac{V_\eta\m V_{\xi-\eta}'\Ups_\xi\Phi}{V_\eta\Phi}}$ are analogous.
\end{itemize}

Finally, we have that 
\begin{equation}
 \abs{V'_{\xi-\eta}\left(\frac{\m\Ups_\xi\Phi}{V_\eta\Phi}\right)}\lesssim \abs{\frac{V'_{\xi-\eta}(\m\Ups_\xi\Phi)}{V_\eta\Phi}}+\abs{\frac{\m\Ups_\xi\Phi}{V_\eta\Phi}\frac{V'_{\xi-\eta}V_\eta\Phi}{V_\eta\Phi}}.
\end{equation}
The first term has already been estimated in \eqref{eq:1mixibp}, while for the second by \eqref{eq:bds_1ibp} and \eqref{eq:mixvf_quot} there holds
\begin{equation}
\begin{aligned}
 \abs{\frac{\m\Ups_\xi\Phi}{V_\eta\Phi}\frac{V'_{\xi-\eta}V_\eta\Phi}{V_\eta\Phi}}&\lesssim (1+2^{k_{\max}-k_2+p_1})2^{2k_1-k_2}(2^{p-p_1}+2^{k-k_1})\\
 &\lesssim
  \begin{cases}
   2^{k_{\max}}(2^{p-p_1}+2^{k-k_1}),&k=k_{\min},\\
   2^{2k_1-k_2}(1+2^{p-p_1})+2^{3k_{\max}-2k_2}(2^p+2^{p_1}),&k_2=k_{\min}.
  \end{cases}
\end{aligned}
\end{equation}
One thus obtains the claim \eqref{eq:bds_1ibp'}.

This concludes the proof of the lemma.
\end{proof}

As can be seen in the proof of Lemma \ref{lem:2ibp}, we will also have to control the following ``cross terms'' of vector fields.
\begin{lemma}\label{lem:crossvf}
 For axisymmetric $f$ we have that
\begin{equation}\label{eq:crossvf}
\begin{aligned}
 \abs{V'_{\xi-\eta}V_\eta f(\xi-\eta)}&\lesssim (1+2^{k_2-k_1})\abs{Sf(\xi-\eta)}+(1+2^{k_2-k_1})(2^{p_1}+2^{p_2})2^{q_{\max}}\abs{\Ups f(\xi-\eta)}\\
 &+2^{k_2-k_1}\abs{V'Sf(\xi-\eta)}+2^{k_2-k_1}[2^{p_2+q_1}+2^{p_1+q_2}]\abs{V'\Ups f(\xi-\eta)}.
\end{aligned} 
\end{equation}
\end{lemma}
\begin{proof}
 We recall from Section \ref{ssec:crossterms} that
\begin{equation}
\begin{aligned}
 V'_{\xi-\eta}S_\eta f(\xi-\eta)&=V'_{\xi-\eta}\left[\frac{\abs{\eta}}{\abs{\xi-\eta}}\left[\omega_c\sqrt{1-\Lambda^2}(\eta)\sqrt{1-\Lambda^2}(\xi-\eta)-\Lambda(\xi-\eta)\Lambda(\eta)\right]S_{\xi-\eta}\right]\\
 &\qquad +V'_{\xi-\eta}\left[\frac{\abs{\eta}}{\abs{\xi-\eta}}\left[\omega_c\sqrt{1-\Lambda^2}(\eta)\Lambda(\xi-\eta)+\Lambda(\eta)\sqrt{1-\Lambda^2}(\xi-\eta)\right]\Upsilon_{\xi-\eta}\right]
\end{aligned} 
\end{equation}
and
\begin{equation}
\begin{aligned}
 V'_{\xi-\eta}\Omega_\eta f(\xi-\eta)&=-V'_{\xi-\eta}\left[\frac{\abs{\eta}}{\abs{\xi-\eta}}\omega_s\cdot\sqrt{1-\Lambda^2}(\eta)\left[\Lambda(\xi-\eta)\Upsilon_{\xi-\eta}+\sqrt{1-\Lambda^2}(\xi-\eta)S_{\xi-\eta}\right]\right].
\end{aligned} 
\end{equation}
To estimate this we compute as in Section \ref{sec:morevf} (with the symmetric versions of \eqref{eq:Vonangle2}, \eqref{eq:Vonfrac}) that
\begin{equation}
 \abs{V'_{\xi-\eta}\left(\frac{\abs{\eta}}{\abs{\xi-\eta}}\right)}\lesssim 1+\frac{\abs{\eta}}{\abs{\xi-\eta}},\qquad \abs{V'_{\xi-\eta}\omega_c}+\abs{V'_{\xi-\eta}\omega_s}\lesssim 1+\frac{\abs{\xi_\h-\eta_\h}}{\abs{\eta_\h}},
\end{equation}
and by direct computation from \eqref{eq:Lamgrad}
\begin{equation}
\begin{aligned}
 \abs{V'_{\xi-\eta}\Lambda(\eta)}&\lesssim 2^{k_1-k_2}2^{p_2}
 \begin{cases}
 2^{p_1+q_2},&V'=\Omega,\\
 2^{p_1+q_2}+2^{q_1+p_2},&V'=S,                                                     
  \end{cases}\\
 \abs{V'_{\xi-\eta}\sqrt{1-\Lambda^2}(\eta)}&\lesssim 2^{k_1-k_2}2^{q_2}
 \begin{cases}
 2^{p_1+q_2},&V'=\Omega,\\
 2^{p_1+q_2}+2^{q_1+p_2},&V'=S.                                                     
  \end{cases}
\end{aligned}  
\end{equation}
Collecting terms gives that
\begin{equation}
\begin{aligned}
 \abs{V'_{\xi-\eta}\left(\frac{\abs{\eta}}{\abs{\xi-\eta}}\left[\omega_c\sqrt{1-\Lambda^2}(\eta)\sqrt{1-\Lambda^2}(\xi-\eta)-\Lambda(\xi-\eta)\Lambda(\eta)\right]\right)}&\lesssim (1+2^{k_2-k_1})[2^{p_1+p_2}+2^{q_1+q_2}]+2^{2p_1}\\
 &\qquad+(2^{p_1+q_2}+2^{q_1+p_2})^2
\end{aligned} 
\end{equation}
and
\begin{equation}
\begin{aligned}
 \abs{V'_{\xi-\eta}\left(\frac{\abs{\eta}}{\abs{\xi-\eta}}\left[\omega_c\sqrt{1-\Lambda^2}(\eta)\Lambda(\xi-\eta)+\Lambda(\eta)\sqrt{1-\Lambda^2}(\xi-\eta)\right]\right)}&\lesssim (1+2^{k_2-k_1})[2^{p_2+q_1}+2^{p_1+q_2}]+2^{p_1+q_1}\\
 &\qquad+(2^{p_1+q_2}+2^{q_1+p_2})(2^{p_1+p_2}+2^{q_1+q_2}),
\end{aligned} 
\end{equation}
so that
\begin{equation}
\begin{aligned}
 \abs{V'_{\xi-\eta}S_\eta f(\xi-\eta)}&\lesssim (1+2^{k_2-k_1})\abs{Sf(\xi-\eta)}+(1+2^{k_2-k_1})(2^{p_1}+2^{p_2})2^{q_{\max}}\abs{\Ups f(\xi-\eta)}\\
 &+2^{k_2-k_1}\abs{V'Sf(\xi-\eta)}+2^{k_2-k_1}[2^{p_2+q_1}+2^{p_1+q_2}]\abs{V'\Ups f(\xi-\eta)}.
\end{aligned} 
\end{equation}
Analogously we compute that
\begin{equation}
\begin{aligned}
 \abs{V'_{\xi-\eta}\left[\frac{\abs{\eta}}{\abs{\xi-\eta}}\omega_s\cdot\sqrt{1-\Lambda^2}(\eta)\sqrt{1-\Lambda^2}(\xi-\eta)\right]}\lesssim (1+2^{k_2-k_1})2^{p_1+p_2}+2^{2p_2}+2^{p_1+q_2}(2^{p_1+q_2}+2^{p_2+q_1})
\end{aligned}
\end{equation}
and
\begin{equation}
\begin{aligned}
 \abs{V'_{\xi-\eta}\left[\frac{\abs{\eta}}{\abs{\xi-\eta}}\omega_s\cdot\sqrt{1-\Lambda^2}(\eta)\Lambda(\xi-\eta)\right]}\lesssim (1+2^{k_2-k_1})2^{p_2+q_1}+2^{p_1+q_1}+2^{q_1+q_2}(2^{p_1+q_2}+2^{p_2+q_1}),
\end{aligned}
\end{equation}
which gives the claim as above.
\end{proof}

\subsection{Symbol bounds}\label{sec:symbols}
In this section we give the relevant symbol estimates for the multipliers we need. To simplify notations we introduce the following shorthand for the localizations (see also \eqref{eq:loc_def})
\begin{equation}
\begin{aligned}
 \varphi_{k,p,q}(\xi)&:=\varphi(2^{-k}\vert\xi\vert)\varphi(2^{-2(p+k)}(\xi_1^2+\xi_2^2))\varphi(2^{-q-k}\xi_3),\\
 \widetilde{\varphi}(\xi,\eta)&:=\varphi_{k,p,q}(\xi)\cdot\varphi_{k_1,p_1,q_1}(\xi-\eta)\cdot\varphi_{k_2,p_2,q_2}(\eta),
\end{aligned}
\end{equation}
and for a multiplier $m\in L^1_{loc}(\R^3\times\R^3)$ we let
\begin{equation}
 \norm{m}_{\W}:=\sup_{k,p,q,\,k_i,p_i,q_i,i=1,2}\norm{\mathcal{F}(\widetilde{\varphi}m)}_{L^1(\R^3\times\R^3)}.
\end{equation}
In particular, if $f=P_{k_1,p_1,q_1}f$ and $g=P_{k_2,p_2,q_2}$, we then have H\"older's inequality (see also \eqref{ProdRule})
\begin{equation}\label{ProdRule2}
 \norm{P_{k,p,q}Q_m[f,g]}_{L^r}\lesssim \norm{m}_{\W}\norm{f}_{L^p}\norm{g}_{L^q},\qquad \frac{1}{r}=\frac{1}{p}+\frac{1}{q}.
\end{equation}
Moreover, we note that 
\begin{equation}\label{eq:alg_prop}
 \norm{m_1\cdot m_2}_{\W}\lesssim \norm{m_1}_{\W}\norm{m_2}_{\W}.
\end{equation}
Furthermore, we recall from \eqref{eq:vfloc_not} the notation
\begin{equation}
 \bar\chi_{V_\eta}=\bar\chi(L^{-1}2^{2k_1-p_1}V_\eta\Phi),\qquad \bar\chi_{V_{\xi-\eta}}=\bar\chi(L^{-1}2^{2k_2-p_2}V_{\xi-\eta}\Phi),\qquad V\in\{S,\Omega\}.
\end{equation}
We have the following symbol bounds:
\begin{lemma}\label{lem:vf_symb}
 Let $L:=2^{k_{\max}+k_{\min}+q_{\max}+p_{\max}}$. Then we have that for $V\in\{S,\Omega\}$ there holds
 \begin{equation}\label{eq:S_symb}
  \norm{\frac{1}{V_\eta\Phi}\cdot \bar\chi_{V_\eta}}_{\W}\lesssim L^{-1}\cdot 2^{2k_1-p_1}.
 \end{equation}

\end{lemma}

\begin{remark}\label{rem:vf_symb}
 The analogous estimates also hold for $V'_{\xi-\eta}$, i.e.\ $\norm{\frac{1}{V'_{\xi-\eta}\Phi}\cdot \bar\chi_{V'_{\xi-\eta}}}_{\W}\lesssim L^{-1}\cdot 2^{2k_2-p_2}$ for $V'\in\{S,\Omega\}$.
\end{remark}

\begin{proof}[Proof of Lemma \ref{lem:vf_symb}]
 It suffices to establish \eqref{eq:S_symb} in the case $V=S$, the case $V=\Omega$ is analogous.
 
 We thus want to compute
 \begin{equation}
 M(x_\h,x_3,y_\h,y_3):= \mathcal{F}\left(\frac{\widetilde{\varphi}}{\abs{\xi-\eta}^{-3}\bar\sigma\cdot(\xi_\h-\eta_\h)}\bar\chi(L^{-1}2^{2k_1-p_1}V_\eta\Phi)\right)(x_\h,x_3,y_\h,y_3).
 \end{equation}
For this we change variables. Firstly, we let
\begin{equation}
 (A,a):=(2^{-p_1-k_1}(\xi_\h-\eta_\h),2^{-q_1-k_1}(\xi_3-\eta_3)),\qquad A\in\R^2,a\in\R.
\end{equation}
We further choose $(B,b)$ such that $B\in\R^2$ is the variable corresponding to $\min\{2^{p+k},2^{p_2+k_2}\}$ rescaled by the inverse of this factor, and analogously for $b\in\R$ with $\min\{2^{q+k},2^{q_2+k_2}\}$. This leaves four options which are all dealt with analogously, so without loss of generality we make the following
\begin{equation}\label{eq:assump_min}
 \textnormal{Assumption: }\quad \min\{2^{p+k},2^{p_2+k_2}\}=2^{p_2+k_2},\quad \min\{2^{q+k},2^{q_2+k_2}\}=2^{q+k}.
\end{equation}
We then let
 \begin{equation}
  (B,b):=(2^{-p_2-k_2}\eta_\h,2^{-q-k}\xi_3).
 \end{equation}
Since we have that
\begin{equation}
 \xi_\h=2^{p_1+k_1}A+2^{p_2+k_2}B,\quad \eta_3=2^{q+k}b-2^{q_1+k_1}a,
\end{equation}
changing variables $(\xi_\h,\xi_3,\eta_\h,\eta_3)\mapsto(A,a,B,b)$ gives that 
\begin{equation}
\begin{aligned}
 M(x_\h,x_3,y_\h,y_3)&=2^{2p_1+2k_1}2^{q_1+k_1}\cdot 2^{2p_2+2k_2}2^{q+k}\cdot \mathcal{F}(I)(2^{p_1+k_1}x_\h,2^{q_1+k_1}x_3,2^{p_2+k_2}y_\h,2^{q+k}y_3),
\end{aligned} 
\end{equation}
where
\begin{equation}
\begin{aligned}
 I(A,a,B,b)&:=\frac{2^{3k_1}\abs{(A,a)}^3}{\bar\sigma(A,a,B,b)\cdot  2^{k_1+p_1}A}\cdot\bar\chi(L^{-1}\abs{(A,a)}^{-3}\abs{\bar\sigma\cdot A})\cdot\varphi(2^{-2p-2k}\abs{2^{p_1+k_1}A+2^{p_2+k_2}B}^2)\varphi(b)\\
 &\qquad\qquad\cdot\varphi(\abs{A}^2)\varphi(a)\cdot\varphi(\abs{B}^2)\varphi(2^{-q_2-k_2}(2^{q+k}b-2^{q_1+k_1}a)).
\end{aligned} 
\end{equation}
Now notice that since $\bar\sigma=(\xi_3-\eta_3)\eta_\h-\eta_3(\xi_\h-\eta_\h)$ we have
\begin{equation}
\begin{aligned}
 \bar\sigma(A,a,B,b)&=2^{q_1+k_1}a2^{p_2+k_2}B-(2^{q+k}b-2^{q_1+k_1}a)2^{p_1+k_1}A\\
 &=2^{k_1+k_2+p_{\max}+q_{\max}}\left[2^{q_1-q_{\max}}a2^{p_2-p_{\max}}B-(2^{q+k-q_{\max}-k_2}b-2^{q_1+k_1-q_{\max}-k_2}a)2^{p_1-p_{\max}}A\right]\\
 &=:L\cdot\Sigma.
\end{aligned} 
\end{equation}
Here we highlight that by assumption \eqref{eq:assump_min} we have $2^{q+k-q_{\max}-k_2}\leq 1$ and $2^{q_1+k_1-q_{\max}-k_2}\leq 1$. It thus follows that $\Sigma$ and its derivatives are uniformly bounded,
\begin{equation}\label{eq:tildesigma_prop}
 \abs{\Sigma}\sim 1,\qquad \abs{\nabla_{A,a,B,b}^N\Sigma}\leq 1, \quad N\in\N,
\end{equation}
and satisfies moreover $\abs{\Sigma\cdot A}\gtrsim 1$. Thus we have
\begin{equation}
\begin{aligned}
 M(x_\h,x_3,y_\h,y_3)&=2^{-k_1-k_2-p_{\max}-q_{\max}}\cdot  2^{2k_1-p_1} \cdot  2^{2p_1+2k_1}2^{q_1+k_1}2^{2p_2+2k_2}2^{q+k}\\
 &\qquad \cdot \mathcal{F}(\widetilde{I})(2^{p_1+k_1}x_\h,2^{q_1+k_1}x_3,2^{p_2+k_2}y_\h,2^{q+k}y_3),
\end{aligned} 
\end{equation}
where now
\begin{equation}
\begin{aligned}
 \widetilde{I}(A,a,B,b)&:=\frac{\abs{(A,a)}^{3}}{\Sigma\cdot A}\cdot \bar\chi(\abs{(A,a)}^{-3}\abs{\Sigma\cdot A})\cdot\varphi(2^{-2p-2k}\abs{2^{p_1+k_1}A+2^{p_2+k_2}B}^2)\varphi(b)\cdot\varphi(\abs{A}^2)\varphi(a)\\
 &\qquad\qquad\cdot\varphi(\abs{B}^2)\varphi(2^{-q_2-k_2}(2^{q+k}b-2^{q_1+k_1}a)).
\end{aligned} 
\end{equation}
By assumption \eqref{eq:assump_min} and the triangle inequality we have that $2^{p+k}\gtrsim 2^{p_1+k_1}+2^{p_2+k_2}$ as well as $2^{q_2+k_2}\gtrsim 2^{q+k}+2^{q_1+k_1}$, so that with \eqref{eq:tildesigma_prop} we have that $\widetilde{I}\in C^\infty_c$ with bound uniformly in $k,p,q$, $k_i,p_i,q_i$, $i=1,2$,
\begin{equation}
 \abs{\nabla_{A,a,B,b}^N\widetilde{I}}\leq C, \quad N\in\N. 
\end{equation}
We then have that
\begin{equation}
\begin{aligned}
 &\int_{(x_\h,x_3,y_\h,y_3)}\abs{M(x_\h,x_3,y_\h,y_3)} d(x_\h,x_3,y_\h,y_3)\\
 &\qquad=2^{-k_1-k_2-p_{\max}-q_{\max}} \cdot 2^{2k_1-p_1} \int_{(x_\h,x_3,y_\h,y_3)} \abs{\mathcal{F}(\widetilde{I})(2^{p_1+k_1}x_\h,2^{q_1+k_1}x_3,2^{p_2+k_2}y_\h,2^{q+k}y_3)}\\
 &\hspace{5cm}\cdot 2^{2p_1+2k_1}2^{q_1+k_1}\cdot 2^{2p_2+2k_2}2^{q+k} d(x_\h,x_3,y_\h,y_3)\\
 &\qquad =2^{-k_1-k_2-p_{\max}-q_{\max}}\cdot 2^{2k_1-p_1}\int_{(x_\h,x_3,y_\h,y_3)}\abs{\mathcal{F}(\widetilde{I})(x_\h,x_3,y_\h,y_3)}d(x_\h,x_3,y_\h,y_3)\\
 &\qquad \leq 2^{-k_1-k_2-p_{\max}-q_{\max}}\cdot 2^{2k_1-p_1}\cdot C,
\end{aligned} 
\end{equation}
by the aforementioned properties of $\widetilde{I}$.

\end{proof}

In a similar spirit one shows the following symbol bound:
\begin{lemma}\label{lem:phasesymb_bd}
 \begin{equation}\label{eq:phasesymb_bd}
 \begin{aligned}
  \norm{\frac{1}{\Phi}\bar\chi(2^{-q_{\max}}\Phi)}_{\W}\lesssim 2^{-q_{\max}}.
 \end{aligned}
 \end{equation}
\end{lemma}
 
\begin{proof}
Writing 
\begin{equation}
 \frac{1}{\Phi}\bar\chi(2^{-q_{\max}}\Phi)=2^{-q_{\max}}\tilde{\bar\chi}(2^{-q_{\max}}\Phi),\qquad \tilde{\bar\chi}(x)=\frac{1}{x}\bar\chi(x),
\end{equation}
where $\tilde{\bar\chi}$ has similar support properties as $\bar\chi$, we see that it suffices to establish the bound
\begin{equation}
 \norm{\tilde{\bar\chi}(2^{-q_{\max}}\Phi)}_{\W}\lesssim 1.
\end{equation}
To this end, since we consider all the possible signs in the phase we may assume that 
\begin{equation}\label{eq:assump_sizes}
 \textnormal{Assumption: }\quad 2^{k+q}\leq\min\{2^{k_1+q_1},2^{k_2+q_2}\},\quad 2^{p_2+k_2}\leq\min\{2^{p_1+k_1},2^{p+k}\},
\end{equation}
the alternative that $2^{p+k}\leq\min\{2^{p_1+k_1},2^{p_2+k_2}\}$ being treated analogously.

We observe that with the notation $\vert\xi^\alpha\vert=\vert\xi_h\vert^{\vert\alpha_1\vert+\vert\alpha_2\vert}\vert\xi_3\vert^{\alpha_3}$ for $\alpha\in\N_0^3$ there holds that
\begin{equation}
 \abs{\xi^\alpha}\abs{\partial^\alpha_{\xi}\Lambda(\xi)}\lesssim \vert\Lambda(\xi)\vert\lesssim 2^{q_{\max}}.
\end{equation} 
By assumption \eqref{eq:assump_sizes} it then follows that
\begin{equation}\label{eq:phi_deriv}
 \abs{\xi^\alpha}\abs{\eta^\beta} \cdot 2^{-q_{\max}}\abs{\partial_\xi^\alpha\partial_\eta^\beta\Phi}\lesssim 1.
\end{equation}
Letting
\begin{equation*}
\begin{split}
M(x_h,x_3,y_h,y_3)&:=\iint_{\R^3\times\R^3}\tilde{\bar\chi}(2^{-q_{max}}\Phi(\xi,\eta))e^{-i\xi\cdot x}e^{-i\eta\cdot y}\widetilde{\varphi}(\xi,\eta)d\xi d\eta
\end{split}
\end{equation*}
we see that
\begin{equation*}
\begin{split}
&\left(1+2^{2(k+p)}\vert x_h\vert^2\right)^2\left(1+2^{2k+2q}\vert x_3\vert^2\right)\cdot \left(1+2^{2(k_2+p_2)}\vert y_h\vert^2\right)^2\left(1+2^{2k_2+2q_2}\vert y_3\vert^2\right)M(x_h,x_3,y_h,y_3)\\
&\qquad=\iint_{\R^3\times\R^3}e^{-i\xi\cdot x}e^{-i\eta\cdot y}{\bf M}(\xi,\eta)d\xi d\eta,
\end{split}
\end{equation*}
where
\begin{equation*}
\begin{split}
{\bf M}(\xi,\eta)&:=\left(1-2^{2(k+p)}\Delta_{\xi_h}\right)^2\left(1-2^{2k+2q}\partial_{\xi_3}^2\right)\\
&\qquad\cdot \left(1-2^{2(k_2+p_2)}\Delta_{\eta_h}\right)^2\left(1-2^{2k_2+2q_2}\partial_{\eta_3}^2\right)\left[\tilde{\bar\chi}(2^{-q_{max}}\Phi(\xi,\eta))\widetilde{\varphi}(\xi,\eta)\right].
\end{split}
\end{equation*}
By virtue of \eqref{eq:phi_deriv} we obtain that
\begin{equation*}
\begin{split}
\vert {\bf M}(\xi,\eta)\vert&\lesssim \widetilde{\overline{\varphi}}(\xi,\eta),
\end{split}
\end{equation*}
where $\widetilde{\overline{\varphi}}(\xi,\eta)$ has similar support properties as $\widetilde{\varphi}(\xi,\eta)$. A crude integration allows us to conclude that
\begin{equation}
 \abs{\iint_{\R^3\times\R^3}e^{-i\xi\cdot x}e^{-i\eta\cdot y}{\bf M}(\xi,\eta)d\xi d\eta}\lesssim 2^{2p+2k}2^{k+q}2^{2p_2+2k_2}2^{k_2+q_2},
\end{equation}
so that
\begin{equation}
\begin{aligned}
 &\norm{\mathcal{F}(\tilde{\bar\chi}(2^{-q_{\max}}\Phi))}_{L^1(\R^3\times\R^3)}=\norm{M}_{L^1(\R^3\times\R^3)}\\
 &\qquad\lesssim 2^{2p+2k}2^{k+q}\norm{(1+2^{2(k+p)}\vert x_h\vert^2)^{-2}(1+2^{2k+2q}\vert x_3\vert^2)^{-1}}_{L^1(\R^3)}\\
 &\qquad\qquad \cdot 2^{2p_2+2k_2}2^{k_2+q_2}\norm{(1+2^{2(k_2+p_2)}\vert y_h\vert^2)^{-2}(1+2^{2k_2+2q_2}\vert y_3\vert^2)^{-1}}_{L^1(\R^3)}\lesssim 1.
\end{aligned} 
\end{equation}

\end{proof}

\section{Some Useful Lemmata}

\subsection{More set size gain}\label{apdx:set_gain} 
The idea here is that in the bilinear estimates we can always gain the smallest of \emph{both $p,p_1$ and $q,q_1$}, since they correspond to different directions.
\begin{lemma}\label{lem:set_gain}
Consider a typical bilinear expression $Q_\mathfrak{m}$ with localizations and a multiplier $\mathfrak{m}$, i.e.\
\begin{equation}
\begin{aligned}
 &\widehat{Q_\mathfrak{m}(f,g)}(\xi):=\int_\eta \chi(\xi,\eta) \mathfrak{m}(\xi,\eta)\hat{f}(\xi-\eta)\hat{g}(\eta) d\eta,\\
 &\chi(\xi,\eta)=\varphi_{k,p,q}(\xi)\varphi_{k_1,p_1,q_1}(\xi-\eta)\varphi_{k_2,p_2,q_2}(\eta).
\end{aligned}
\end{equation}
Then with
\begin{equation}
 \Sz:=\min\{2^{p+k},2^{p_1+k_1},2^{p_2+k_2}\}\cdot\min\{2^{\frac{q+k}{2}},2^{\frac{q_1+k_1}{2}},2^{\frac{q_2+k_2}{2}}\}
\end{equation} 
we have that
\begin{equation}
 \norm{Q_\mathfrak{m}(f,g)}_{L^2}\lesssim \Sz\cdot\norm{\mathfrak{m}}_{L^\infty_{\xi,\eta}} \norm{P_{k_1,p_1,q_1}f}_{L^2}\norm{P_{k_2,p_2,q_2}g}_{L^2}.
\end{equation}

\end{lemma}
\begin{proof}
 To begin, let us assume that $p+k<p_1+k_1$ and $q+k>q_1+k_1$ (the ``symmetric cases'' of $p+k<p_1+k_1$ with $q+k<q_1+k_1$ and reverse are direct).
 Then, for any $h\in L^2$ we find that
 \begin{equation*}
 \begin{split}
  \abs{\ip{Q_\m(f,g),h}}&\lesssim \iint_{\mathbb{R}^3}\abs{\mathfrak{m}(\xi,\eta)}\varphi(2^{-k-p}\xi_\h)\vert\widehat{f}(\xi-\eta)\vert\varphi(2^{-k_1-q_1}(\xi_3-\eta_3))\vert\widehat{g}(\xi-\eta)h(\xi)\vert d\xi d\eta\\
  &\lesssim \Vert \mathfrak{m}\Vert_{L^\infty}\Vert \widehat{h}(\xi)\widehat{f}(\xi-\eta)\Vert_{L^2_{\xi,\eta}}\Vert \varphi(2^{-k-p}\xi_\h)\varphi(2^{-k_1-q_1}(\xi_3-\eta_3))\widehat{g}(\eta)\Vert_{L^2_{\xi,\eta}}\\
  &\lesssim 2^{k+\frac{k_1}{2}}2^{p+\frac{q_1}{2}}\Vert \mathfrak{m}\Vert_{L^\infty}\Vert f\Vert_{L^2}\Vert g\Vert_{L^2}\Vert h\Vert_{L^2}.
 \end{split}
 \end{equation*}
 The claim then follows upon changing variables $\eta\leftrightarrow\xi-\eta$.
\end{proof}

\subsection{Control of Fourier transform in $L^\infty$}
We record here that our decay norm $D$ also controls the Fourier transform in $L^\infty$:
\begin{lemma}\label{lem:ControlLinfty}
Assume that $f$ is axisymmetric. Then there holds that
\begin{equation*}
\begin{split}
\Vert \widehat{P_{k;p,q}f}\Vert_{L^\infty}&\lesssim 2^{-\frac{3k}{2}}\left[2^{-p}\Vert f\Vert_{L^2}+2^{-p}\Vert Sf\Vert_{L^2}+\Vert \Upsilon f\Vert_{L^2}+\Vert \Upsilon Sf\Vert_{L^2}\right]
\end{split}
\end{equation*}
\end{lemma}

\begin{proof}
We may assume that $f=P_{k,q,p}f$. Switching to spherical coordinates, we simply write that, for any $(\rho_0,\phi_0)$ on the support of $\widehat{f}$,
\begin{equation*}
\begin{split}
\widehat{f}(\rho,\phi)&=\widehat{f}(\rho_0,\phi_0)+\int_{\rho_0}^\rho\partial_\rho\widehat{f}(s,\phi_0)ds+\int_{\phi_0}^\phi\int_{\rho_0}^\rho\partial_\rho\partial_\phi\widehat{f}(s,\theta)dsd\theta
\end{split}
\end{equation*}
On the one hand, for any choice of $(\rho_0,\phi_0)$, we obtain that
\begin{equation*}
\begin{split}
\left\vert \int_{\phi_0}^\phi\int_{\rho_0}^\rho\partial_\rho\partial_\phi\widehat{f}(s,\theta)dsd\theta\right\vert&\lesssim 2^{-\frac{k}{2}}\left(\iint_{\substack{\rho\in [2^k,2^{k+1}],\\\sin\phi\in[2^{p-2},2^{p+2}]}}\vert\partial_\rho\partial_\phi\widehat{f}\vert^2\rho^2\sin\phi d\rho d\phi\right)^\frac{1}{2}
\end{split}
\end{equation*}
and we can now average over $\rho_0$ and $\phi_0$ to obtain similarly
\begin{equation*}
\begin{split}
\left\vert \int_{\rho_0}^\rho\partial_\rho\widehat{f}(s,\phi_0)ds\right\vert&\lesssim 2^{-p}\left\vert \iint_{s,\phi}\partial_\rho\widehat{f}(s,\phi)dsd\phi\right\vert\lesssim 2^{-k/2-p}\Vert \partial_\rho\widehat{f}\Vert_{L^2},\\
\vert \widehat{f}(\rho_0,\phi_0)\vert&\lesssim 2^{-(k+p)}\left\vert \iint_{s,\phi}\vert \widehat{f}(s,\phi)\vert dsd\phi\right\vert\lesssim 2^{-3k/2-p}\Vert \widehat{f}\Vert_{L^2}.
\end{split}
\end{equation*}
\end{proof}

\bibliographystyle{siam}
\bibliography{refs.bib}

\end{document}